\documentclass[11pt,a4paper]{amsart}
\usepackage{amscd}
\usepackage{comment}
\usepackage{oldgerm}
\usepackage{amsmath}
\usepackage{amssymb}
\usepackage{amsthm}
\usepackage[svgnames]{xcolor}
\usepackage{bm}
\usepackage[colorlinks,linktocpage=true,citecolor=blue,linkcolor=blue,urlcolor=blue]{hyperref}
\usepackage{graphicx}

\usepackage[all]{xy}
\usepackage[foot]{amsaddr}

\makeatletter
\@addtoreset{equation}{section}
\makeatother

\theoremstyle{definition}
\newtheorem{defn}[equation]{Definition}
\newtheorem{defn'}[equation]{Definition'}
\theoremstyle{plain}
\newtheorem{thm}[equation]{Theorem}

\newtheorem{prop}[equation]{Proposition}
\newtheorem{fact}[equation]{Fact}
\newtheorem{cor}[equation]{Corollary}
\newtheorem{lem}[equation]{Lemma}

\newtheorem{conj}[equation]{Conjecture}
\theoremstyle{remark}
\newtheorem{rem}[equation]{Remark}
\newtheorem{ex}[equation]{Example}

\newcommand{\Z}{\mathbb{Z}}

\newcommand{\R}{\mathbb{R}}
\newcommand{\C}{\mathbb{C}}
\newcommand{\Q}{\mathbb{Q}}

\newcommand{\pt}{\mathrm{pt}}

\newcommand{\del}{\partial}

\newcommand{\dR}{\mathrm{dR}}
\newcommand{\Bord}{\mathrm{Bord}}
\newcommand{\cw}{\mathrm{cw}}
\newcommand{\hBord}{h\Bord^{G_\nabla}}
\newcommand{\Ori}{\mathrm{Ori}}
\newcommand{\HS}{\mathrm{HS}}
\newcommand{\hBordph}{h\Bord^{G_{\mathrm{ph}}}}
\newcommand{\Thom}{\mathrm{Thom}}
\newcommand{\id}{\mathrm{id}}
\newcommand{\hBordone}{h\Bord^{(G_1)_\nabla}}
\newcommand{\hBordthree}{h\Bord^{(G_3)_\nabla}}

\newcommand{\pullbackcorner}[1][dr]{\save*!/#1-1.2pc/#1:(-1,1)@^{|-}\restore}

\usepackage[shortlabels]{enumitem}

\newcommand{\Sp}{\mathsf{Sp}}

\newcommand{\fr}{\textup{fr}}
\newcommand{\st}{\textup{st}}

\DeclareMathOperator{\Ho}{Ho}
\DeclareMathOperator{\Hom}{Hom}

\def\beq#1\eeq{\begin{align}#1\end{align}}

\begin{document}

\title[Differential models for $I\Omega^G$, I]{Differential models for the Anderson dual to bordism theories and invertible QFT's, I}

\author[M. Yamashita]{Mayuko Yamashita}
\address{Department of Mathematics, Kyoto University, 
Kita-shirakawa Oiwake-cho, Sakyo-ku, Kyoto, 606-8502, Japan
}
\email{yamashita.mayuko.2n@kyoto-u.ac.jp}
\author[K. Yonekura]{Kazuya Yonekura}
\address{Department of Physics, Tohoku University, Sendai 980-8578, Japan}
\email{yonekura@tohoku.ac.jp}
\subjclass[]{}
\maketitle

\begin{abstract}
    In this paper, we construct new models for the Anderson duals $(I\Omega^G)^*$ to the stable tangential $G$-bordism theories and their differential extensions. 
    The cohomology theory $(I\Omega^G)^*$ is conjectured by Freed and Hopkins \cite{Freed:2016rqq} to classify deformation classes of possibly non-topological invertible quantum field theories (QFT's). 
    Our model is made by abstracting certain properties of invertible QFT's, thus supporting their conjecture. 
\end{abstract}

\tableofcontents

\section{Introduction}
In this paper, we construct new models for the Anderson duals $(I\Omega^G)^*$ to the stable tangential $G$-bordism theories and their differential extensions. 
Freed and Hopkins \cite{Freed:2016rqq} conjectured that the generalized cohomology theory $(I\Omega^G)^*$ classifies deformation classes of possibly non-topological invertible 
quantum field theories (QFT's) on stable tangential $G$-manifolds. 
Our model is motivated by this conjecture, since it is made by abstracting certain properties of invertible QFT's.

Associated to a generalized cohomology theory $E^*$, its {\it Anderson dual} (\cite[Appendix B]{HopkinsSinger2005}, \cite[Appendix B]{FMS07}) is a generalized cohomology theory which we denote by $IE^*$. 
The crucial property of this theory is that it fits into the following exact sequence for any spectrum $X$.  
\begin{align}\label{eq_exact_DE_abst}
    \cdots \to \mathrm{Hom}(E_{n-1}(X), \R) &\to \Hom(E_{n-1}(X), \R/\Z) \to IE^n(X) \\ & \to \mathrm{Hom}(E_{n}(X), \R) \to \Hom(E_n(X), \R/\Z) \to \cdots \ (\mbox{exact}). \notag
\end{align}
In this paper we are interested in the Anderson dual to {\it stable tangential $G$-bordism theories} $\Omega^G$. 
Here
$G = \{G_d, s_d, \rho_d\}_{d \in \Z_{\ge 0}}$ is a sequence of compact Lie groups equipped with homomorphisms $s_d \colon G_d \to G_{d+1}$ and $\rho_d \colon G_d \to \mathrm{O}(d, \R)$ for each $d$ which are compatible with the inclusion $\mathrm{O}(d, \R) \hookrightarrow \mathrm{O}(d+1, \R)$. 
The homology theory corresponding to $\Omega^G$ is given by the stable tangential $G$-bordism groups $\Omega^G_*(X)$, and the exact sequence \eqref{eq_exact_DE_abst} becomes
\begin{align}\label{eq_exact_DOmega_abst}
    \cdots \to \mathrm{Hom}(\Omega^G_{n-1}(X), \R) &\to \Hom(\Omega^G_{n-1}(X), \R/\Z) \to(I\Omega^G)^n(X) \\ & \to \mathrm{Hom}(\Omega^G_{n}(X), \R) \to \Hom(\Omega^G_n(X), \R/\Z) \to \cdots \ (\mbox{exact}). \notag
\end{align}

The starting point of this work is the following conjecture by Freed and Hopkins. 
\begin{conj}[{\cite[Conjecture 8.37]{Freed:2016rqq}}] \label{conj_intro}
There is a $1 : 1$ correspondence\footnote{Here {\it symmetry types} of QFT's in \cite{Freed:2016rqq} are certain classes of $G$'s in this paper which satisfy an additional set of conditions.     }
\begin{align}\label{eq_conj_intro}
    \left\{
    \parbox{18em}{deformation classes of reflection positive invertible $n$-dimensional fully extended field theories with symmetry type $G$}
    \right\} \simeq (I\Omega^G)^{n+1}(\pt). 
\end{align}
\end{conj}

There are many difficulties in Conjecture \ref{conj_intro}, and here we point out two of them. 
First, we do not have the axioms for non-topological fully extended QFT's.
Thus the left hand side of \eqref{eq_conj_intro} is not a mathematically well-defined object.\footnote{
On the other hand, there is the axiom system by Kontsevich and Segal \cite{Kontsevich:2021dmb} 
for non-extended QFT's which are physically reasonable. 
It would be interesting to prove the modified version of Conjecture \ref{conj_intro} by using it. 
}
Second, although the cohomology theory $(I\Omega^G)^*$ is mathematically defined, its definition is abstract. 
So the right hand side of \eqref{eq_conj_intro} is difficult to treat directly, in particular from the physical point of view. 
Actually, those difficulties are overcome if we are interested in {\it topological} QFT's, and
Freed and Hopkins prove the version of Conjecture \ref{conj_intro} for {\it topological} QFT's, where the right hand side of \eqref{eq_conj_intro} is replaced by its torsion part \cite[Theorem 1.1]{Freed:2016rqq}. 

This work is intended to overcome the second difficulty mentioned above, and to give a new approach to Conjecture \ref{conj_intro}. 
We construct a physically motivated model for the theory $(I\Omega^G)^*$, which is made by abstractizing certain properties of invertible QFT's. 
This result can be regarded as supporting Conjecture \ref{conj_intro}. 
On the other hand, our results also turn out to be mathematically interesting, in view of their relation to {\it differential cohomology theories}. 

In the rest of the introduction, we first explain the main results of this paper in Subsection \ref{subsec_intro_mainthm}, and then explain its physical and mathematical significance in Subsections \ref{subsec_intro_phys} and \ref{subsec_intro_math}, respectively. 

\subsection{A sketch of the main results}\label{subsec_intro_mainthm}
The first main result of this paper is the construction of models for the generalized cohomology theory $(I\Omega^G)^*$ and its differential extension. 
The model of $(I\Omega^G)^*$ is denoted by $(I\Omega_{\mathrm{dR}}^G)^*$, and the differential extension is denoted by $(\widehat{I\Omega_{\mathrm{dR}}^G})^*$. 
Both are defined on the category $\mathrm{MfdPair}$ of pairs of manifolds. 
The precise definition is given in Definition \ref{def_hat_DOmegaG}. 
The physical meaning of this construction is explained in Subsection \ref{subsec_intro_phys} below. 
For simplicity, in this subsection we concentrate on the absolute case, and we only consider the case where $G$ is oriented, i.e., the image of $\rho_d \colon G_d \to \mathrm{O}(d, \R)$ lies in $\mathrm{\mathrm{SO}}(d, \R)$ for each $d$. 

Let $X$ be a manifold and let $n$ be a nonnegative integer. 
The differential group $(\widehat{I\Omega_{\mathrm{dR}}^G})^n(X)$ consists of pairs $(\omega, h)$, where
\begin{itemize}
    \item $\omega \in \Omega^n_{\mathrm{clo}}(X; \varprojlim_d(\mathrm{Sym}^{\bullet/2}\mathfrak{g}_d^*)^{G_d})$, i.e., $\omega$ is a closed differential form on $X$ with values in invariant polynomials on $\mathfrak{g} := \varinjlim_d \mathfrak{g}_d$, of total degree $n$,
    where $\mathfrak{g}_d$ is the Lie algebra of $G_d$. 
    \item $h$ is a map which assigns an $\R/\Z$-value to a triple $(M, g, f)$, where $M$ is a closed $(n-1)$-dimensional manifold with a stable tangential $G$-structure with connection, which we call a {\it differential stable tangential $G$-structure} and symbolically denoted by $g$, and a smooth map $f \colon M \to X$. 
    \item $\omega$ and $h$ should satisfy the following compatibility condition. 
    Suppose we have two triples $(M_-, g_-, f_-)$ and $(M_+, g_+, f_+)$ as above, and a bordism $(W, g_W, f_W)$ between them, by a compact $n$-dimensional manifold with differential stable tangential $G$-structure and a map to $X$. 
    The data of $g_W$ allows us to define a top form on $W$,
    \begin{align}\label{eq_chernweil_intro}
        \mathrm{cw}_{g_W}(f_W^*\omega) \in \Omega^n(W),
    \end{align}
    by applying the Chern-Weil construction with respect to $g_W$ to the coefficient of $f_W^* \omega$. 
    We require that, 
    \begin{align}\label{eq_compatibility_intro}
        h([M_+, g_+, f_+]) - h([M_-, g_-, f_-]) = \int_W \mathrm{cw}_{g_W}(f_W^*\omega) \pmod \Z, 
    \end{align}
\end{itemize}
To define $({I\Omega_{\mathrm{dR}}^G})^n(X)$, we introduce the equivalence relation $\sim $ on $(\widehat{I\Omega_{\mathrm{dR}}^G})^n(X)$. 
We set $(\omega, h) \sim (\omega', h')$ 
if there exists $\alpha \in \Omega^{n-1}(X; \varprojlim_d(\mathrm{Sym}^{\bullet/2}\mathfrak{g}_d^*)^{G_d})$ such that 
\begin{align}\label{eq_er_intro}
    \omega' &= \omega + d\alpha \nonumber, \\
h'\left( [M,g,f] \right) &= h\left( [M,g,f] \right) + \int_M  \mathrm{cw}_{g}(f^*\alpha).
\end{align}
We define
\begin{align}\label{eq_equivalence_intro}
    (I\Omega^G_{\mathrm{dR}})^n(X) := (\widehat{I\Omega^G_{\mathrm{dR}}})^n(X)/\sim. 
\end{align}

For the functor $(\widehat{I\Omega^G_\dR})^*$, we construct the structure homomorphisms $R$, $I$ and $a$ (Definition \ref{def_str_map_IOmega_dR}) as well as the $S^1$-integration map $\int$ (Definition \ref{def_integration_dR}). 
The main result of this paper concerning the differential model is the following. 
\begin{thm}[{Theorem \ref{thm_IOmega_dR=IOmega}}]\label{thm_intro_main}
$(I\Omega^G_{\mathrm{dR}})^*$ gives a model for the generalized cohomology theory $(I\Omega^G)^*$, restricted to the category of manifolds. 
Moreover, $(\widehat{I\Omega^G_{\mathrm{dR}}}, R, I, a, \int)$ is a differential extension with $S^1$-integration (Definition \ref{def_diff_integration}) of $\left( (I\Omega^G)^*, \mathrm{ch}' \right)$. 
\end{thm}
Here the homomorphism $\mathrm{ch}'$ is defined in \eqref{eq_ch_N_G} and coincides with the Chern-Dold homomorphism for examples of $G$ we are usually interested in. 

For example, in the case $G = \mathrm{SO}$, if we have a hermitian line bundle with unitary connection $(L, \nabla)$ over $X$, the pair of first Chern form $c_1(\nabla) = \frac{\sqrt{-1}}{2\pi} F_\nabla \in \Omega_{\mathrm{clo}}^2(X)$ and the holonomy functional with respect to $\nabla$ gives an element $(c_1(\nabla), \mathrm{Hol}_{\nabla}) \in (\widehat{I\Omega^{\mathrm{\mathrm{SO}}}_{\mathrm{dR}}})^{2}(X)$. See Example \ref{ex_hol_target} for details. 
The compatibility condition \eqref{eq_compatibility_intro} follows from the relation of curvature and holonomy. 
For more examples, see Subsection \ref{subsec_examples}. 

We remark that in this paper we mainly focus on {\it tangential} $G$-bordism theories, as opposed to {\it normal} $G$-bordism theories which we denote by $\Omega^{G^\perp}$. 
But a straightforward modification of the tangential case gives a model for the Anderson dual to normal $G$-bordism theories $(I\Omega^{G^\perp})^*$, as explained in Subsection \ref{subsec_normal}. 

With these physically-motivated models at hand, we can sometimes understand known operations in physics regarding QFT's as natural transformations between differential cohomology theories, thus giving a mathematical understanding. 
As we now explain, the results in Section \ref{sec_module} and Section \ref{sec_push} are such examples. 
The subsequent paper \cite{Yamashita2021} by one of the authors of this paper is devoted to another type of such transformations, which is related to differential refinements of multiplicative genera. 

Assume we have a homomorphism $\mu \colon G_1 \times G_2 \to G_3$ of tangential structure groups. 
Typical examples arise from multiplicative tangential structure groups $G$ such as $\mathrm{SO}$ and $\mathrm{Spin}$, where we set $G = G_1 = G_2 = G_3$. 
On the topological level, the homomorphism $\mu$ induces the following. 
\begin{itemize}
    \item The natural transformation
    \begin{align}\label{eq_module_first_intro}
        (I\Omega^{G_3})^n(-) \otimes (\Omega^{G_2})^{-r}(-) \to (I\Omega^{G_1})^{n-r}(-), 
    \end{align}
    where $(\Omega^{G_2})^{-r}(X)$ is the stable tangential $G_2$-bordism {\it co}homology theory group. 
    This comes from the map of spectra \eqref{eq_module_IMTG} and can be thought of as a kind of cup product. 
    \item The {pushforward maps} for tangentially stably $G_2$-oriented proper maps $(p \colon N \to X, g_p^{\mathrm{top}})$, 
    \begin{align}\label{eq_push_intro}
        (p, g_p^{\mathrm{top}})_* \colon (I\Omega^{G_3})^{n}(N) \to (I\Omega^{G_1})^{n-r}(X).  
    \end{align}
    We remark that our terminology ``pushforward'', explained in Subsection \ref{subsec_push_general}, is a certain generalization of the most usual notion of pushforwards, which is associated to multiplicative genera. 
\end{itemize}
In Section \ref{sec_module} and Section \ref{sec_push}, we construct differential refinements of each of the above maps, respectively. 
Here, for the differential refinement $(\widehat{\Omega^{G_2}})$, we use a tangential variant of the cycle-based model constructed by Bunke, Schick, Schr\"{o}der and Wiethaup \cite{BunkeSchickSchroderMU}, which turns out to be very much suited to our model of Anderson duals. 
The model $(\widehat{\Omega^{G_2}})$ is defined in terms of {\it differential relative stable tangential $G_2$-cycles}, and the differential refinements of \eqref{eq_module_first_intro} and \eqref{eq_push_intro} are given in terms of fiber products between differential relative stable tangential $G_2$-cycles and differential stable tangential $G_1$-cycles. 
As we explain in Subsubsection \ref{subsubsec_compactification}, in the physical interpretation, these homomorphisms correspond to {\it compactifications} of invertible QFT's.

\subsection{Physical significance}\label{subsec_intro_phys}

Our results are motivated by the problem of classification of invertible field theories, and this
problem is also related to the classification of action functionals of background fields up to continuous deformations.
We mention only most important points in this subsection, and discuss more details in Section~\ref{sec:physics}. 
A reader who is interested in physics may directly go to Section~\ref{sec:physics} after this subsection.

For an $(n-1)$-dimensional closed manifold $M$, let $g_\mathrm{ph}$ denote a {\it physical tangential 
$G$-structure} on $M$ (Definition~\ref{def_phys_Gstr}),
such as a Riemannian metric, a $G$ fiber bundle and its connection. 
Also, for a given manifold $X$, let $f : M \to X$ denote a smooth map.
Physically, $X$ may be the target space of a background sigma model or a parameter space of 
coupling constants which possibly depend on spacetime positions.  

An $(n-1)$-dimensional (non-anomalous) QFT that can be coupled to the above geometric data
assigns a complex number to such a triple $(M,g_\mathrm{ph}, f)$
and this assignment is called the {\it partition function} of the theory which we denote by $Z$,
\beq
Z : (M,g_\mathrm{ph}, f) \mapsto Z(M,g_\mathrm{ph}, f) \in \mathbb{C}.
\eeq
There is a special class of QFT's called {\it invertible field theories} or {\it classical actions}.
Typical examples of our interest are given by exponentials of Chern-Simons terms and Wess-Zumino-Witten terms,
but terms which affect absolute values of $Z$ such as the cosmological constant term and the Einstein-Hilbert term are also possible.
This class of theories are important for describing anomalies of $(n-2)$-dimensional theories.
It is also important for action functionals of $(n-1)$-dimensional dynamical gauge, gravitational, and sigma model fields. 
In this class of theories, the partition function takes values in $\mathbb{C}^*= \mathbb{C}\setminus \{0\}$.
By continuous deformations which eliminate terms contributing to the absolute value of $Z$, 
the partition function may be deformed to satisfy $|Z|=1$.
In this case we may write
\beq\label{eq:Z-h}
Z(M,g_\mathrm{ph}, f) = \exp\left( 2\pi \sqrt{-1} h([M,g_\mathrm{ph}, f]) \right),
\eeq
where $h$ takes values in $\mathbb{R}/\mathbb{Z}$.

We will see at the end of Subsection~\ref{subsec_model} 
that the difference between a differential stable tangential $G$-structure and a physical tangential $G$-structure 
is not at all significant for our final results. Thus we consider $g$ instead of $g_\mathrm{ph}$ in the rest of this subsection. 

Empirically in physics, it is known that $h$ satisfies the property \eqref{eq_compatibility_intro}.
For instance, this is indeed the case for Chern-Simons terms and Wess-Zumino-Witten terms.
Therefore, the group $(\widehat{I\Omega_{\mathrm{dR}}^G})^n(X)$ introduced in Subsection~\ref{subsec_intro_mainthm} describes
the phase part of possible partition functions via the relation \eqref{eq:Z-h}.
Moreover, two partition functions related by \eqref{eq_er_intro} are continuously deformed to each other.
Therefore, the group $({I\Omega_{\mathrm{dR}}^G})^n(X)$ defined in \eqref{eq_equivalence_intro} gives the classification of 
invertible field theories or classical actions up to continuous deformations, under the physical hypothesis that \eqref{eq_compatibility_intro}
is the physically correct condition. One of our main results implies that it coincides with $({I\Omega_{}^G})^n(X)$
which is defined in algebraic topology.

\subsection{Mathematical significance}\label{subsec_intro_math}
In this subsection we explain our results from a mathematical point of view, especially their relation with the {\it differential cohomology theories}. 
Given a generalized cohomology theory $E^*$, a {\it differential extension} $\widehat{E}^*$, defined on manifolds,  refines $E^*$ with additional differential-geometric data.
$\widehat{E}^*$ itself is also called a {\it generalized differential cohomology theory}. 
For example, $H^2(X; \Z)$ classifies line bundles on $X$, whereas $\widehat{H}^2(X; \Z)$ classifies hermitian line bundles with connections on $X$. 
See \ref{subsec_diffcoh} for necessary backgrounds. 

One interesting point of the construction of $(\widehat{I\Omega_{\mathrm{dR}}^G})^*$ lies in its similarity with the {\it differential character group} of Cheeger-Simons \cite{CheegerSimonsDiffChar}, which is a model for differential ordinary cohomology theory. 
For the ordinary cohomology $H\Z^*$, there are several known models for its differential extension, and the relevant one for us is the {\it differential character group} $\widehat{H}_{\mathrm{CS}}^*(-; \Z)$ of Cheeger-Simons \cite{CheegerSimonsDiffChar} (Example \ref{ex_diff_character}). 
For a manifold $X$, 
$\widehat{H}_{\mathrm{CS}}^n(X; \Z)$ consists of pairs $(\omega, k)$, where
\begin{itemize}
    \item A closed form $\omega \in \Omega_{\mathrm{clo}}^n(X)$, 
    \item A group homomorphism $k \colon Z_{\infty, n-1}(X; \Z) \to \R/\Z$, 
    \item $\omega$ and $k$ satisfy the following compatibility condition. 
    For any $c \in C_{\infty, n}(X; \Z)$ we have
    \begin{align*}
        k(\del c) \equiv \langle c, \omega \rangle_X \pmod \Z. 
    \end{align*}
\end{itemize}
Here $Z_{\infty, *}$ and $C_{\infty, *}$ are the groups of smooth singular cycles and chains, respectively. 
We immediately see that the definition of the group $(\widehat{I\Omega_{\mathrm{dR}}^G})^n(X)$ explained in Subsection \ref{subsec_intro_mainthm} is analogous to that of $\widehat{H}_{\mathrm{CS}}^n(X; \Z)$. 
The essential difference is the domains of $h$ and $k$, which, in the physical interpretation explained in Subsection \ref{subsec_intro_phys}, play the roles of the partition functions of invertible QFT's. 
The authors feel that it is interesting that we found an analogy of differential characters out of classifications of invertible QFT's. 

This is actually not just the analogy, but related to the {\it Anderson self-duality} of the ordinary cohomology. 
By the universal coefficient theorem we have $H\Z^* \simeq IH\Z^*$, with the {\it duality element} $\gamma_H \in IH\Z^0(\pt)$, equivalently, the natural transformation $\gamma_H \colon H\Z^* \to (I\Omega^{\fr})^* ( = I\Z^*)$. 
The above analogy allows us to construct the differential refinement $\widehat{\gamma}_\mathrm{dR}^n \colon\widehat{H}_{\mathrm{CS}}^*(-; \Z) \to (\widehat{I\Z_\dR})^*(-) := (\widehat{I\Omega^\fr_{\mathrm{dR}}})^*(-)$ of the above transformation in Subsection \ref{subsec_self_duality}. 

Finally we comment on the subsequent paper \cite{Yamashita2021} by one of the authors of this paper. 
As explained in Subsection \ref{subsec_intro_phys}, there are many examples of elements in $(\widehat{I\Omega^G_\dR})^*(X)$ which have physical interpretations. 
Then, a natural mathematical question arises: {\it what are these elements mathematically?}
It is natural to expect a topological characterization of these elements. 
Questions of this kind also appears in \cite[Conjecture 9.70]{Freed:2016rqq}. 
Actually, it turns out that many interesting examples arise from {\it pushforwards} (also called {\it integration}) in differential cohomology theories. 
For example, recall the element $(c_1(\nabla), \mathrm{Hol}_{\nabla}) \in (\widehat{I\Omega^{\mathrm{\mathrm{SO}}}_{\mathrm{dR}}})^{2}(X)$ mentioned in Subsection \ref{subsec_intro_mainthm}. 
The holonomy function is the pushforward in second differential ordinary cohomology. 
Another example is $(( \mathrm{Todd})|_{2k}, \overline{\eta}) \in \left(\widehat{I\Omega^{\mathrm{Spin}^c}_\dR}\right)^{2k}(\pt) $, where $\mathrm{Todd} \in (\mathrm{Sym}(\mathfrak{spin^c})^*)^{\mathrm{Spin}^c}$ is the Todd polynomial and the reduced eta invariant $\overline{\eta}$ is the pushforward in differential $K$-theory for odd-dimensional closed differential Spin$^c$-manifolds by Freed and Lott \cite{FL2010}. 
In general, pushforwards in generalized differential cohomology theory are certain refinements of multiplicative genera, homomorphisms of ring spectra of the form $\mathcal{G} \colon MTG \to E$. 
In \cite{Yamashita2021} we will see that in terms of such a differential pushforward we can construct a natural transformation of differential cohomology theories of the form $\widehat{E}^*(-) \otimes IE^n(\pt) \to (\widehat{I\Omega^G_\dR})^{*+n}(-)$, which refines the transformation on the topological level coming from the Anderson dual to $\mathcal{G}$ and the $E$-module structure on $IE$. 
In this way, we get a unified understanding of an important class of elements in $(\widehat{I\Omega^G_\dR})^*(X)$ in terms of natural transformations between differential cohomology theories, as mentioned in Subsection \ref{subsec_intro_mainthm}.

\subsection{The structure of the paper}
The paper is organized as follows. 
We provide necessary preliminaries in Section \ref{sec_preliminary}. 
In Section \ref{sec_Gstr} introduce the differential cycles and related constructions which we use throughout this paper. 
Section \ref{sec_model} is the main section of this paper.  
We define the differential models $I\Omega^G_\dR$ and $\widehat{I\Omega_\dR}$ in Subsection \ref{subsec_model}, and prove the first main result, Theorem \ref{thm_intro_main}, in Subsection \ref{subsec_proof_isom}. 
Using these models, we give a refinement of the transformation \eqref{eq_module_first_intro} in Section \ref{sec_module} and of the pushforward map \eqref{eq_push_intro} in Section \ref{sec_push}. 

\begin{rem}
Here we remark on the first version of this article in arXiv (arXiv:2106.09270v1 , ``Differential models for the Anderson dual to bordism theories and invertible QFT's''). 
The content of this article has been greatly changed and improved from the first version. 
The first version was devoted to the constructions of the models and its verifications, and those results are contained in Sections \ref{sec_Gstr} and \ref{sec_model} of the current paper. 
The constructions are basically unchanged, but the proof has been completely changed and simplified. 
The content of Sections \ref{sec_module} and \ref{sec_push} are new. 
The second part of this article \cite{Yamashita2021} by one of the authors is also new. 

\end{rem}

\subsection{Notations and Conventions}

\begin{itemize}
\item By a topological space, we always mean a compactly generated topological space. 
\item By {\it manifolds}, we mean a smooth manifold with corners (Subsection \ref{subsec_mfd_corner}). 
A {\it pair of manifolds} $(X, Y)$ is a manifold $X$ and its submanifold $Y$, which is a closed subset of $X$.  
We set
\begin{align*}
    C^\infty((X, Y), (X', Y')) := \{f \colon X \to X' \ | \ f \mbox{ is smooth and } f(Y) \subset Y'\}. 
\end{align*}
The category of manifolds is denoted by $\mathrm{Mfd}$ and the category of pairs of manifolds is denoted by $\mathrm{MfdPair}$. 
\item The category of $\Z$-graded abelian groups is denoted by $\mathrm{Ab}^\Z$. 
\item 
The space of $\R$-valued differential forms on a manifold $X$ is denoted by $\Omega^*(X)$. 
For a pair of manifolds $(X, Y)$, 
    we set
    \begin{align*}
        \Omega^n(X, Y) := \{\omega \in \Omega^n(X) \ | \ \omega|_Y = 0 \}. 
    \end{align*}
    \item We also deal with differential forms with values in a graded real vector space $V^\bullet$. 
In the notation $\Omega^n(-; V^\bullet)$, $n$ means the {\it total} degree. 
In the case if $V^\bullet$ is infinite-dimensional, we topologize it as the colimit of all its finite-dimensional subspaces with the canonical topology, and set $\Omega^n(X; V^\bullet) := C^\infty(X; (\wedge T^*X \otimes_\R V^\bullet)^n)$. 
This means that, any element in $\Omega^n(X; V^\bullet)$ can locally be written as a finite sum $\sum_{i} \xi_i \otimes \phi_i$
with $\xi_i \in \Omega^{m_i}(X)$ and $\phi_i \in V^{n-m_i}$ for some $m_i$ for each $i$. 
The space of closed forms are denoted by $\Omega_{\mathrm{clo}}^n(-; V^\bullet)$. 
    \item For a manifold $X$ and a real vector space $V$, we denote by $\underline{V}$ the trivial bundle $\underline{V} := X \times V$ over $X$. 
\item For a topological space $X$, we denote by $p_X \colon X \to \pt$ the map to $\pt$. 
We set $X^+ := (X \sqcup \{*\}, \{*\})$. 
\item For two topological spaces $X$ and $Y$, we denote by $\mathrm{pr}_X \colon X \times Y \to X$ the projection to $X$. 
\item We set $I = [0, 1]$. 
\item For a real vector bundle $V$ over a topological space, we denote its orientation line bundle (rank-$1$ real vector bundle) by $\Ori(V)$. 
For a manifold $M$, we set $\Ori(M) := \Ori(TM)$. 
\item For a {\it spectrum} $\{E_n\}_{n \in \Z}$, we require the adjoints $E_n \to \Omega E_{n+1}$ of the structure homomorphisms are homeomorphisms. 
For a sequence of pointed spaces $\{E'_n\}_{n \in \Z_{\ge a}}$ with maps $\Sigma E'_n \to E'_{n+1}$, we define its {\it spectrification} $LE' :=\{(LE')_n\}_{n \in \Z}$ to be the spectrum given by
\begin{align*}
    (LE')_n := \varinjlim_k \Omega^k E'_{n+k}. 
\end{align*}
    
\end{itemize}

\section{Preliminaries}\label{sec_preliminary}

\subsection{The Anderson duals}\label{subsec_anderson}

In this subsection we collect basics on the Anderson duals for generalized cohomology theories. 
For more details, see for example \cite[Appendix B]{HopkinsSinger2005} and \cite[Appendix B]{FMS07}. 
In this subsection we entirely work with spectra. 
The corresponding statement for CW-pairs $(X, Y)$ is obtained by considering the suspension spectrum $\Sigma^\infty(X / Y)$. 
Remark that $\pi^\st_*(X, Y) = \pi_*(\Sigma^\infty(X/Y))$. 

First note that the functor $X \mapsto \mathrm{Hom}(\pi_*(X), \R/\Z)$ on the stable homotopy category of spectra $\Ho(\Sp)^{\mathrm{op}}$ satisfies the Eilenberg-Steenrood axioms, so it is represented by an $\Omega$-spectrum denoted by $I(\R/\Z)$.
We also have the functor $X \mapsto \mathrm{Hom}(\pi_*(X), \R) $, and the corresponding $\Omega$-spectrum $I\R$.
By the Hurewicz isomorphism we have $\mathrm{Hom}(\pi_*(X), \R)   = H^*(X; \R)$ and 
$I\R$ is isomorphic to the Eilenberg-MacLane spectrum $H\R$. 
Therefore we just take $I\R = H\R$.
The morphism in $\Ho(\Sp)$ representing the transformation 
$\mathrm{Hom}(\pi_*(-), \R)  \to \mathrm{Hom}(\pi_*(-), \R/\Z)$
is denoted by
\begin{align}\label{eq_def_pi}
    \pi \colon H\R \to I(\R/\Z). 
\end{align}
\begin{defn}[{The Anderson dual to the sphere spectrum}]\label{def_IZ}
The {\it Anderson dual to the sphere spectrum}, $I\Z$, is the homotopy fiber of the map \eqref{eq_def_pi}. 
\end{defn}

Applying the homotopy fiber exact sequence, for each spectrum $X$ we get the following exact sequence. 
\begin{align}\label{eq_exact_IZ}
    \cdots \to H^{n-1}(X; \R) &\xrightarrow{\pi} \Hom(\pi_{n-1}(X), \R/\Z) \to I\Z^n(X) \\ & \to H^n(X; \R) \xrightarrow{\pi} \Hom(\pi_n(X), \R/\Z) \to \cdots \ (\mbox{exact}). \notag
\end{align}

\begin{defn}[{The Anderson dual to a spectrum}]\label{def_Anderson_dual}
Let $E$ be a spectrum. 
The {\it Anderson dual to $E$}, denoted by $IE$, is a spectrum defined as the function spectrum from $E$ to $I\Z$, 
\begin{align*}
    IE := F(E, I\Z). 
\end{align*}
\end{defn}
This implies that we have the following exact sequence. 
\begin{align}\label{eq_exact_DE}
    \cdots \to \mathrm{Hom}(E_{n-1}(X), \R) &\xrightarrow{\pi} \Hom(E_{n-1}(X), \R/\Z) \to IE^n(X) \\ & \to \mathrm{Hom}(E_{n}(X), \R) \xrightarrow{\pi} \Hom(E_n(X), \R/\Z) \to \cdots \ (\mbox{exact}). \notag
\end{align}

Hopkins and Singer gave a model for $I\Z^*(X)$ in terms of functors between Picard groupoids in \cite[Corollary B.17]{HopkinsSinger2005}. 
A symmetric monoidal category is called a {\it Picard groupoid} if all the objects are invertible under the monoidal product and all morphisms are invertible under the composition. 
For example, a homomorphism $\del \colon A \to B$ between abelian groups associates a Picard groupoid $(A \xrightarrow{\del} B)$. 
Namely, the objects are elements of $B$, and the morphism from $b$ to $b'$ is given by an element $a \in A$ such that $b' - b = \del(a)$. 
The monoidal structure is given by the addition. 
Another class of examples comes from spectra $X = \{X_n\}_{n \in \Z}$. 
The fundamental groupoid $\pi_{\le 1}(X_n)$ of its each $n$-th space $X_n$ can be equipped with a structure of a Picard groupoid, uniquely up to equivalence (\cite[Example B.7]{HopkinsSinger2005}). 

\begin{fact}[{\cite[Corollary B.17]{HopkinsSinger2005}}]\label{fact_IZ_model}
Let $X= \{X_n\}_{n \in \Z}$ be a spectrum and $n$ be an integer. 
We have an isomorphism
\begin{align}\label{eq_fact_IZ_model}
    I\Z^n(X) \simeq \pi_0 \mathrm{Fun}_{\mathrm{Pic}} \left(\pi_{\le 1}(X_{1-n}) , (\R \xrightarrow{\mathrm{mod} \ \Z} \R/\Z)
    \right), 
\end{align}
where the right hand side means the group of natural isomorphism class of functors of Picard groupoids\footnote{
In \cite[Corollary B.17]{HopkinsSinger2005} they use $(\Q \to \Q/\Z)$ instead of $(\R \to \R/\Z)$. We can use the latter because the inclusion $\Q \to \R$ induces an equivalence of groupoids $(\Q \to \Q/\Z) \to (\R \to \R/\Z)$. 
Also note that there is an obvious typo in the statement of \cite[Corollary B.17]{HopkinsSinger2005}, where $X_{1-n}$ is written as $X_{n-1}$
}. 
The isomorphism \eqref{eq_fact_IZ_model} fits into the commutative diagram, 
\begin{align}\label{diag_fact_IZ_model}
    \xymatrix{
    \Hom(\pi_{n-1}(X), \R/\Z) \ar[r] \ar@{=}[d] & I\Z^n(X) \ar[r] \ar[d]^-{\simeq} &  H^n(X; \R) \ar@{=}[d]\\
    \Hom(\pi_{n-1}(X), \R/\Z) \ar[r] & \pi_0\mathrm{Fun}_{\mathrm{Pic}} \left(\pi_{\le 1}(X_{1-n}) , (\R \to \R/\Z)
    \right)\ar[r] &  H^n(X; \R)
    }
\end{align}
Here the top row is \eqref{eq_exact_IZ}. The bottom left arrow is the obvious one and the right arrow sends a functor to the induced element $\Hom(\pi_1(X_{1-n}), \R) \simeq H^n(X; \R)$. 
\end{fact}
In the above model, the Picard groupoid $(\R \to \R/\Z)$ arises because we have
\begin{align*}
    \pi_{\le 1}(I\Z_{1}) \simeq (\Z \xrightarrow{0} 0) \simeq (\R \to \R/\Z). 
\end{align*}
The isomorphism \eqref{eq_fact_IZ_model} assigns, to an element $I\Z^n(X) = [ X, \Sigma^n I\Z]$, the induced functor between the fundamental Picard groupoids. 
The commutativity of \eqref{diag_fact_IZ_model} is implicit in the discussion in \cite[Appendix B.3]{HopkinsSinger2005}, 
but this commutativity easily follows by the fact that the homotopy fiber sequence $\Sigma^{-1}I\R/\Z \to I\Z \to H\R$ induces the sequence $(0 \to \R/\Z) \to (\R \to \R/\Z) \simeq (\Z \to 0) \to (\R \to 0)$ on the fundamental Picard groupoids of their first spaces. 
The above model for $I\Z$ is crucial to the proofs of our main results of this paper.

In this paper, we are interested in the case where $E$ is the {\it stable tangential $G$-bordism theory}. 
Let $G = \{G_d, s_d, \rho_d\}_{d \in \Z_{\ge 0}}$ be a sequence of compact Lie groups equipped with homomorphisms $s_d \colon G_d \to G_{d+1}$ and $\rho_d \colon G_d \to \mathrm{O}(d, \R)$ for each $d$ such that the following diagram commutes. 
 \begin{align*}
     \xymatrix{
 G_d \ar[r]^-{\rho_d} \ar[d]^{s_d} & \mathrm{O}(d, \R) \ar[d] \\
 G_{d+1} \ar[r]^-{\rho_{d+1}} & \mathrm{O}(d+1, \R)  \ar@{}[lu]|{\circlearrowright}
 }. 
 \end{align*}
Here we use the inclusion $\mathrm{O}(d, \R) \hookrightarrow \mathrm{O}(d+1, \R) $ defined by 
\begin{align}\label{eq_stabilization}
     A \mapsto \left[
     \begin{array}{c|c}
     1 & 0  \\ \hline
     0 &A 
     \end{array}
     \right]  
 \end{align}
throughout this paper. 
We call such $G$ {\it tangential structure groups}. 
Given $G$, the stable tangential $G$-bordism homology theory assigns the stable tangential bordism group $(\Omega^G)_*(-)$. 
It is represented by the {\it Madsen-Tillmann spectrum} $MTG$, which is a variant of the Thom spectrum $MG$. 
For details see for example \cite[Section 6.6]{Freed19}. 
In this paper we take $MTG$ and $MG$ to be a spectrum (as opposed to a prespectrum) as in \cite[(4.60)]{HopkinsSinger2005}. 
In this case, the exact sequence \eqref{eq_exact_DE} becomes
\begin{align}\label{eq_exact_IOmegaG}
    \cdots \to \mathrm{Hom}(\Omega^G_{n-1}(X), \R) &\xrightarrow{\pi} \Hom(\Omega^G_{n-1}(X), \R/\Z) \to (I\Omega^G)^n(X) \\ & \xrightarrow{\mathrm{ch}'} \mathrm{Hom}(\Omega^G_{n}(X), \R) \xrightarrow{\pi} \Hom(\Omega^G_n(X), \R/\Z) \to \cdots \ (\mbox{exact}). \notag
\end{align}


\begin{ex}\label{ex_Gstr}
Here are some examples of tangential structure groups $G$ and the corresponding stable tangential $G$-structures. 
\begin{enumerate}
     \item The case where $G_d = \{1\}$ for all $d$. 
     Since the stable tangential $G$-structure is a stable tangential framing in this case, we denote this $G$ by $\mathrm{fr}$. 
 \item $\mathrm{SO} := \{\mathrm{SO}(d, \R)\}_d$ with $\rho_d$ and $s_d$ given by the inclusions. 
     A stable tangential $\mathrm{SO}$-structure is an orientation with a Riemannian metric. 
     \item $\mathrm{Spin} := \{\mathrm{Spin}(d)\}_d$ with $\rho_d \colon \mathrm{Spin}(d) \to \mathrm{O}(d, \R)$ given by the double covering of $\mathrm{SO}(d, \R)$ composed with the inclusion, and $s_d$ given by the inclusion. 
     A stable tangential Spin-structure is a $\mathrm{Spin}$-structure in the usual sense. 
     \item Let $H$ be a compact Lie group, and set $\mathrm{SO} \times H :=\{\mathrm{SO}(d, \R) \times H\}_d$ with $\rho_d$ given by the composition of the projection $\mathrm{SO}(d, \R) \times H \to \mathrm{SO}(d, \R) $ and the inclusion, and $s_d$ given by the inclusion. 
    A stable tangential $\mathrm{SO}\times H$-structure is an orientation with a Riemannian metric, together with a choice of principal $H$-bundle. 
    This group $H$ is called the {\it internal symmetry group} in physics. 
\end{enumerate}
\end{ex}

\subsection{Generalized differential cohomology theories}\label{subsec_diffcoh}

In this subsection we give a brief review of generalized differential cohomology theories, based on the axiomatic framework given in \cite{BSDiffKSurvey} (see also \cite{BunkeSchick2010}). 
A {\it differential extension} (also called a {\it smooth extension}) of a generalized cohomology theory $E^*$ is a refinement $\widehat{E}^*$ of the restriction of $E^*$ to the category of smooth manifolds, which containes differential-geometric data. 

Let $E^*$ be a generalized cohomology theory. 
Let $N^\bullet$ be a graded vector space over $\R$ equipped with a transformation of cohomology theories
\begin{align}
    \mathrm{ch} \colon E^* \to H^*(-; N^\bullet). 
\end{align}
The universal choice is $N^\bullet = E^\bullet(\pt) \otimes \R =: V_E^\bullet$ with $\mathrm{ch}$ the Chern-Dold homomorphism (\cite[Chapter II, 7.13]{rudyak1998}) for $E$. 

For a manifold $X$, set $\Omega^*(X; N^\bullet) := C^\infty(X; \wedge T^*M \otimes_\R N^\bullet)$ with the $\Z$-grading by the total degree. 
Let $d \colon \Omega^*(X; N^\bullet) \to \Omega^{*+1}(X; N^\bullet)$ be the de Rham differential. 
We have the natural transformation
\begin{align*}
    \mathrm{Rham} \colon \Omega^*_{\mathrm{clo}}(X; N^\bullet) \to H^*(X; N^\bullet). 
\end{align*}

\begin{defn}[{Differential extensions of a cohomology theory, \cite[Definition 2.1]{BSDiffKSurvey}}]\label{def_diffcoh}
A {\it differential extension} of the pair $(E^*, \mathrm{ch})$ is a quadruple $(\widehat{E}, R, I, a)$, where
\begin{itemize}
    \item $\widehat{E}$ is a contravariant functor $\widehat{E} \colon \mathrm{Mfd}^\mathrm{op} \to \mathrm{Ab}^\Z$. 
    \item $R$, $I$ and $a$ are natural transformations
    \begin{align*}
        R &\colon \widehat{E}^* \to \Omega_{\mathrm{clo}}^*(-; N^\bullet) \\
        I &\colon \widehat{E}^* \to E^*\\
        a &\colon \Omega^{*-1}(-; N^\bullet) / \mathrm{im}(d) \to \widehat{E}^* . 
    \end{align*}
\end{itemize}
We require the following axioms. 
\begin{itemize}
    \item $R \circ a = d$. 
    \item $\mathrm{ch} \circ I = \mathrm{Rham} \circ R$. 
   \item For all manifolds $X$, the sequence
   \begin{align}\label{eq_exact_diffcoh}
       E^{*-1}(X) \xrightarrow{\mathrm{ch}} \Omega^{*-1}(M ; N^\bullet) / \mathrm{im}(d) \xrightarrow{a} \widehat{E}(X) \xrightarrow{I} E^*(X) \to 0
   \end{align}
   is exact. 
\end{itemize}
In the case $N^\bullet = V_E^\bullet$ and $\mathrm{ch}$ is the Chern-Dold homomorphism, we simply call it a {\it differential extension of $E^*$}. 
\end{defn}
Such a quadruple $(\widehat{E}, R, I, a)$ itself is also called a {\it generalized differential cohomology theory}. 
We usually abbreviate the notation and just write a generalized cohomology theory as $\widehat{E}^*$.

\begin{ex}[Differential characters]\label{ex_diff_character}
Here we explain a model for a differential extension of $H\Z$ given by Cheeger and Simons, in terms of {\it differential characters} \cite{CheegerSimonsDiffChar}. 
Actually, our definition of the group $(\widehat{I\Omega^G_{\mathrm{dR}}})^*$ (Definition \ref{def_hat_DOmegaG}) is analogous to it. 
We relate them in Subsection \ref{subsec_self_duality}. 
For later use, we explain the relative version. 
For another formulation see \cite{BTreldiffchar}. 

For a pair of manifolds $(X, Y)$ and a nonnegative integer $n$, the group of {\it differential characters} $\widehat{H}_{\mathrm{CS}}^n(X, Y; \Z)$ is the abelian group consisting of pairs $(\omega, k)$, where
\begin{itemize}
    \item A closed differential form $\omega \in \Omega_{\mathrm{clo}}^n(X, Y)$, 
    \item A group homomorphism\footnote{
    Precisely speaking in \cite{CheegerSimonsDiffChar} they use normalized smooth singular cubic chains, but we can also work in terms of the usual smooth singular chains as in \cite{BTreldiffchar}. 
    } $k \colon Z_{\infty, n-1}(X, Y; \Z) \to \R/\Z$, 
    \item $\omega$ and $k$ satisfy the following compatibility condition. 
    For any $c \in C_{\infty, n}(X; \Z)$ we have
    \begin{align}\label{eq_compatibility_diffchar}
        k(\del c) \equiv \langle \omega, c \rangle_X \pmod \Z. 
    \end{align}
\end{itemize}
We have homomorphisms
\begin{align*}
R_{\mathrm{CS} } &\colon \widehat{H}_{\mathrm{CS}}^n(X, Y; \Z) \to \Omega_{\mathrm{clo}}^n(X, Y), \quad (\omega, k) \mapsto \omega \\
a_{\mathrm{CS}} &\colon \Omega^{n-1}(X, Y)/\mathrm{Im}(d) \to  \widehat{H}_{\mathrm{CS}}^n(X, Y; \Z) , \quad
    \alpha \mapsto (d\alpha, \alpha). 
\end{align*}
and the quotient map gives
\begin{align*}
    I_{\mathrm{CS}} \colon \widehat{H}_{\mathrm{CS}}^n(X, Y; \Z) \to\widehat{H}_{\mathrm{CS}}^n(X, Y; \Z) / \mathrm{Im}(a_{\mathrm{CS}}) \simeq  H^n(X, Y; \Z). 
\end{align*}
The quadruple $(\widehat{H}^*_{\mathrm{CS}}, R_{\mathrm{CS}}, I_{\mathrm{CS}} , a_{\mathrm{CS}} )$ is a differential extension of $H\Z$. 
\end{ex} 

We can also consider the differential refinement of $S^1$-integration maps in $E$.  
We have the $S^1$-integration map for differential forms, 
\begin{align}\label{eq_int_form}
    \int \colon \Omega^{n+1}(S^1 \times X; N^\bullet) \to \Omega^n(X; N^\bullet)
\end{align}
for any manifold $X$, which realizes the $S^1$-integration map in de Rham cohomology. 
The sign is defined so that
\begin{align}\label{eq_int_form_sign}
    \int \mathrm{pr}_{S^1}^*\tau_{S^1} \wedge \mathrm{pr}_{X}^* \omega_X = \omega_X
\end{align}
for $\omega_X \in \Omega^*(X; V)$ and $\tau_{S^1} \in \Omega^{1}_{\mathrm{clo}}(S^1)$ which represents the fundamental class of $S^1$ for the standard orientation. 
\begin{defn}[{Differential extensions with $S^1$-integrations, \cite[Definition 2.12]{BSDiffKSurvey}}]\label{def_diff_integration}
A differential extension {\it with integration} of $E^*$ is a quintuple $(\widehat{E}, R, I, a, \int)$, where $(\widehat{E}, R, I, a)$ is a differential extension of $E^*$ and $\int$ is a natural transformation
\begin{align*}
    \int \colon \widehat{E}^{*+1}(S^1 \times -) \to \widehat{E}^*
\end{align*}
such that
\begin{itemize}
    \item $\int \circ (t \times \mathrm{id})^* = -\int$, where $t \colon S^1 \to S^1$ is given by $t(x) = -x$ for $x \in [-1, 1]/\{-1, 1\} = S^1$. 
    \item $\int \circ (p_{S^1} \times \mathrm{id})^* = 0 $. 
    \item The diagram
    \begin{align}\label{diag_diff_integration}
    \xymatrix{
       \Omega^{*}(S^1 \times X; N^\bullet)  \ar[r]^-{a}\ar[d]^{\int } & \widehat{E}^{*+1}(S^1\times X) \ar[d]^{\int}\ar[r]^-{I} \ar@/^18pt/[rr]^R& E^{*+1}(S^1\times X)\ar[d]^{\int} &  \Omega_{\mathrm{clo}}^{*+1}(S^1\times X; N^\bullet)\ar[d]^{\int } \\
    \Omega^{*-1}(X; N^\bullet) \ar[r]^-{a} & \widehat{E}^{*}(X) \ar[r]^-{I} \ar@/_18pt/[rr]^R& E^{*}(X) &  \Omega_{\mathrm{clo}}^{*}(X; N^\bullet)
        }
    \end{align}
    commutes for all manifolds $X$. 
\end{itemize}
\end{defn}

\subsection{Manifolds with corners}\label{subsec_mfd_corner}

In this paper, we need to deal with manifolds with corners. 
There are some variants in the definition among the literatures. 
In this paper the appropriate notion is {\it $\langle k \rangle$-manifolds} defined in \cite{Janich1968}, which we recall here. 

A {\it manifold with corners} of dimension $n$ is a paracompact topological space $M$ with an equivalence class of local coordinate system $\{V_i, \varphi_i\}_{i \in I}$, where $V_i$ is an open subset of $M$ for each $i$ with $M = \cup_i V_i$, 
\begin{align*}
\varphi_i \colon V_i \to \R^n_{\le 0} := (-\infty, 0]^n
\end{align*}
is a homeomorphism onto an open subset of $\R^n_{\le 0}$ for each $i$, and 
\begin{align*}
    \varphi_i\varphi_j^{-1} \colon \varphi_j(V_i \cap V_j ) \to \varphi_i(V_i \cap V_j ) 
\end{align*}
is a diffeomorphism for all pairs $(i, j)$. 
In this paper, by {\it manifolds} we always mean manifolds with corners. 
For a point $x \in M$, we define the {\it depth} of $x$, denoted by $\mathrm{depth}(x)$, to be the number of zeros in a local chart. 
For each nonnegative integer $k$, we set 
\begin{align}
    S^k(M) := \{x \in M \ | \ \mathrm{depth}(x)  = k\}. 
\end{align}
Each $S^k(M)$ has the structure of $(n-k)$-dimensional manifold without boundary. 
The {\it tangent bundle} $TM \to M$ can be defined as a vector bundle of rank $n$ over $M$. 

A map $f \colon M \to N$ between manifolds with corners is called {\it smooth} if, taking a local coordinate system $\{V_i, \varphi_i\}_{i \in I}$ and $\{U_j, \psi_j\}_{j \in J}$ for $M$ and $N$ respectively, the map $\psi_j \circ f \circ \varphi_i^{-1} \colon\varphi_i^{-1}f^{-1}(U_j) \to \R_{\le 0}^{\dim N} $ is a restriction of a smooth map between open subsets of $\R^{\dim M}$ to $\R^{\dim N}$ for each $(i, j) \in I \times J$. 
A smooth map $f$ induces a vector bundle map $df \colon TM \to TN$. 
$f$ is called an {\it embedding} if $f$ is injective and $df \colon T_x M \to T_{f(x)}N$ is injective for all $x \in M$. 
In such a case, we also regard $M$ as a subspace of $N$ and call $M$ a {\it submanifold} of $N$. 
A smooth map $f \colon M \to N$ is called a {\it submersion} if $df \colon T_x M \to T_{f(x)}N$ and $df \colon T_x S^{\mathrm{depth}(x)}(X) \to T_{f(x)} S^{\mathrm{depth}(f(x))}(X)$ are surjective for all $x \in N$.

A manifold with corners is called a {\it manifold with faces} if each $x \in M$ belongs to the closure of $\mathrm{depth}(x)$-different components of $S^1(M)$\footnote{
For example, this excludes the case of the ``teardrop'', the $2$-dimensional disk with a corner. }. 
For a manifold with faces $M$ of dimension $n$, the closure of a connected component of $S^1(M)$ is called a {\it connected face} of $M$, which has the induced structure of an $(n-1)$-dimensional manifold with faces. 
Any union of pairwise disjoint connected faces is called a {\it face} of $M$. 

\begin{defn}[{$\langle k \rangle$-manifolds, \cite[Definition 1]{Janich1968}}]\label{def_k_mfds}
Let $k$ be a nonnegative integer. 
A {\it $\langle k \rangle$-manifold} is a manifold with faces $M$ together with an $k$-tuple $(\del_0 M, \del_{1} M, \cdots, \del_{k-1} M)$ of faces of $M$, satisfying
\begin{enumerate}
    \item $\del_0 M \cup \cdots \cup \del_{k-1} M = \del M$. 
    \item $\del_i M \cap \del_j M $ is a face of $\del_i M$ and  of $\del_j M$ for $i \neq j$. 
\end{enumerate}
\end{defn}
In particular, a $\langle 0 \rangle$-manifold is equivalent to a manifold without boundary, and a $\langle 1 \rangle$-manifold is equivalent to a manifold with boundary. 
See Figure \ref{fig_2_manifold} for $k = 2$. 
Note that each $\del_i M$ can be empty. 
For $I = \{0 \le i_1 < \cdots < i_m \le k - 1 \}$, the intersection $\del_I M := \del_{i_1} M \cap \cdots \cap \del_{i_m} M$ is called an $\langle k-m \rangle$-face of $M$. It has the induced structure of an $\langle k-m \rangle$-manifold so that the order of the labels of the faces are preserved. 

In this paper we are particularly interested in $\langle 0 \rangle$, $\langle 1 \rangle$ and $\langle 2 \rangle$-manifolds. 

\begin{figure}[htbp]
\centering
\includegraphics[width=8cm]{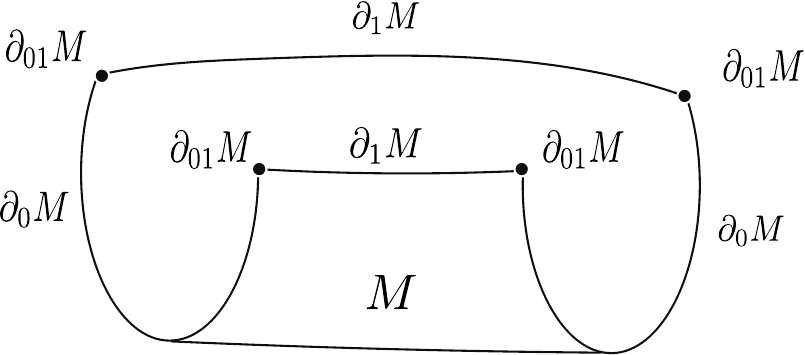}
\caption{A $\langle 2 \rangle$-manifold}
\label{fig_2_manifold}
\end{figure}

\section{Differential stable tangential $G$-structures}\label{sec_Gstr}

In this section we introduce the differential stable tangential $G$-cycles and related constructions which we use throughout this paper. 
Let $G= \{G_d, s_d, \rho_d\}_{d \in \Z_{\ge 0}}$ be tangential structure groups. 

\begin{defn}[{Differential stable $G$-structures on vector bundles}]\label{def_diff_Gstr_vec}
Let $V$ be a real vector bundle of rank $n$ over a manifold $M$. 
\begin{enumerate}
    \item 
A {\it representative of differential stable $G$-structure} on $V$ is a quadruple $\widetilde{g} = (d, P, \nabla, \psi)$, where $d \ge n$ is an integer, 
$(P, \nabla)$ is a principal $G_d$-bundle with connection over $M$ and
$\psi \colon P \times_{\rho_d} \R^d \simeq  \underline{\R}^{d-n} \oplus V $ is an isomorphism of vector bundles over $M$. 
\item We define the {\it stabilization} of such $\widetilde{g}$ by $\widetilde{g}(1) := (d+1, P(1) := P \times_{s_{d}}G_{d+1}, \nabla(1), \psi(1))$, where $\nabla(1)$ and $\psi(1)$ are naturally induced on $P(1)$ from $\nabla$ and $\psi$, respectively. 
\item A {\it differential stable $G$-structure} $g$ on $V$ is a class of representatives $\widetilde{g}$ under the relation $\sim_{\mathrm{stab}}$ defined by $\widetilde{g} \sim_{\mathrm{stab}} \widetilde{g}(1)$. 
\item Suppose we have two representatives of the forms $\widetilde{g} = (d, P, \nabla, \psi)$ and $\widetilde{g} = (d, P, \nabla, \psi')$, such that $\psi$ and $\psi'$ are homotopic. 
In this case, the resulting differential stable $G$-structures $g$ and $g'$ are called {\it homotopic}. 
\end{enumerate}
\end{defn}

Note that we do not require any condition on the connection $\nabla$, e.g., the compatibility with the Levi-Civita connections. 
Such a restriction is not necessary for our purpose, because what we need is the Chern-Weil construction in Definition \ref{def_chern_weil}, which works for any connection. 
This allows us to deal with the examples like $G=\fr = \{1\}$, where connections compatible to Levi-Civita connections rarely exist. 

If we forget the information of the connection $\nabla$, we get the corresponding notion of {\it (topological) differential stable $G$-structures}. 
For a differential stable $G$-structure $g$, we denote the underlying topological structure by $g^{\mathrm{top}}$. 
Similar remarks apply to the various definitions below. 

\begin{defn}[{Differential stable tangential $G$-structures}]\label{def_diff_Gstr_mfd}
Let $M$ be a manifold. 
A {\it differential stable tangential $G$-structure} is a differential stable $G$-structure on the tangent bundle $TM$. 
\end{defn}

\begin{defn}[{Opposite differential stable tangential $G$-structure}]
Let $M$ be an $n$-dimensional manifold. 
Given a differential stable tangential $G$-structure $g$ represented by $\widetilde{g} = (d, P, \nabla, \psi)$ on $M$ with $d > n$, 
we define its {\it opposite differential stable tangential $G$-structure} $g_{\mathrm{op}}$ to be the one represented by $\widetilde{g}_{\mathrm{op}} := \left(d, P, \nabla, (\mathrm{id}_{\underline{\R}^{d-n-1}} \oplus -\mathrm{id}_{\underline{\R}} \oplus\mathrm{id}_{TM}) \circ \psi\right)$. 
\end{defn}

\begin{defn}[Differential stable tangential $G$-cycles]\label{def_geom_smooth_Gcyc}
Let $(X, Y)$ be a pair of manifolds. 
A {\it differential stable tangential $G$-cycle} of dimension $n$ over $(X, Y)$ is a triple $(M, g, f)$, where
$M$ is an $n$-dimensional compact $\langle 1 \rangle$-manifold, 
$g$ is a differential stable tangential $G$-structure on $M$ and
$f \in C^\infty((M, \del M), (X, Y))$. 
We define isomorphisms between two differential stable tangential $G$-cycles in an obvious way. 

In the case $(X, Y) = \pt$, we use the notation $(M, g) := (M, g, p_M)$. 
\end{defn}

\begin{defn}[{$\mathcal{C}^{G_\nabla}_{n}(X, Y)$}]\label{def_mathcalC}
Let $(X, Y)$ be a pair of manifolds. 
We introduce the equivalence relation $\sim$ on differential stable tangential $G$-cycles of dimension $n$ over $(X, Y)$, generated by
\begin{itemize}
    \item Isomorphisms. 
    \item $(M, g, f)\sqcup (M, g_{\mathrm{op}},f) \sim \varnothing$, where $\sqcup$ is the disjoint union defined in the obvious way.
    \item $(M, g, f) \sim (M, g', f)$ if $g$ and $g'$ are homotopic (Definition \ref{def_diff_Gstr_vec} (4)). 
\end{itemize}
The set of equivalence classes under $\sim$
is denoted by $\mathcal{C}^{G_\nabla}_{n}(X, Y)$. 
The class of $(M, g, f)$ is denoted by $[M, g, f]$. 
We introduce an abelian group structure on $\mathcal{C}^{G_\nabla}_{n}(X, Y)$ by disjoint union. 
\end{defn}

Next we proceed to define the bordism Picard groupoid of differential stable tangential $G$-cycles. 
Let $J \subset \R$ be an interval. 
For an $M$ and a differential stable tangential $G$-structure $g$ on it represented by $\widetilde{g} = (d, P, \nabla, \psi)$ with $d > \dim M $, we define $g_J$ to be the differential stable tangential $G$-structure on $J \times M$ represented by
\begin{align}\label{eq_Gstr_cylinder}
    \widetilde{g}_J := (d, \mathrm{pr}_{M}^*P, \mathrm{pr}_{M}^*\nabla, \mathrm{pr}_{ M}^*\psi). 
\end{align}

\begin{defn}[{Bordism between differential stable tangential $G$-cycles}]\label{def_bordism_Gcyc}
Let $(M_-, g_-, f_-)$ and $(M_+, g_+, f_+)$ be two differential stable tangential $G$-cycles of dimension $n$ over $(X, Y)$. 
A {\it bordism} from $(M_-, g_-, f_-)$ to $(M_+, g_+, f_+)$ consists of the following set of data. 
\begin{itemize}
        \item An $(n+1)$-dimensional $\langle 2 \rangle$-manifold $W$ and a differential stable tangential $G$-structure $g_W$ on $W$. 
        \item A disjoint decomposition $\del_0 W = \del_{0, -} W \sqcup \del_{0, +} W$ and open neighborhoods $U_\pm$ of $\del_{0, \pm} W$, respectively, such that $U_+ \cap U_- = \varnothing$. 
        \item Isomorphisms $\varphi_- \colon (U_-, g_W|_{U_-}) \simeq \left([0, \epsilon) \times M_-, (g_-)_{[0, \epsilon)}\right)$ and $\varphi_+ \colon (U_+, g_W|_{U_+}) \simeq \left((-\epsilon, 0] \times M_+, (g_+)_{(-\epsilon, 0]}\right) $ of $\langle 2 \rangle$-manifolds with differential stable tangential $G$-structures for some $\epsilon > 0$. 
        \item A map $f_W \in C^\infty\left((W, \del_1 W), (X, Y)\right)$ such that $f_W|_{U_\pm} = f_\pm \circ \mathrm{pr}_{M_\pm} \circ \varphi_\pm$, respectively. 
\end{itemize}
Given such a set of data, for any $0 < \epsilon' < \epsilon$ we get another set of data by restriction. 
We regard them as defining the same bordism. 
We denote a bordism data typically as $(W, g_W, f_W)$, with the understanding that the collar structures $U_\pm$, $\varphi_\pm$ are also included.  
\end{defn}

Given $(M_-, g_-, f_-)$ and $(M_+, g_+, f_+)$, we introduce the bordism relation between two bordisms $(W_-, g_{W_-}, f_{W_-})$ and $(W_+, g_{W_+}, f_{W_+})$ between them in a routine way. 
Namely, they are called {\it bordant} if there exists the following set of data. 
\begin{itemize}
    \item an $(n+1)$-dimensional $\langle 3 \rangle$-manifold $N$ equipped with a differential stable tangential $G$-structure $g_N$ and $f_N \in C^\infty((N, \del_2 W), (X, Y))$. 
    \item a decomposition $\del_0 N = \del_{0, +}N \sqcup \del_{0, -}N$ and $\del_1 N = \del_{1, +}N \sqcup \del_{1, -}N$
with isomorphisms $\del_{0, \pm} N \simeq W_\pm$ and $\del_{1, \pm} N \simeq [0, 1] \times M_\pm$ of $\langle 2 \rangle$-manifolds.  
\item Collar structures near $\del_{0, \pm} N$ and $\del_{1, \pm} N$ which restricts to the collar structures on $W_\pm$ in Definition \ref{def_bordism_Gcyc}. 
\item Corresponding isomorphisms of differential stable tangential $G$-structures on the collars, extending those on $W_\pm$ in Definition \ref{def_bordism_Gcyc}. 
\end{itemize}
The bordism class of a bordism $(W, g_W, f_W) \colon (M_-, g_-, f_-) \to (M_+, g_+, f_+)$ is denoted by $[W, g_W, f_W]$. 

\begin{defn}[{$h\Bord^{G_\nabla}_n(X, Y)$}]\label{def_Bord_Pic_cat}
Let $n$ be a positive integer and $(X, Y)$ be a pair of manifolds. 
We define the symmetric monoidal category $h\Bord^{G_\nabla}_n(X, Y)$ by the following. 
\begin{itemize}
    \item The objects are differential stable tangential $G$-cycles $(M, g, f)$ of dimension $n$ over $(X, Y)$. 
    \item The morphisms from $(M_-, g_-, f_-)$ to $(M_+, g_+, f_+)$ are  the bordism classes\footnote{
Note that the morphisms are {\it bordism classes}, rather than diffeomorphism classes, of bordisms. 
This is why we obtain Picard groupoids (Corollary \ref{cor_hBord_Pic}). } $[W, g_W, f_W]$ of bordisms between them. 
    \item The identity morphism on $(M, g, f)$ is the cylinder $(I \times M, g_I, f \circ \mathrm{pr}_M)$ with the obvious collar structure. 
    \item The monoidal product is the disjoint union, with the unit $\varnothing$. 
\end{itemize}
\end{defn}

\begin{rem}
The notation is due to the fact that the above category $h\Bord^{G_\nabla}_n(X, Y)$ is the homotopy $1$-category of an $(\infty, 0)$-category version of it. 
Although we do not introduce the higher categories, we use this notation. 
\end{rem}

Recall the Madsen-Tillmann spectrum $MTG$ which represents the tangential $G$-bordism theory which appeared in Subection \ref{subsec_anderson}. 
The symmetric monoidal category $h\Bord^{G_\nabla}_n(X, Y)$ is equivalent to the fundamental groupoid of the $(-n)$-th space of the Madsen-Tillmann spectrum via the Pontryagin-Thom construction, as follows. 

\begin{lem}\label{lem_cat_equivalence}
Let $(X, Y)$ be a pair of manifolds. 
There is a equivalence of symmetric monoidal categories
\begin{align}\label{eq_lem_cat_equivalence}
    h\Bord^{G_\nabla}_n(X, Y) \simeq \pi_{\le 1}(L((X/Y)\wedge MTG)_{-n}), 
\end{align}
which is natural in $(X, Y)$. 
\end{lem}
\begin{proof}
The theorem of Pontryagin-Thom implies there is an equivalence of symmetric monoidal categories which is natural in $(X, Y)$, 
\begin{align}
    \pi_{\le 1}(L((X/Y)\wedge MTG)_{-n}) \simeq h\Bord^{G}_{n}(X, Y). 
\end{align}
where $h\Bord^{G}_{n}(X, Y)$ is obtained from $h\Bord^{G_\nabla}_{n}(X, Y)$ by forgeting the connections\footnote{
Remark that we obtain {\it stable} $G$-manifolds, as opposed to $G_n$-manifolds, by the Pontryagin Thom construction. This is because the spectrum $MTG$ classifies the stable tangential $G$-bordism theory, and we are taking the Omega-spectrification so that $\pi_{i}((L(X/Y) \wedge MTG)_{-n}) \simeq \Omega^{G}_{n+i}(X, Y)$ for $i = 0, 1$. 
}. 
In turn, the forgetful functor $h\Bord^{G_\nabla}_{n}(X, Y) \to h\Bord^{G}_{n}(X, Y)$ is an equivalence because the space of connections on a given principal bundle is contractible, so we get the result. 
\end{proof}

By \cite[Example B.7]{HopkinsSinger2005}, the right hand side of \eqref{eq_lem_cat_equivalence} is a Picard groupoid. 
Thus we get the following. 

\begin{cor}\label{cor_hBord_Pic}
The category $h\Bord^{G_\nabla}_n(X, Y)$ is a Picard groupoid. 
\end{cor}

Explicitly, the inverse under the monoidal product of an object $(M, g, f)$ is $(M, g_{\mathrm{op}}, f)$. 
The inverse under the composition of a morphism $[W, g_W, f_W] \colon (M_-, g_-, f_-) \to (M_+, g_+, f_+)$ is essentially given by ``reversing $g_W$''. 
If $G$ is such that ``reversing'' $g_{W} \mapsto \overline{g}_{W}$ makes sense so that $(W, \overline{g}_W, f_W)$ is a bordism from $(M_+, g_+, f_+)$ to $(M_-, g_-, f_-)$, this represents the required inverse.
Such $G$ includes $\mathrm{SO}$ and $\mathrm{Spin}$. 
However, for general $G$, for example $G = 1$, such a reversal does not make sense, so the explicit description of the inverse is complicated (given by a suitable deformation of $g_W$ on the collar). 
In any case, Corollary \ref{cor_hBord_Pic} implies that we can take the inverse $[W, g_W, f_W]^{-1}$ on the level of bordism classes.

\section{Physically motivated models for the Anderson duals of $G$-bordisms}\label{sec_model}
Let $G= \{G_d, s_d, \rho_d\}_{d \in \Z_{\ge 0}}$ be tangential structure groups. 
In this section, we give models $(I\Omega^G_{\mathrm{dR}})^*$ and $(\widehat{I\Omega^G_{\mathrm{dR}}})^*$ of the Anderson dual to $G$-bordism theory and a differential extension of it, respectively. 

Let us define a $G_d$-module $\R_{G_d}$ with the underlying vector space $\R$, and the $G_d$-module structure given by the multiplication via $G_d \xrightarrow{\rho_d} \mathrm{O}(d, \R) \xrightarrow{ \mathrm{det}} \{\pm 1\}$.

\begin{lem}\label{lem_IOmega_R}
We have 
\begin{align*}
    H^*(MTG; \R) = \varprojlim_{d} H^*(BG_d; EG_d \times_{G_d} \R_{G_d}). 
\end{align*}
\end{lem}
\begin{proof}
Recall that the Madsen-Tillmann spectrum $MTG$ is defined as the direct limit of the Thom spaces of the normal bundles of the universal bundles over (approximations of) $BG_d$ \cite[Section 6.6]{Freed19}. 
We see that their orientation bundle is the pullback of the bundle $EG_d \times_{G_d} \R_{G_d}$ over $BG_d$. 
The result follows by the Thom isomorphism. 
\end{proof}

\begin{lem}\label{lem_sym_poly}
Fix a nonnegative integer $d$.  
For each integer $n$, we have
\begin{align*}
    H^{2n}(BG_d; EG_d \times_{G_d} \R_{G_d}) &\simeq (\mathrm{Sym}^n\mathfrak{g}_d^* \otimes_\R \R_{G_d})^{G_d}, \\
    H^{2n+1}(BG_d; EG_d \times_{G_d} \R_{G_d}) &=0. 
\end{align*}
Here $\mathfrak{g}_d$ is the Lie algebra of $G_d$ and $\mathfrak{g}_d^*$ is its dual. 
The notation $(-)^{G_d}$ means the $G_d$-invariant part, where $G_d$ acts on $\mathfrak{g}_d$ by the adjoint. 
\end{lem}
\begin{proof}
Let $K_d$ be the kernel of the homomorphism $ \mathrm{det} \circ \rho_d \colon G_d \to \{\pm 1\}$. 
In the case where $G_d = K_d$, the $G_d$-module $\R_{G_d}$ is trivial $\R$, and the desired results follow from the Chern-Weil isomorphism for $G_d$. 
In the case where $G_d \neq K_d$, apply the Leray-Serre spectral sequence for the fibration
\begin{align*}
   BK_d \to BG_d \to B\{\pm 1\}. 
\end{align*}
We get the isomorphism $H^*(BG_d; EG_d \times_{G_d} \R_{G_d} ) \simeq (H^*(BK_d; \R) \otimes_\R \R_{G_d})^{\{\pm 1\}} = (H^*(BK_d; \R) \otimes_\R \R_{G_d})^{G_d}$. 
Using the Chern-Weil isomorphism for $K_d$, we get the result. 
\end{proof}

We denote 
\begin{align}\label{eq_inv_poly}
    N_G^\bullet := H^\bullet(MTG; \R) = \varprojlim_{d} H^\bullet(BG_d; \R_{G_d}) = \varprojlim_{d}(\mathrm{Sym}^{\bullet/2}\mathfrak{g}_d^* \otimes_\R \R_{G_d})^{G_d}. 
\end{align}
In the case where $G$ is oriented, i.e., the image of $\rho_d$ lies in $\mathrm{SO}(d, \R)$ for each $d$, the $G_d$-module $\R_{G_d}$ is trivial and $N_G^\bullet$ is the projective limit of invariant polynomials on $\mathfrak{g}_d$. 
In general cases, $N_G^\bullet$ can be regarded as the projective limit of polynomials on $\mathfrak{g}_d$ which change the sign by the action of $G_d$.

For any CW-complex $X$, by the K\"unneth formula and the Hurewicz theorem we have
\begin{align}
    H^*(X; N_G^\bullet) \simeq \Hom (\Omega^G_*(X), \R), 
\end{align}
so the fourth arrow in \eqref{eq_exact_IOmegaG} gives the transformation
\begin{align}\label{eq_ch_N_G}
    \mathrm{ch}' \colon (I\Omega^G)^* \to H^*(-; N_G^\bullet). 
\end{align}
In general, $N_G^\bullet$ is not isomorphic to $V_{I\Omega^G}^\bullet = (I\Omega^G)^\bullet(\pt)\otimes \R$. 
The relation is the following. 
Let $E$ be any spectrum. 
Applying the exact sequence \eqref{eq_exact_DE} to $X = \pt$, we get the following short exact sequence, 
\begin{align}
    0 \to \mathrm{Ext}(E_{n-1}(\pt), \Z) \to IE^n(\pt) \to \Hom(E_n(\pt), \Z) \to 0. 
\end{align}
Since $\R$ is a flat $\Z$-module, we get the short exact sequence
\begin{align}
    0 \to \mathrm{Ext}(E_{n-1}(\pt), \Z)\otimes \R \to V_{IE}^n \to \Hom(E_n(\pt), \Z) \otimes \R \to 0. 
\end{align}
The group $\mathrm{Ext}(E_{n-1}(\pt), \Z)\otimes \R$ does not vanish in general, but it vanishes for example if $E_{n-1}(\pt)$ is finitely generated. 
Furthermore, we have a canonical map
\begin{align}
    \Hom(E_n(\pt), \Z) \otimes \R \to \Hom(E_n(\pt), \R),
\end{align}
and it is an isomorphism if $E_n(\pt)$ is finitely generated. 
Thus we see that
\begin{prop}\label{prop_N_G_canonical}
For any tangential structure groups $G$, we have a canonical homomorphism
\begin{align}\label{eq_exact_V_N}
    q \colon V_{I\Omega^G}^\bullet \to N_G^\bullet. 
\end{align}
It is an isomorphism if $\Omega^G_n(\pt)$ is finitely generated for all $n$. 
\end{prop}
The finiteness condition in Proposition \ref{prop_N_G_canonical} is satisfied in examples we are usually interested in. 

\subsection{The models}\label{subsec_model}

\subsubsection{The differential models}
In this subsubsection, 
we define the models. 
The definition uses a variant of the Chern-Weil construction, as we now explain.  

\begin{defn}\label{def_chern_weil}
\begin{enumerate}
    \item 
Let $W$ be a manifold and $(P, \nabla)$ be a principal $G_d$-bundle with connection. 
Let $n$ be an even nonnegative integer and $\phi \in N_G^n$. 
Apply the forgetful map\footnote{
Here we denoted by $G_d = \{(G_d)_{d'}, (s_d)_{d'}, (\rho_d)_{d'}\}_{d' \in \Z_{\ge 0}}$ the tangential structure groups defined by truncating $G$ at degree $d$, i.e., $(G_{d})_{d'} := G_d'$ for $d' \le d$ and $(G_{d})_{d'} := G_d$ for $d' \ge d$, with $(s_d)_{d'} = \mathrm{id}$ for $d' \ge d$ and the other structure maps are obvious ones. 
}
\begin{align}\label{eq_fgt_coeff}
    N_G^\bullet \to N_{G_d}^\bullet = (\mathrm{Sym}^{\bullet/2}\mathfrak{g}_d^* \otimes_\R \R_{G_d})^{G_d}.
\end{align}
to $\phi$ and denote by $\phi_d$ the resulting element. 
We set
\begin{align}\label{eq_def_cw_1}
    \mathrm{cw}_{\nabla} (\phi) =\mathrm{cw}_{\nabla} (\phi_d) := \phi_d(F_\nabla) \in \Omega_{\mathrm{clo}}^n(W; P \times_{G_d}\R_{G_d}), 
\end{align}
where $F_\nabla \in \Omega_{\mathrm{clo}}^2(W; P \times_{\mathrm{ad}}\mathfrak{g}_d)$ is the curvature form for $(P, \nabla)$, and we used the identification $\R_{G_d}=\R_{G_d}^*$ of $G_d$-modules. 
This construction gives a homomorphism
$\mathrm{cw}_\nabla \colon N_G^n \to \Omega^n_\mathrm{clo}(W; P \times_{G_d}\R_{G_d})$. 
Extending it $\Omega^*(W)$-linearly, we get a homomorphism of $\Z$-graded real vector spaces, 
\begin{align}\label{eq_def_cw_hom_nabla}
    \mathrm{cw}_\nabla \colon \Omega^*(W; N_G^\bullet) \to \Omega^*(W; P \times_{G_d}\R_{G_d}), 
\end{align}
which restricts to a homomorphism $\mathrm{cw}_\nabla \colon \Omega_{\mathrm{clo}}^*(W; N_G^\bullet) \to \Omega_{\mathrm{clo}}^*(W;P \times_{G_d}\R_{G_d})$. 

\item
If $V \to W$ is a real vector bundle with a differential stable $G$-structure $g$ represented by $\widetilde{g} = (d, P, \nabla, \psi)$,  the isomorphism $\psi$ induces the isomorphism $\psi \colon P\times_{G_d} \R_{G_d}\simeq \mathrm{Ori}(V)$. 
Composing the homomorphism \eqref{eq_def_cw_hom_nabla} with this $\psi$ on the coefficients, we define
a homomorphism of $\Z$-graded real vector spaces, 
\begin{align}\label{eq_def_cw_hom}
    \mathrm{cw}_g := \psi \circ \mathrm{cw}_\nabla \colon \Omega^*(W; N_G^\bullet) \to \Omega^*(W; \mathrm{Ori}(V)), 
\end{align}
which restricts to the homomorphism $\mathrm{cw}_g \colon \Omega_{\mathrm{clo}}^*(W; N_G^\bullet) \to \Omega_{\mathrm{clo}}^*(W; \mathrm{Ori}(V))$.
\end{enumerate}
\end{defn}
In particular, if we have a differential stable tangential $G$-structure $g$ on $W$, the homomorphism \eqref{eq_def_cw_hom} becomes
\begin{align*}
    \mathrm{cw}_g \colon \Omega^*(W; N_G^\bullet) \to \Omega^*(W; \mathrm{Ori}(W)).
\end{align*}

\begin{rem}\label{rem_cw_dR_htpy}
In Definition \ref{def_chern_weil} (2), the homomorphism $\mathrm{cw}_g$ actually depends on $g$ only through its {\it homotopy class} (Definition \ref{def_diff_Gstr_vec} (4)). 
This is because we only used $\psi$ to construct the isomorphism $\psi \colon P\times_{G_d} \R_{G_d}\simeq \mathrm{Ori}(V)$, and it does not change when we replace $\psi$ to a homotopic one. 
\end{rem}

\begin{rem}\label{rem_cw_coefficient}
Definition \ref{def_chern_weil} admits the following generalization, which we use in Section \ref{sec_module} and Section \ref{sec_push}, as well as in the next paper \cite{Yamashita2021}. 
Let $\mathcal{V}^*$ be any $\Z$-graded vector space over $\R$. 
We can generalize $N_G^\bullet = H^\bullet(MTG; \R)$ in the above definition to $H^\bullet(MTG; \mathcal{V}^*)$. 
Then the homomorphism \eqref{eq_def_cw_hom} becomes
\begin{align}\label{eq_def_cw_hom_coefficient}
    \mathrm{cw}_g \colon \Omega^n(W;H^\bullet(MTG; \mathcal{V}^*)) \to \Omega^n(W; \mathrm{Ori}(V)\otimes_\R \mathcal{V}^*). 
\end{align}
The construction is basically by just applying the above procedure $\mathcal{V}^*$-linearly. 
But we need some care because we need to allow $\mathcal{V}^*$ to be infinite-dimensional, since $H^\bullet(MTG_d; \mathcal{V}^*) \neq (N_{G_d}^{*'} \otimes_\R \mathcal{V}^*)^\bullet$ in general. 
To fix this point, in the construction corresponding Definition \ref{def_chern_weil} (1), take a closed manifold $\mathcal{B}$ with a $(\dim W + 1)$-connected map
$\mathcal{B} \to BG_d$. 
Then consider the pullback 
\begin{align}\label{eq_rem_cw_coeff}
    H^\bullet(MTG_d; \mathcal{V}^*) \to H^\bullet(\mathcal{B} ; (\mathcal{B} \times_{G_d} \R_{G_d}) \otimes_\R  \mathcal{V}^*) = \oplus_{k = 0}^{\dim \mathcal{B}} H^k(\mathcal{B} ; \mathcal{B} \times_{G_d} \R_{G_d})\otimes_\R \mathcal{V}^{*-k}. 
\end{align}
We have $N_{G_d}^k \simeq H^k(\mathcal{B} ; \mathcal{B} \times_{G_d} \R_{G_d})$ for $0 \le k \le \dim W$. 
Thus we can apply the image under the composition \eqref{eq_rem_cw_coeff} to the curvature $F_\nabla$ to get the Chern-Weil form. 
The result does not depend on the choice of $\mathcal{B}$. 
\end{rem}

Suppose we are given a pair of manifolds $(X, Y)$ and a form $\omega \in \Omega^*(X, Y; N_G^\bullet)$. 
If $W$ is a manifold equipped with a differential stable tangential $G$-structure $g$, 
given a map $f \in C^\infty((W, \varnothing), (X, Y))$, the homomorphism \eqref{eq_def_cw_hom} gives 
\begin{align*}
    \mathrm{cw}_{g}(f^*\omega) \in \Omega^*(W, f^{-1}(Y); \mathrm{Ori}(W)). 
\end{align*}
If $\omega\in \Omega_{\mathrm{clo}}^n(X, Y; N_G^\bullet)$, the resulting form is closed, so
the above construction induces a homomorphism
\begin{align}\label{eq_def_cw_omega_mor}
    \mathrm{cw}(\omega) \colon \Hom_{h\Bord^{G_\nabla}_{n-1}(X, Y)}((M_-, g_-, f_-), (M_+, g_+, f_+)) &\to \R \\
    [W, g_W, f_W] &\mapsto \int_W \mathrm{cw}_{g_W}((f_W)^*\omega),  \notag
\end{align}
for each pair of objects $(M_\pm, g_\pm, f_\pm)$ in $h\Bord^{G_\nabla}_{n-1}(X, Y)$. 
In particular, applying this to $(M_\pm, g_\pm, f_\pm) = \varnothing$, we get
\begin{align}\label{eq_def_cw_omega_closed}
    \mathrm{cw}(\omega) \colon \Omega_n^G(X, Y) \to \R, \quad
    [W, g_W, f_W] \mapsto  \int_W \mathrm{cw}_{g_W}((f_W)^*\omega). 
\end{align}
The above map only depends on the cohomology class of $\omega$ and $\mathrm{Rham}(\omega) \mapsto \cw(\omega)$ gives the isomorphism $H^n(X, Y; N_G^\bullet) \simeq \mathrm{Hom}(\Omega^G_n(X, Y), \R)$.

\begin{defn}[{$(\widehat{I\Omega^G_{\mathrm{dR}}})^*$ and $(I\Omega^G_{\mathrm{dR}})^*$}]\label{def_hat_DOmegaG}
Let $(X, Y)$ be a pair of manifolds and $n$ be a nonnegative integer. 
\begin{enumerate}
    \item 
Define $(\widehat{I\Omega^G_{\mathrm{dR}}})^n(X, Y)$ to be an abelian group consisting of pairs $(\omega, h)$, such that
\begin{enumerate}
    \item $\omega$ is a closed $n$-form $\omega\in \Omega_{\mathrm{clo}}^n(X, Y; N_G^\bullet)$. 
    \item $h$ is a group homomorphism
    $h \colon \mathcal{C}^{G_\nabla}_{n-1}(X, Y) \to \R/\Z$. 
\item $\omega$ and $h$ satisfy the following compatibility condition. 
Assume that we are given two objects $(M_-, g_-, f_-)$ and $(M_+, g_+, f_+)$ in $h\Bord^{G_\nabla}_{n-1}(X, Y)$ and a morphism $[W, g_W, f_W]$ from the former to the latter. 
Then we have
\begin{align}\label{eq_compatibility_dR}
    h([M_+, g_+, f_+]) - h([M_-, g_-, f_-]) = \mathrm{cw}(\omega)([W, g_W, f_W]) \pmod \Z, 
\end{align}
where the right hand side is defined in \eqref{eq_def_cw_omega_mor}. 
\end{enumerate}
Abelian group structure on $(\widehat{I\Omega^G_{\mathrm{dR}}})^n(X, Y)$ is defined in the obvious way. 

\item
We define a homomorphsim of abelian groups, 
\begin{align}\label{eq_def_DOmega_a}
    a \colon \Omega^{n-1}(X, Y; N_G^\bullet)/\mathrm{Im}(d) &\to  (\widehat{I\Omega^G_{\mathrm{dR}}})^n(X, Y) \\
    \alpha &\mapsto (d\alpha, \mathrm{cw}(\alpha)). \notag
\end{align}
Here the homomorphism $\mathrm{cw}(\alpha) \colon \mathcal{C}^{G_\nabla}_{ n-1}(X, Y) \to \R/\Z$ is defined by (see Remark \ref{rem_cw_dR_htpy})
\begin{align}\label{eq_def_cw_alpha}
    \mathrm{cw}(\alpha) ([M, g, f]) := \int_M \mathrm{cw}_g(f^*\alpha) \pmod \Z. 
\end{align}
We set
\begin{align*}
    (I\Omega^G_{\mathrm{dR}})^n(X, Y) := (\widehat{I\Omega^G_{\mathrm{dR}}})^n(X, Y)/ \mathrm{Im}(a). 
\end{align*}

\end{enumerate}
For negative integers $n$, we set $(\widehat{I\Omega^G_{\mathrm{dR}}})^n(X, Y) := 0$ and $(I\Omega^G_{\mathrm{dR}})^n(X, Y) := 0$. 

For a smooth map $\phi \in C^\infty((X, Y), (X', Y'))$ between two pairs of manifolds, by the pullback we get the homomorphisms $\phi^* \colon (\widehat{I\Omega^G_{\mathrm{dR}}})^n(X', Y')\to (\widehat{I\Omega^G_{\mathrm{dR}}})^n(X, Y)$ and $\phi^* \colon (I\Omega^G_{\mathrm{dR}})^n(X', Y')\to (I\Omega^G_{\mathrm{dR}})^n(X, Y)$. 
Thus we get contravariant functors, 
\begin{align*}
    (\widehat{I\Omega^G_{\mathrm{dR}}})^*, \ (I\Omega^G_{\mathrm{dR}})^*
    \colon \mathrm{MfdPair}^{\mathrm{op}} \to \mathrm{Ab}^\Z. 
\end{align*}
\end{defn}

In the case where $G_d = 1$ for all $d$, i.e., the case of the stably framed bordism theory $\Omega^{\mathrm{fr}}$, the corresponding Madsen-Tillmann spectrum is the sphere spectrum and $(I\Omega^{\mathrm{fr}})^* = I\Z^*$. 
So in this case, we also denote $ (\widehat{I\Z}_{\mathrm{dR}})^* := (\widehat{I\Omega^{\mathrm{fr}}_{\mathrm{dR}}})^*$ and $(I\Z_{\mathrm{dR}} )^*:= (I\Omega^{\mathrm{fr}}_{\mathrm{dR}})^*$. 

We sometimes say that ``evaluate $(\omega, h)$ on $[M, g, f]$'' to just mean getting the value $h([M, g, f])$. 



The functor $(\widehat{I\Omega^G_{\mathrm{dR}}})^*$ is equipped with the following structure maps. 
In Theorem \ref{thm_IOmega_dR=IOmega}, we will see that the functors $R$, $I$ and $a$ make $(\widehat{I\Omega^G_{\mathrm{dR}}})^*$ into a differential extension of $\left((I\Omega^G)^*, \mathrm{ch}'\right)$ where $\mathrm{ch}'$ is defined in \eqref{eq_ch_N_G}. 

\begin{defn}[{Structure maps for $(\widehat{I\Omega^G_{\mathrm{dR}}})^*$ and $(I\Omega^G_{\mathrm{dR}})^*$}]\label{def_str_map_IOmega_dR}
We define the following maps natural in $(X, Y)$. 
The well-definedness is easy by Definition \ref{def_hat_DOmegaG}. 
\begin{itemize}
    \item We denote the quotient map by
\begin{align*}
    I \colon (\widehat{I\Omega^G_{\mathrm{dR}}})^*(X, Y) \to  (I\Omega^G_{\mathrm{dR}})^*(X, Y). 
\end{align*}
\item We define
\begin{align*}
    R \colon (\widehat{I\Omega^G_{\mathrm{dR}}})^*(X, Y) \to \Omega_{\mathrm{clo}}^*(X, Y; N_G^\bullet) , \quad
    (\omega, h) \mapsto \omega. 
\end{align*}
\item We define
\begin{align*}
    \mathrm{ch}' \colon (I\Omega^G_{\mathrm{dR}})^*(X, Y) \to H^n(X, Y; N_G^\bullet) \left(\simeq \mathrm{Hom}(\Omega^G_{*}(X, Y), \R)\right), \quad
    I((\omega, h)) \mapsto \mathrm{Rham}(\omega)\left(\mapsto \mathrm{cw}(\omega)\right). 
\end{align*}
\item We define
\begin{align*}
    p \colon \mathrm{Hom}(\Omega^G_{*-1}(X, Y), \R/\Z) \to (I\Omega^G_{\mathrm{dR}})^*(X, Y), \quad 
    h \mapsto I((0, h)),
\end{align*}
where we regard $h \in \mathrm{Hom}(\Omega^G_{n-1}(X, Y), \R/\Z)$ as a group homomorphism $h \colon \mathcal{C}^{G_\nabla}_{ n-1}(X, Y) \to \R/\Z$. 

\end{itemize}

\end{defn}

In the physical interpretation of the group $(\widehat{I\Omega^G_{\mathrm{dR}}})^n(X)$ explained in Subsection \ref{subsec_intro_phys}, an invertible $(n-1)$-dimensional QFT on manifolds equipped with differential stable tangential $G$-structures and maps to $X$ gives an element $(\omega, h) \in (\widehat{I\Omega^G_{\mathrm{dR}}})^n(X)$. 
In this picture, the value $\exp(2\pi \sqrt{-1}h([M, g, f]))$ corresponds to the value of partition function applied to $[M, g, f] \in \mathcal{C}^{G_\nabla}_{n-1}(X)$.

Now we make some easy observations on these models. 
An important consequence of the compatibility condition of $\omega$ and $h$ in Definition \ref{def_hat_DOmegaG} is the following integrality condition.  
\begin{lem}[The integrality condition]\label{lem_integrality}
We have 
\begin{align}\label{eq_integrality}
    \mathrm{Im} (\mathrm{ch}') \subset \mathrm{Hom}(\Omega^G_{*}(X, Y), \Z),
\end{align}
i.e., any element $(\omega, h) \in (\widehat{I\Omega^G_{\mathrm{dR}}})^*(X, Y)$ induces a $\Z$-valued homomorphism $\mathrm{cw}(\omega)\colon \Omega^G_{*}(X, Y)\to \Z $. 
\end{lem}
\begin{rem}
    The inclusion \eqref{eq_integrality} is actually an equality by Proposition \ref{prop_exact_DOmegaG_dR}. 
\end{rem}
    
\begin{proof}
This follows from the condition (c) in Definition \ref{def_hat_DOmegaG} (1) applied to $M = \varnothing$. 
\end{proof}

Now we show that $I\Omega^G_{\mathrm{dR}}$ fits into the same exact sequence as in \eqref{eq_exact_IOmegaG}, which is an important property of the Anderson dual.

\begin{prop}\label{prop_exact_DOmegaG_dR}
\item For any pair of manifolds $(X, Y)$ and integer $n$, the following sequence is exact.  
    \begin{align}\label{eq_prop_exact_dR}
         \mathrm{Hom}(\Omega^G_{n-1}(X, Y), \R)\to \mathrm{Hom}(\Omega^G_{n-1}(X, Y), \R/\Z) \xrightarrow{p} (I\Omega^G_{\mathrm{dR}})^n(X, Y)\\ \xrightarrow{\mathrm{ch}'}
        \mathrm{Hom}(\Omega^G_n(X, Y), \R)
        \to \mathrm{Hom}(\Omega^G_{n}(X, Y), \R/\Z). \notag
    \end{align}
\end{prop}

To prove Proposition \ref{prop_exact_DOmegaG_dR}, we need the following lemma. 

\begin{lem}\label{lem_welldef_int_omega}
Let $(X, Y)$ be a pair of manifolds and $n$ be a nonnegative integer. 
Let $\omega \in \Omega_{\mathrm{clo}}^n(X, Y; N_G^\bullet)$ be a closed form 
such that the associated homomorphism $\cw(\omega)$ in \eqref{eq_def_cw_omega_closed} is $\Z$-valued,  $\mathrm{cw}(\omega)  \colon \Omega^G_{n}(X, Y)\to \Z$. 

Let $(M_-, g_-, f_-)$ and $(M_+, g_+, f_+)$ be two objects in $h\Bord^{G_\nabla}_{n-1}(X, Y)$. 
Given any two morphisms $[W, g_W, f_W], [W', g_{W'}, f_{W'}] \colon (M_-, g_-, f_-) \to (M_+, g_+, f_+)$, we have
\begin{align}\label{eq_lem_welldef_int_omega}
     \cw(\omega)([W, g_W, f_W]) =\cw(\omega)([W', g_{W'}, f_{W'}])  \pmod \Z. 
\end{align}
\end{lem}
\begin{proof}[Proof of Lemma \ref{lem_welldef_int_omega}]
By Corollary \ref{cor_hBord_Pic} (also see the remarks following it), we can take the inverse $[W', g_{W'}, f_{W'}]^{-1}$ of the morphism. 
We have
\begin{align*}
    &\cw(\omega)([W', g_{W'}, f_{W'}]) + \cw(\omega)([W', g_{W'}, f_{W'}]^{-1}) \\
    &= \cw(\omega)\left([W', g_{W'}, f_{W'}] \circ [W', g_{W'}, f_{W'}]^{-1}
    \right) \\
    &= \cw(\omega)\left(\mathrm{id}_{(M_+, g_+, f_+)}\right) \\
    &= 0, 
\end{align*}
by the obvious additivity of $\cw(\omega)$ under composition of morphisms. 
We also have
\begin{align*}
    &\cw(\omega)([W, g_W, f_W]) + \cw(\omega)([W', g_{W'}, f_{W'}]^{-1}) \\
    &= \cw(\omega)\left([W, g_W, f_W] \circ [W', g_{W'}, f_{W'}]^{-1}
    \right) \in \Z, 
\end{align*}
by the integrality assumption of $\omega$. 
Combining these, we get Lemma \ref{lem_welldef_int_omega}. 
\end{proof}

\begin{proof}[Proof of Proposition \ref{prop_exact_DOmegaG_dR}]
The composition at $\mathrm{Hom}(\Omega^G_n(X, Y), \R)$ is zero by Lemma \ref{lem_integrality}. 
The other compositions are obviously zero. 

First we show the exactness at $\mathrm{Hom}(\Omega^G_{n-1}(X, Y), \R/\Z)$. 
Suppose that $h \in \mathrm{Hom}(\Omega^G_{n-1}(X, Y), \R/\Z)$ satisfies $I((0, h)) = 0$. 
Then there exists $\alpha \in \Omega_{\mathrm{clo}}^{n-1}(X, Y; N_G^\bullet)/\mathrm{Im}(d)$ with $h = \mathrm{cw}(\alpha) $. 
Since the homomorphism $\mathrm{cw}(\alpha)$ lifts to an $\R$-valued homomorphism defined by the same formula as \eqref{eq_def_cw_alpha}, we see that $h$ is in the image from $\mathrm{Hom}(\Omega^G_{n-1}(X, Y), \R)$. 

Next we show the exactness at $(I\Omega^G_{\mathrm{dR}})^n(X, Y)$. 
Suppose that $I((\omega, h)) \in (I\Omega^G_{\mathrm{dR}})^n(X, Y)$ satisfies $\mathrm{Rham}(\omega) = 0 $. 
There exists $\alpha \in \Omega^{n-1}(X, Y; N_G^\bullet )$ such that $\omega = d\alpha$. 
Thus we have $I((\omega, h)) = I((0, h-\mathrm{cw}(\alpha)))$. 
This implies $p\left(h-\mathrm{cw}(\alpha)\right) = I((\omega, h))$. 

Finally we show the exactness at $\mathrm{Hom}(\Omega^G_n(X, Y), \R)$. 
It is equivalent to the claim that $\mathrm{ch}' \colon (I\Omega^G_{\mathrm{dR}})^n(X, Y)\to \mathrm{Hom}(\Omega^G_n(X, Y), \Z)$ is surjective. 
Take any element in $\mathrm{Hom}(\Omega^G_n(X, Y), \Z) \subset \mathrm{Hom}(\Omega^G_n(X, Y), \R) \simeq H^n(X, Y;N_G^\bullet )$ and take a representative $\omega \in \Omega^n_{\mathrm{clo}}(X, Y; N_G^\bullet)$. 
We would like to find a group homomorphism $h \colon {\mathcal{C}}^{G_\nabla}_{n-1}(X, Y) \to \R/\Z$ which satisfies the compatibility condition in Definition \ref{def_hat_DOmegaG} (1) (c) with $\omega$. 

The compatibility condition with $\omega$ already determines the value of $h$ on the kernel of the forgetful map ${\mathcal{C}}^{G_\nabla}_{n-1}(X, Y) \to \Omega^G_{n-1}(X, Y)$. 
Namely, given a differential smooth stable tangential $G$-cycle $(M, g, f)$ over $(X, Y)$ of dimension $(n-1)$ which is null-bordant, take any morphism $[W, g_W, f_W] \colon \varnothing \to (M, g, f)$ in $h\Bord^{G_\nabla}_{n-1}(X, Y)$ and set
\begin{align}\label{eq_proof_exact_seq}
    h([M, g, f]) := \cw(\omega)([W, g_W, f_W]) \pmod \Z.
\end{align}
The right hand side does not depend on the choice of $[W, g_W, f_W]$ by Lemma \ref{lem_welldef_int_omega}. 
The fact that the formula \eqref{eq_proof_exact_seq} induces a  
group homomorphism $h$ on the kernel of the forgetful map ${\mathcal{C}}^{G_\nabla}_{n-1}(X, Y) \to \Omega^G_{n-1}(X, Y)$ can be checked easily. 
Since $\R/\Z$ is an injective group, there exists a group homomorphism $h \colon {\mathcal{C}}^{G_\nabla}_{n-1}(X, Y) \to \R/\Z$ extending it, so we get the result.
\end{proof}

In Subsection \ref{subsec_proof_isom} we show that $I\Omega^G_\dR$ is a model for the Andeson dual to the $G$-bordism theory. 
The following result, combined with this result, implies that the quadruple $((\widehat{I\Omega^G_{\mathrm{dR}}})^*, R, I, a)$ is the differential extension of the pair $\left((I\Omega^G)^*, \mathrm{ch}' \right)$ with $\mathrm{ch}'$ in \eqref{eq_ch_N_G}. 

\begin{prop}\label{prop_axiom_diffcoh}
\begin{enumerate}
\item We have $R \circ a = d$. 
\item For any pair of manifolds $(X, Y)$, the following diagram commutes. 
\begin{align*}
    \xymatrix{
    (\widehat{I\Omega^G_{\mathrm{dR}}})^*(X, Y) \ar[r]^R \ar[d]^I & \Omega^*_{\mathrm{clo}}(X, Y; N_G^\bullet) \ar[d]^{\mathrm{Rham}} \\
    (I\Omega^G_{\mathrm{dR}})^{*}(X, Y) \ar[r]^{\mathrm{ch}'} & H^*(X, Y; N_G^\bullet)
    }. 
\end{align*}

    \item For any pair of manifolds $(X, Y)$, the following sequence is exact. 
\begin{align}
    (I\Omega^G_{\mathrm{dR}})^{*-1}(X, Y) \xrightarrow{\mathrm{ch}'} 
    \Omega^{*-1}(X, Y; N_G^\bullet)/\mathrm{Im}(d) \xrightarrow{a} (\widehat{I\Omega^G_{\mathrm{dR}}})^*(X, Y)
    \xrightarrow{I}(I\Omega^G_{\mathrm{dR}})^*(X, Y) \to 0. 
\end{align}
\end{enumerate}

\end{prop}
\begin{proof}
(1) and (2) are obvious. For (3), 
the exactness at $\Omega^{*-1}(X, Y; N_G^\bullet)/\mathrm{Im}(d)$ easily follows from the exactness of \eqref{eq_prop_exact_dR} at $\mathrm{Hom}(\Omega^G_n(X, Y), \R)$. 
The remaining parts are exact by definition. 
\end{proof}

Our differential model $(\widehat{I\Omega^G_{\mathrm{dR}}})^*$ also has an $S^1$-integration. 
As shown in Theorem \ref{thm_IOmega_dR=IOmega}, it makes our model a {\it differential extension with $S^1$-integration} in the sense of \cite[Definition 2.12]{BSDiffKSurvey} (Definition \ref{def_diff_integration}). 
Actually, as we will see in Remark \ref{rem_integration_push}, the $S^1$-integration map is a special case of the differential pushforwards which we introduce in Section \ref{sec_push}. 

In order to define the $S^1$-integration map, we need some preparation. 
Let $M$ be an $n$-dimensional $\langle k \rangle$-manifold and $g$ be a differential stable tangential $G$-structure on $M$ represented by $\widetilde{g}=(d,P,\nabla,\psi)$ with $d \ge n+2$.
For the 2-dimensional disk $D^2 = \{ (x,y) \in \R^2\, |\, x^2+y^2 \le 1 \}$,
let $g_{D^2}$
be the differential stable tangential $G$-structure on $D^2 \times M$ represented by $\widetilde{g}_{D^2} := (d, \mathrm{pr}_{M}^*P, \mathrm{pr}_{M}^*\nabla, \mathrm{pr}_M^*\psi)$, where we identify
$\underline{\R}^{d-n-2} \oplus T(D^2\times  M) \simeq \mathrm{pr}_{M}^*( \underline{\R}^{d-n-2} \oplus \underline{\R}^2 \oplus T M) = \mathrm{pr}_{M}^*( \underline{\R}^{d-n} \oplus T M)$. 
We can take the obvious collar structure near the boundary $\partial D^2 \times M$ which is induced by the polar coordinates $(x,y)=(r\cos\theta, r\sin\theta)$.
Then we get an isomorphism
\begin{align}\label{eq_bounding_str}
    \psi_{S^1} \colon  \mathrm{pr}_{M}^*P \times_{\rho_d}  \underline{\R}^{d} \simeq    \underline{\R}^{d-n-1} \oplus T(S^1 \times  M)
\end{align}
such that $(D^2 \times M, g_{D^2},  f \circ \mathrm{pr}_M)$ is a bordism from $\varnothing$ to $(S^1 \times M, g_{S^1},  f \circ \mathrm{pr}_M)$ for any $f \colon M \to X$,
where $g_{S^1}$ is represented by $\widetilde{g}_{S^1} := (d, \mathrm{pr}_{M}^*P, \mathrm{pr}_{M}^*\nabla, \psi_{S^1})$.

\begin{defn}[The bounding differential stable tangential structure]\label{def_Gstr_S1}
Let $M$ be an $n$-dimensional manifold and $g$ be a differential stable tangential $G$-structure on $M$ represented by $\widetilde{g}=(d,P,\nabla,\psi)$ with $d \ge n+2$. 
The {\it bounding differential stable tangential $G$-structure} $g_{S^1}$ on $S^1 \times M$ is represented by
\begin{align*}
    \widetilde{g}_{S^1} := (d, \mathrm{pr}_{M}^*P, \mathrm{pr}_{M}^*\nabla, \psi_{S^1}),
\end{align*}
where $\psi_{S^1}$ is defined in \eqref{eq_bounding_str}. 
\end{defn}

\begin{defn}[{The $S^1$-integration map for $(\widehat{I\Omega^G_{\mathrm{dR}}})^*$}]\label{def_integration_dR}
Let $n$ be a nonnegative integer. We define the following map natural in $(X, Y)$,
\begin{align*}
    \int \colon (\widehat{I\Omega^G_{\mathrm{dR}}})^{n+1}( S^1 \times (X, Y) ) \to (\widehat{I\Omega^G_{\mathrm{dR}}})^{n}(X, Y),
\end{align*}
by mapping $(\omega, h)$ to $(\int \omega, \int h )$, where (see Remark \ref{rem_welldef_int_dR})
\begin{itemize}
    \item $\int \omega$ is the image of $\omega$ under the $S^1$-integration of differential forms \eqref{eq_int_form}. 
    \item We define the homomorphism $\int h \colon\mathcal{C}^{G_\nabla}_{n-1}(X, Y) \to \R/\Z $ by
    \begin{align}\label{eq_def_S^1_integration}
       \left(\int h\right) ([M, g, f])
       := -h([S^1 \times M, g_{S^1},\mathrm{id}_{S^1} \times f]). 
    \end{align}
    Here $g_{S^1}$ is given by Definition \ref{def_Gstr_S1}. 
\end{itemize}
The natural transformation $\int$ induces a natural transformation on the topological level, also denoted by
\begin{align}\label{eq_def_S1integration_dR}
    \int \colon (I\Omega^G_{\mathrm{dR}})^{n+1}(S^1 \times -) \to (I\Omega^G_{\mathrm{dR}})^{n}(-). 
\end{align}
We call them the {\it $S^1$-integration map} for $(\widehat{I\Omega^G_{\mathrm{dR}}})^*$ and $(I\Omega^G_{\mathrm{dR}})^*$, respectively. 
\end{defn}

\begin{rem}\label{rem_welldef_int_dR}
The minus sign in \eqref{eq_def_S^1_integration} is due to the fact that the assignment $g \mapsto g_{S^1}$ does not preserve the bordism relation. 
Rather, we need an additional automorphism on $ \underline{\R}^{d-n-1} \oplus T(S^1 \times  M)$ for objects which reverses the orientation. 
Also we have
\begin{align}\label{eq_minus_susp_reason}
\cw_{g_W}\left(\int \omega\right)
=\int\cw_{(g_W)_{S^1}}(\omega). 
\end{align}
Using these, the compatibility of the pair $(\int \omega, \int h)$ can be checked easily.

\end{rem}

\subsubsection{The models in terms of equivalent Picard subgroupoids of $\hBord_-(-)$}\label{sec_physical_app}
In the definition of the model $I\Omega^G_\dR$ so far, we have used the Picard groupoid $\hBord_-(-)$. 
However, in some cases the partition function $h$ is naturally defined only for objects of a Picard subgroupoid of it, and it has enough information to define an element in $I\Omega^G_\dR$. 

Let $\mathcal{D} \subset \hBord_{n-1}(X, Y)$ be a Picard subcategory such that the inclusion is an equivalence. 
We define $\mathcal{C}_{\mathcal{D}} \subset \mathcal{C}^{G_\nabla}_{n-1}(X, Y)$ to be the subgroup generated by the isomorphism classes of objects in $\mathcal{D}$. 
Then, consider the following group. 

\begin{defn}[{$(\widehat{I\Omega^G_{\mathrm{dR}, \mathcal{D}}})^n(X, Y)$ and $(I\Omega^G_{\mathrm{dR}, \mathcal{D}})^n(X, Y)$}]\label{def_sub_model}
In the above settings, we define $(\widehat{I\Omega^G_{\mathrm{dR}, \mathcal{D}}})^n(X, Y)$ and $(I\Omega^G_{\mathrm{dR}, \mathcal{D}})^n(X, Y)$ to be the abelian groups defined by replacing $\hBord_{n-1}(X, Y)$ with $\mathcal{D}$ and $\mathcal{C}^{G_\nabla}_{n-1}(X, Y)$ with $\mathcal{C}_{\mathcal{D}}$ in Definition \ref{def_hat_DOmegaG}. 
\end{defn}
Thus, an element in $(\widehat{I\Omega^G_{\mathrm{dR}, \mathcal{D}}})^n(X, Y)$ is a pair $(\omega, h_{\mathcal{D}})$ where $\omega\in \Omega_{\mathrm{clo}}^n(X, Y; N_G^\bullet)$ as before, but the domain of $h_{\mathcal{D}}$ is now smaller, $h_{\mathcal{D}} \colon \mathcal{C}_{\mathcal{D}} \to \R/\Z$. 
We show that the resulting groups are isomorphic. 

\begin{prop}\label{prop_sub_model_isom}
The obvious forgetful maps by the restriction of $h$, 
\begin{align}
    \mathrm{fgt} \colon (\widehat{I\Omega^G_{\mathrm{dR}}})^n(X, Y) &\to (\widehat{I\Omega^G_{\mathrm{dR}, \mathcal{D}}})^n(X, Y) \label{eq_fgt_sub_model_hat}\\
    \mathrm{fgt} \colon({I\Omega^G_{\mathrm{dR}}})^n(X, Y) &\to ({I\Omega^G_{\mathrm{dR}, \mathcal{D}}})^n(X, Y) \label{eq_fgt_sub_model}
\end{align}
are isomorphisms. 
\end{prop}

\begin{proof}
We can construct the inverse of \eqref{eq_fgt_sub_model_hat} easily as follows. 
Given an element $(\omega, h_{\mathcal{D}})\in (\widehat{I\Omega^G_{\mathrm{dR}, \mathcal{D}}})^n(X, Y)$, 
we need to extend $h_{\mathcal{D}}$ to $h \colon \mathcal{C}^{G_\nabla}_{n-1}(X, Y) \to \R/\Z$. 
Since $\mathcal{D} \hookrightarrow \hBord_{n-1}(X, Y)$ is an equivalence, any object in $\hBord_{n-1}(X, Y)$ is bordant to an object in $\mathcal{D}$. 
By the compatibility condition in Definition \ref{def_hat_DOmegaG} (1) (c), we are forced to define the value of $h$ using such a bordism. 
The well-definedness follows from the fact that $\mathcal{D} \hookrightarrow \hBord_{n-1}(X, Y)$ is full, and the compatibility of $(\omega, h_{\mathcal{D}})$. 
It is obvious that this assignment gives the inverse of \eqref{eq_fgt_sub_model_hat}. 
The result for \eqref{eq_fgt_sub_model} also follows from this. 
\end{proof}

A typical class of examples of such situations is the following. 
In the context of unitary QFT's in physics, we are usually interested in $G$ with the following properties. 
First we require that the image of $\rho_d$ contains at least $\mathrm{SO}(d,\R)$,
\begin{align}\label{eq_FHgroup_1}
\mathrm{SO}(d,\R) \subset \rho_d(G_d) .
\end{align}
Next we require that the following commutative diagram is a pullback diagram,
\begin{align}\label{eq_FHgroup_2}
    \xymatrix{
G_d \ar[r]^-{\rho_d} \ar[d]^{s_d} \pullbackcorner & \mathrm{O}(d, \R) \ar[d] \\
G_{d+1} \ar[r]^-{\rho_{d+1}} & \mathrm{O}(d+1, \R) 
}. 
\end{align}
In Example \ref{ex_Gstr}, (2), (3) and (4) satisfy these assumtions, but (1) does not. 
For $G$ satisfying these properties, we define
\begin{defn}\label{def_phys_Gstr}
Let $M$ be an $n$-dimensional manifold. 
A {\it physical tangential $G$-structure} on $M$ is a triple $g_{\mathrm{ph}}=(P,\nabla,\psi)$, where
\begin{itemize}
\item The quadruple $(n, P,\nabla,\psi)$ is a representative of differential stable $G$-structure (Definition \ref{def_diff_Gstr_vec}) for $TM$. 
In particular, there is no stabilization of $TM$. 
\item We have a Riemannian metric on $TM$ induced from the standard metric on $P \times_{\rho_d} \R^{\dim M} $ by 
the isomorphism $\psi : P \times_{\rho_d} \R^{\dim M} \simeq TM$.
The connection induced on $P \times_{\rho_d} \R^{\dim M} \simeq TM$ from $\nabla$ coincides with the Levi-Civita connection of the Riemannian metric.
\end{itemize}
\end{defn}

A physical tangential $G$-structure can be regarded as a differential stable tangential $G$-structure in the obvious way. 
Then, we define $\hBordph_{n-1}(X, Y)$ to be the full subcategory of $\hBord_{n-1}(X, Y)$ spanned by the objects with physical tangential $G$-structures. 
It is a standard fact that the inclusion is an equivalence, by the requirements \eqref{eq_FHgroup_1} and \eqref{eq_FHgroup_2}. 
We remark that any morphism in $\hBordph_{n-1}(X, Y)$ can be represented by a bordism with physical tangential $G$-structure by the same reason. 

We often encounter such situations. 
For example in \cite{Freed:2016rqq}, they require the conditions \eqref{eq_FHgroup_1} and \eqref{eq_FHgroup_2} in the definition of ``symmetry types''. 
Also see Examples \ref{ex_complex_eta} and \ref{ex_real_eta} below. 
Since they are so typical, we use the notations
$(\widehat{I\Omega^G_{\mathrm{ph}}})^*$ and $(I\Omega^G_{\mathrm{ph}})^*$ for the groups in Definition \ref{def_sub_model} in the case $\mathcal{D} = \hBordph_{-}(-)$. 

We also encounter another type of $\mathcal{D}$ in Section \ref{sec_module}. 
There, we use Pirard subgroupoids spanned by objects $(M, g, f)$ such that $f$ satisfies certain transversality conditions.

\subsection{The proof of the isomorphism $(I\Omega^G_\mathrm{dR})^* \simeq (I\Omega^G)^*$}\label{subsec_proof_isom}
In this subsection we prove the main result of this section, Theorem \ref{thm_IOmega_dR=IOmega}. 

First we relate our models with functors from the bordism Picard categories. 
Recall that, as explained in Subsection \ref{subsec_anderson}, a homomorphism $\del \colon A \to B$ between abelian groups associates a Picard groupoid $(A \xrightarrow{\del} B)$. 
Given an element $(\omega, h) \in (\widehat{I\Omega^G_{\mathrm{dR}}})^n(X, Y)$, we get the associated functor of Picard groupoids, 
\begin{align}\label{eq_associated_functor}
        F_{(\omega, h)} \colon h\Bord^{G_\nabla}_{n-1}(X, Y) \to (\R \to \R/\Z)
    \end{align}
by $F_{(\omega, h)}(M, g, f) := h([M, g, f])$ on objects and $F_{(\omega, h)}([W, g_W, f_W]) := \cw(\omega)([W, g_W, f_W])$ on morphisms. 
Moreover, given two elements $(\omega, h)$ and $(\omega', h')$, and an element $\alpha \in \Omega^{n-1}(X, Y; N_G^\bullet)/\mathrm{Im}(d)$ so that $(\omega', h') - (\omega, h) = a(\alpha)$, we get the associated natural transformation, 
\begin{align}\label{eq_associated_transformation}
    F_\alpha \colon F_{(\omega, h)} \Rightarrow F_{(\omega', h')}, 
\end{align}
by $F_\alpha (M, g, f) := \mathrm{cw}(\alpha)(M, g, f)$. 
Summarizing, we get the following. 

\begin{lem}\label{lem_functor}
The assignment \eqref{eq_associated_functor} and \eqref{eq_associated_transformation} gives
a symmetric monoidal functor
\begin{align}
    F_{(X, Y)} \colon \left( \Omega^{n-1}(X, Y; N_G^\bullet)/\mathrm{Im}(d) \xrightarrow{a}  (\widehat{I\Omega^G_{\mathrm{dR}}})^n(X, Y) \right) 
    \to \mathrm{Fun}_{\mathrm{Pic}} \left(h\Bord^{G_\nabla}_{n-1}(X, Y), (\R \to \R/\Z) \right)
\end{align}
which is natural in $(X, Y)$. 
Here $\mathrm{Fun}_{\mathrm{Pic}}$ is regarded as a symmetric monoidal category. 
In particular, passing to the isomorphism classes of objects, we get the following natural transformation of functors $\mathrm{MfdPair}^{\mathrm{op}} \to \mathrm{Ab}$, 
\begin{align}\label{eq_lem_functor}
    F \colon (I\Omega^G_\dR)^n(-) \to \pi_0  \mathrm{Fun}_{\mathrm{Pic}} \left(h\Bord^{G_\nabla}_{n-1}(-), (\R \to \R/\Z) \right). 
\end{align}
\end{lem}

Now we show the main result of this section. 

\begin{thm}\label{thm_IOmega_dR=IOmega}
There is a natural isomorphism of the functors $\mathrm{MfdPair}^{\mathrm{op}} \to \mathrm{Ab}^\Z$, 
\begin{align*}
    F \colon I\Omega^G_\dR \simeq I\Omega^G, 
\end{align*}
which fits into the following commutative diagram. 
\begin{align}\label{diag_mainthm}
    \xymatrix{
 0 \ar[r]& \mathrm{Ext}(\Omega^G_{n-1}(-), \Z) \ar[r]^-p \ar@{=}[d] & (I\Omega^G_{\mathrm{dR}})^n \ar[r]^-{\mathrm{ch}'} \ar[d]_-{\simeq}^-{F} & \mathrm{Hom}(\Omega^G_n(-), \Z) \ar[r]\ar@{=}[d] & 0  \\
 0 \ar[r]& \mathrm{Ext}(\Omega^G_{n-1}(-), \Z) \ar[r] & (I\Omega^G)^n \ar[r] & \mathrm{Hom}(\Omega^G_n(-), \Z) \ar[r] & 0
}
\end{align}
Here we use $\mathrm{Ext}(-, \Z) \simeq \Hom(-, \R/\Z)/\Hom(-, \R)$. 
Moreover, the quintuple $(\widehat{I\Omega^G_\dR}, R, I, a, \int)$ in Definitions \ref{def_str_map_IOmega_dR} and \ref{def_integration_dR} is a differential extension of $\left((I\Omega^G)^*, \mathrm{ch}'\right)$ with integration, where $\mathrm{ch}'$ is defined in \eqref{eq_ch_N_G}. 
In particular, if $\Omega^G_n(\pt)$ is finitely generated for all $n$, it gives a differential extension with integration, with respect to the Chern-Dold homomorphism $\mathrm{ch} \colon (I\Omega^G)^* \to H^*(-; V_{I\Omega^G}^\bullet)$. 
\end{thm}

\begin{proof}
With Lemma \ref{lem_functor} in hand, the proof is essentially the same as a part of the proof of \cite[Proposition 5.24]{HopkinsSinger2005}. 
By Fact \ref{fact_IZ_model}, there is an isomorphism
\begin{align}\label{eq_proof_mainthm_1}
    (I\Omega^G)^n(X, Y) \simeq \pi_0\mathrm{Fun}_{\mathrm{Pic}} (\pi_{\le 1}(L((X/Y)\wedge MTG)_{1-n}), (\R \to \R/\Z))
\end{align}
natural in $(X, Y)$. 
Combining the isomorphism \eqref{eq_proof_mainthm_1} and Lemma \ref{lem_cat_equivalence} with the transformation \eqref{eq_lem_functor}, we get the natural transformation $F \colon I\Omega^G_\dR \simeq I\Omega^G$. 
Moreover, by construction it makes the diagram \eqref{diag_mainthm} commutative. 
Evaluating on each $(X, Y)$, the bottom row of \eqref{diag_mainthm} is exact by \eqref{eq_exact_IOmegaG}, and the top row is also exact by Proposition \ref{prop_exact_DOmegaG_dR}. 
By the five lemma, we see that $F$ gives the desired natural isomorphism. 

For the remaining statement, the fact that $(\widehat{I\Omega^G_\dR}, R, I, a)$ is a differential extension follows from Proposition \ref{prop_axiom_diffcoh}. 
For the $S^1$-integration $\int$, as we will see in Remark \ref{rem_integration_push}, $\int$ is a special case of the differential pushforwards in Section \ref{sec_push}. 
The statement on the $S^1$-integration follows by Theorem \ref{thm_push}. 
But we also remark that it is easy to give a direct proof in this case, using the fact that the bounding stable fr-structure $g_{S^1}$ on $S^1 = S^1 \times \pt$ in Definition \ref{def_Gstr_S1} defines the element in $\Omega^\fr_1(S^1)$ which maps to the suspension of the unit in $\Omega^\fr_1(S^1, \pt)$ and maps to the trivial element in $\Omega^\fr_1(\pt)$. 
The last statement follows from Proposition \ref{prop_N_G_canonical}. 
This completes the proof. 
\end{proof}

\subsection{Examples of elements in $(\widehat{I\Omega^G_\mathrm{dR}})^*$}\label{subsec_examples}
In this subsection, we give examples of elements in $(\widehat{I\Omega^G_\mathrm{dR}})^*(-)$ along with the corresponding invertible QFT's. 
In this subsection we only list examples. 
In the subsequent paper \cite{Yamashita2021} we will give topological characterization of some of the examples. 

\begin{ex}[The holonomy theory (1)]\label{ex_hol_target}
In this example we consider $G = \mathrm{SO}$. 
Fix a manifold $X$ and a hermitian line bundle with unitary connection $(L, \nabla)$ over $X$. 
Then we get an element
\begin{align*}
    (c_1(\nabla), \mathrm{Hol}_{\nabla}) \in (\widehat{I\Omega^{\mathrm{SO}}_{\mathrm{dR}}})^{2}(X). 
\end{align*}
Here, 
\begin{itemize}
    \item $c_1(\nabla) = \frac{\sqrt{-1}}{2\pi} F_\nabla \in \Omega_{\mathrm{clo}}^2(X)$ is the first Chern form of $\nabla$. 
    Identifying $\R$ with the degree-zero component of $N_{\mathrm{SO}}^\bullet$, we regard $\Omega_{\mathrm{clo}}^2(X) \subset  \Omega_{\mathrm{clo}}^2(X; N_{\mathrm{SO}}^\bullet)$. 
    \item The homomorphism $\mathrm{Hol}_\nabla \colon \mathcal{C}^{\mathrm{SO}_\nabla}_{1}(X) \to \R/\Z$ is given by the holonomy along the closed curve in $X$. 
    More precisely, an element $[M, g, f]$ in $\mathcal{C}^{\mathrm{SO}_\nabla}_{ 1}(X)$ consists of a closed oriented one-dimensional manifold $M$ with a map $f \in C^\infty(M, X)$, together with additional information on metric and connection. 
    Regarding it just as an oriented closed curve in $X$, we define $\mathrm{Hol}_\nabla([M, g, f])$ to be the holonomy of $(L, \nabla)$ along the curve, by identifying $\R/\Z \simeq \mathrm{U}(1)$. 
\end{itemize}
\end{ex}

\begin{ex}[The holonomy theory (2)]\label{ex_hol_internal}
In this example we consider $G = \mathrm{SO} \times \mathrm{U}(1)$. 
Here $\mathrm{U}(1)$ is the {\it internal symmetry group} explained in Example \ref{ex_Gstr} (4). 
We have an element
\begin{align*}
    (1 \otimes c_1, \mathrm{Hol}) \in (\widehat{I\Omega^{\mathrm{SO}\times \mathrm{U}(1)}_{\mathrm{dR}}})^{2}(\pt). 
\end{align*}
Here, 
\begin{itemize}
    \item We have $N_{\mathrm{SO} \times \mathrm{U}(1)}^\bullet =\left( \varprojlim_d(\mathrm{Sym} ( \mathfrak{so}(d, \R))^*)^{\mathrm{SO}(d; \R)} \otimes_\R (\mathrm{Sym} (\mathfrak{u}(1))^*)^{\mathrm{U}(1)}\right)^\bullet$. The first Chern polynomial $c_1 \in \left((\mathrm{Sym} (\mathfrak{u}(1))^*)^{\mathrm{U}(1)}\right)^2$ gives the element $1 \otimes c_1 \in \Omega^2_{\mathrm{clo}}(\pt; N_{\mathrm{SO} \times \mathrm{U}(1)}^\bullet) = N_{\mathrm{SO} \times \mathrm{U}(1)}^2$. 
    \item The homomorphism $\mathrm{Hol} \colon \mathcal{C}^{\mathrm{SO} \times \mathrm{U}(1)_\nabla}_{ 1}(\pt) \to \R/\Z$ is given by the holonomy of the internal $\mathrm{U}(1)$-connection. 
    More precisely, an element $[M, g]$ in $\mathcal{C}^{\mathrm{SO}\times \mathrm{U}(1)_\nabla}_{ 1}({\pt})$ consists of a closed oriented one-dimensional manifold $M$ with a principal $\mathrm{U}(1)$-bundle with connection, together with other data. 
    We define $\mathrm{Hol}([M, g])$ to be the holonomy of the $\mathrm{U}(1)$-connection, by identifying $\R/\Z \simeq \mathrm{U}(1)$.
\end{itemize}
\end{ex}

\begin{ex}[The classical Chern-Simons theory]\label{ex_CCS}
Fix a compact Lie group $H$ and an element $\lambda \in H^n(BH; \Z)$. 
The corresponding {\it classical Chern-Simons theory} (\cite{FreedCCS1}, \cite{FreedCCS2}) is an invertible QFT on $(n-1)$-dimensional manifolds equipped with orientations and principal $H$-bundles with connection. 
This generalizes Example \ref{ex_hol_internal}, which corresponds to $c_1 \in H^2(B\mathrm{U}(1); \Z)$. 
Its partition functions are given by the {\it Chern-Simons invariants} of $H$-connections. 
Here we recall its definition. 

Let $\lambda_\R\in H^*(BH; \R)$ be the $\R$-reduction of the element $\lambda$. 
Consider the category $\mathcal{C}_H$ of triples $(P, M, \nabla)$, where $P \to M$ is a smooth principal $H$-bundle over a manifold and $\nabla$ is a $H$-connection on $P$. 
We fix the following data. 
\begin{enumerate}
    \item An object $(\mathcal{E}, \mathcal{B}, \nabla_{\mathcal{E}})$ which is {\it $(n+1)$-classifying}, i.e., any object $(P, M, \nabla)$ in $\mathcal{C}_H$ with $\dim M\le n$ admits a morphism to $(\mathcal{E}, \mathcal{B}, \nabla_{\mathcal{E}})$, and any such morphisms $\phi_1$ and $\phi_2$ are smoothly homotopic. By the theorem of Narasimhan-Ramanan \cite{NarasimhanRamanan1961} such an
object exists. 
    \item A differential lift $\widehat{\lambda}\in \widehat{H}^n(\mathcal{B}; \Z)$ of the element $\lambda \in H^n(\mathcal{B}; \Z) \simeq H^n(BH;\Z)$ such that $R(\widehat{\lambda}) = \mathrm{cw}_{\nabla_{\mathcal{E}}}(\lambda_\R) \in \Omega^n_{\mathrm{clo}}(\mathcal{B})$. 
    Here $\widehat{H}^n(-; \Z)$ is the differential ordinary cohomology group, for example given by the Cheeger-Simons model explained in Example \ref{ex_diff_character}. 
\end{enumerate}

The Chern-Simons invariants are defined using the pushforward in differential ordinary cohomology $\widehat{H\Z}$. 
In terms of the Cheeger-Simons model $\widehat{H\Z}_{\mathrm{CS}}$ in Example \ref{ex_diff_character}, the pushforward map $(p_M, o)_* \colon \widehat{H}^{\dim M + 1}(M; \Z) \to \widehat{H}^{1}(\pt; \Z) \simeq \R/\Z$ for a closed oriented manifold $(M, o)$ is given by the evaluation on the fundamental cycle. 

\begin{defn}[{The Chern-Simons invariants}]
Let $\lambda \in H^n(BH; \Z)$ and fix the data (1) and (2) above. 
Let $M$ be an $(n-1)$-dimensional closed manifold equipped with an orientation $o_M$ and a principal $H$-bundle with connection $(P, \nabla)$. 
Choose a morphism $\phi \colon (M, P, \nabla) \to (\mathcal{E}, \mathcal{B}, \nabla_{\mathcal{E}})$ in $\mathcal{C}_H$.   
We define the {\it Chern-Simons invariant} of $(M, o, P, \nabla)$ by
\begin{align}\label{eq_def_CS_inv}
    h_{\mathrm{CS}_{\widehat{\lambda}}} (M, o, P, \nabla) := (p_M, o)_*\phi^* \widehat{\lambda} \in \widehat{H}^1(\pt; \Z) \simeq \R/\Z. 
\end{align}
The value \eqref{eq_def_CS_inv} does not depend on the choice of $\phi$. 
\end{defn}

The classical Chern-Simons theory corresponds to the element
\begin{align}\label{eq_elem_CCS}
    (1 \otimes\lambda_\R, h_{\mathrm{CS}_{\widehat{\lambda}}} )\in (\widehat{I\Omega^{\mathrm{SO}\times H}_{\mathrm{dR}}})^{n}(\pt). 
\end{align}
Here $1 \otimes \lambda_\R$ is as in Example \ref{ex_hol_internal}, and $h_{\mathrm{CS}_{\widehat{\lambda}}}$ is regarded as a homomorphism from $\mathcal{C}^{\mathrm{SO} \times H_\nabla}_{n-1}(\pt)$. 

Now we analyze the dependence on the choice of a lift $\widehat{\lambda}$ of $\lambda$ in (2). 
By the axioms of differential cohomology (Definition \ref{def_diffcoh}), we see that two choices $\widehat{\lambda}_1$ and $\widehat{\lambda}_2$ differs by an element in $H^{n-1}(\mathcal{B}; \R) \simeq H^{n-1}(BH; \R)$, i.e., there exists an element $\alpha \in H^{n-1}(\mathcal{B}; \R)$ with
\begin{align}\label{eq_ex_CCS_difference}
    a_{\mathrm{CS}}(\alpha) = \widehat{\lambda}_1 - \widehat{\lambda}_2. 
\end{align}
In particular, if $n$ is even, the lift $\widehat{\lambda}$ is unique because $H^{\mathrm{odd}}(BH; \R) = 0$. 
In general it is possible that the difference \eqref{eq_ex_CCS_difference} is nonzero, and in such a case the two elements \eqref{eq_elem_CCS} constructed from them are different. 
But they define the same element in $(I\Omega^{\mathrm{SO}\times H}_{\mathrm{dR}})^{n}(\pt)$, 
\begin{align*}
    I(1 \otimes\lambda_\R, h_{\mathrm{CS}_{\widehat{\lambda}_1}} )
    = I(1 \otimes\lambda_\R, h_{\mathrm{CS}_{\widehat{\lambda}_2}} )
    \in (I\Omega^{\mathrm{SO}\times H}_{\mathrm{dR}})^{n}(\pt). 
\end{align*}
This is because 
\begin{align*}
    (1 \otimes\lambda_\R, h_{\mathrm{CS}_{\widehat{\lambda}_1}} ) - (1 \otimes\lambda_\R, h_{\mathrm{CS}_{\widehat{\lambda}_2}} )
    = a(1 \otimes \alpha). 
\end{align*}
Here the domain of $a$ in \eqref{eq_def_DOmega_a} in this case is 
$\Omega^{n-1}(\pt; N_{\mathrm{SO} \times H}^\bullet) / \mathrm{Im} d = (H^*(B\mathrm{SO}; \R) \otimes_\R H^*(BH; \R))^{n-1}$.  
Thus we see that, the deformation class
\begin{align}\label{eq_elem_CCS_def}
     I(1 \otimes\lambda_\R, h_{\mathrm{CS}_{\widehat{\lambda}}} )
    \in (I\Omega^{\mathrm{SO}\times H}_{\mathrm{dR}})^{n}(\pt). 
\end{align}
is independent of the choice of the lift $\widehat{\lambda}$. 
\end{ex}

\begin{ex}[The theory of massive free complex fermions]\label{ex_complex_eta}
In this example we consider $G = \mathrm{Spin}^c$. 
Let $k$ be a positive integer. 
Recall that we have constructed a model $(\widehat{I\Omega^G_{\mathrm{ph}}})^*$ in Subsubsection \ref{sec_physical_app} which is isomorphic to $(\widehat{I\Omega^G_{\mathrm{dR}}})^*$ by Proposition \ref{prop_sub_model_isom}. 
We are going to construct an element in $\left(\widehat{I\Omega^{\mathrm{Spin}^c}_\mathrm{ph}}\right)^{2k}$. 

Given a closed $(2k-1)$-dimensional manifold $M$ with a physical tangential $\mathrm{Spin}^c$-structure $g$ (Definition \ref{def_phys_Gstr}), we set
\begin{align*}
    \overline{\eta}(M, g) := \overline{\eta}(D_{M})
    = \frac{\eta(D_{M}) + \dim \ker D_{M}}{2} \in \R. 
\end{align*}
where $D_{M}$ is the $\mathrm{Spin}^c$-Dirac operator on $M$ with respect to $g$ and $\eta(D_M) \in \R$ is its eta invariant. 
Note that we have used the assumption that the connection in $g$ is compatible with the Levi-Civita connection. 

Recall that the Atiyah-Patodi-Singer index theorem (\cite{APS1}, \cite{APS2}, \cite{APS3}) says that, if $(W, g_W)$ is a compact $2k$-dimensional manifold with boundary with a collar structure equipped with a geometric $\mathrm{Spin}^c$-structure which is compatible with Levi-Civita connection, we have
\begin{align}\label{eq_APS}
    \mathrm{Ind}_{\mathrm{APS}}(D_{W})= \int_{W}\mathrm{Todd}(g_W) - \overline{\eta}(\del W, \del g_W). 
\end{align}
Here the left hand side of \eqref{eq_APS} is the Atiyah-Patodi-Singer index of the Dirac operator on $W$, which is an integer. 
Thus, regarding $\overline{\eta}$ as a homomorphism $\overline{\eta} \colon \mathcal{C}_{\mathcal{D}}(\pt) \to \R/\Z$ with $\mathcal{D} = h\mathrm{Bord}^{\mathrm{Spin}^c_{\mathrm{ph}}}_{2k-1}(\pt)$, we get the element
\begin{align*}
    (( \mathrm{Todd})|_{2k}, \overline{\eta}) \in \left(\widehat{I\Omega^{\mathrm{Spin}^c}_{\mathrm{ph}}}\right)^{2k}(\pt) \simeq \left({I\Omega^{\mathrm{Spin}^c}_{\mathrm{ph}}}\right)^{2k}(\pt) \simeq \left({I\Omega^{\mathrm{Spin}^c}_{\mathrm{dR}}}\right)^{2k}(\pt). 
\end{align*}

This example can be generalized to include target spaces. 
Fix a manifold $X$ and a hermitian vector bundle with unitary connection $(E, h^E, \nabla^E)$ over $X$. 
Then, using the reduced eta invariants $\overline{\eta}_{\nabla^E}$ for Dirac operators twisted by the pullback of $(E, h^E, \nabla^E)$, we get the element
\begin{align*}
    \left((\mathrm{Ch}(\nabla^E) \otimes \mathrm{Todd})|_{2k}, \overline{\eta}_{\nabla^E}\right) \in \left(\widehat{I\Omega^{\mathrm{Spin}^c}_{\mathrm{ph}}}\right)^{2k}(X) \simeq \left(\widehat{I\Omega^{\mathrm{Spin}^c}_{\mathrm{dR}}}\right)^{2k}(X). 
\end{align*}
Its deformation class in $\left({I\Omega^{\mathrm{Spin}^c}_{\mathrm{ph}}}\right)^{2k}(X)$ only depends on the class $[E]\in K^0(X)$. 
\end{ex}

\begin{ex}[The theory of massive free real fermions]\label{ex_real_eta}
Here we consider the real version of Example \ref{ex_complex_eta}. 
Now $G = \mathrm{Spin}$. 
We consider the theory on $(8m+3)$-dimension with nonnegative integer $m$. This is the dimension where the difference from Example \ref{ex_complex_eta} appears. 
On Spin manifolds, the Atiyah-Patodi-Singer index theorem \eqref{eq_APS} becomes
\begin{align}\label{eq_APS_real}
    \mathrm{Ind}_{\mathrm{APS}}(D_{W})= \int_{W}\widehat{A}(g_W) - \overline{\eta}(\del W, \del g_W). 
\end{align}
Moreover, if $\dim W \equiv 4 \pmod 8$, the APS index is an {\it even} integer. 
This allows us to define the element
\begin{align*}
    \left( \frac{1}{2}\widehat{A}|_{8m+4}, \frac{1}{2}\overline{\eta}\right) \in \left(\widehat{I\Omega^{\mathrm{Spin}}_{\mathrm{ph}}}\right)^{8m+4}(\pt) \simeq \left({I\Omega^{\mathrm{Spin}}_{\mathrm{ph}}}\right)^{8m+4}(\pt) \simeq \left({I\Omega^{\mathrm{Spin}}_{\mathrm{dR}}}\right)^{8m+4}(\pt). 
\end{align*}
\end{ex}

\subsection{The refinement of the Anderson self-duality of $H\Z$}\label{subsec_self_duality}
In this subsection, we relate our model $I\Z_\mathrm{dR}:= I\Omega^\fr_\mathrm{dR}$ (see the end of Definition \ref{def_hat_DOmegaG}) with the ordinary cohomology theory. 
The ordinary cohomology theory $H\Z$ is Anderson self-dual, with the self-duality element $\gamma_H \in [H\Z, I\Z] $, whose multiplication gives the isomorphism $H\Z \simeq IH\Z$. This element corresponds to a generator in $[H\Z, I\Z] = I\Z^0(H\Z) = \Hom(\pi_0(H\Z), \Z) =\Z$. 
Using the obvious analogy of the Cheeger-Simons differential character model $\widehat{H\Z}_{\mathrm{CS}}$ (Example \ref{ex_diff_character}) and our differential model $\widehat{I\Z_{\dR}}$, we can refine $\gamma_H \colon H\Z \to I\Z$ to a transformation $\widehat{\gamma}_\dR \colon \widehat{H}_{\mathrm{CS}}^*(-; \Z) \to (\widehat{I\Z}_\mathrm{dR})^*(-)$. 

First we recall the definition of $\gamma_H \colon H\Z \to I\Z$. 
Since we have $[H\Z, I\Z] = I\Z^0(H\Z) = \Hom(\pi_0(H\Z), \Z) =\Z$ and we have $[H\Z, I\Z] = \varprojlim_{n}[H\Z_n, I\Z_n]$ by the vanishing of phantoms $\varprojlim_{n}[H\Z_{n}, I\Z_{n-1}] = 0$ (\cite[Chapter III, Theorem 4.21]{rudyak1998}), 
the transformation of cohomology theories $H\Z \to I\Z$ on CW-pairs are classified by its value on $\pt$. 

In order to define $\widehat{\gamma}_\mathrm{dR}$, we remark the following. 
Let $(\omega, k) \in \widehat{H}_{\mathrm{CS}}^n(X, Y; \Z)$. 
Then we get a group homomorphism (here $\fr_\nabla = \fr$ in the obvious sense), 
\begin{align*}
    k|_{\fr} \colon \mathcal{C}^\fr_{n-1}(X, Y) \to \R/\Z
\end{align*}
by, given a (differential) smooth stable tangential $\fr$-cycle $(M, g, f)$ of dimension $(n-1)$ over $(X, Y)$, choosing any representative $t_M \in Z_{\infty, n-1}(X, Y; \Z)$ of the fundamental class and applying $k$ to $t_M$. 
This value does not depend on the choice of $t_M$ because of the compatibility condition for $(\omega, k)$. 

\begin{defn}[{$\widehat{\gamma}_\mathrm{dR}$ and $\gamma_\mathrm{dR}$}]\label{def_gamma_ph}
For a pair of manifolds $(X, Y)$ and $n \in \Z$, we define a homomorphism 
\begin{align*}
    \widehat{\gamma}_\mathrm{dR}^n \colon\widehat{H}_{\mathrm{CS}}^n(X, Y; \Z) \to (\widehat{I\Z_\mathrm{dR}})^n(X, Y)
\end{align*}  
by sending an element $(\omega, k)$ to $(\omega, k|_\fr)$. 
The compatibility condition (Definition \ref{def_hat_DOmegaG} (1) (c)) for the pair $(\omega, k|_\fr)$ follows from the compatibility condition \eqref{eq_compatibility_diffchar} for the pair $(\omega, k)$. 

We easily see that we have $a_{\mathrm{dR}} =\widehat{\gamma}^n_\mathrm{dR}\circ a_{\mathrm{CS}}$, so it induces the homomorphism on the quotient, 
\begin{align*}
    {\gamma}_\mathrm{dR}^n \colon H^n(X, Y; \Z) \to ({I\Z_\mathrm{dR}})^n(X, Y). 
\end{align*} 
We also easily see that these homomorphisms are functorial, so give natural transformations $\widehat{\gamma}_\mathrm{dR}\colon \widehat{H}_{\mathrm{CS}}^*(-; \Z) \to I\Z_\mathrm{dR}^*$ and $\gamma_\mathrm{dR}\colon H\Z^* \to I\Z_\mathrm{dR}^*$ between functors $\mathrm{MfdPair}^{\mathrm{op}} \to \mathrm{Ab}^\Z$. 
\end{defn}

\begin{prop}\label{prop_duality_HZ}
Under the isomorphism $I\Z_\dR \simeq I\Z$ in Theorem \ref{thm_IOmega_dR=IOmega}, the natural transformation $\gamma_\dR$ coincides with the self-duality map $\gamma_H \colon H\Z \to I\Z$. 
\end{prop}

To show Proposition \ref{prop_duality_HZ}, we need to describe the Anderson self-duality homomorphism $\gamma_H \colon H\Z \to I\Z$ in the model of $I\Z$ by Hopkins-Singer (Fact \ref{fact_IZ_model}). 
We work in the category of CW-pairs. 
We use the topological variant of the framed bordism Picard groupoid $h\mathrm{Bord}^\fr_{n-1}(X, Y)_{\mathrm{top}}$, which is defined for any CW-pair $(X, Y)$, by requiring the map to $(X, Y)$ to be continuous rather than smooth in Definition \ref{def_Bord_Pic_cat}. 
By the theorem of Pontryagin-Thom, we have an equivalence
\begin{align}\label{eq_equivalence_cat_top}
    h\mathrm{Bord}^\fr_{n-1}(X, Y)_{\mathrm{top}} \simeq \pi_{\le 1}((L(X/Y))_{1-n}). 
\end{align}

\begin{lem}\label{lem_duality_HZ_singular}
Let $(X, Y)$ be a CW-pair.  
Choose a functor of Picard groupoids, 
\begin{align}\label{eq_functor_triangulation}
    T_{X, Y} \colon h\mathrm{Bord}^\fr_{n-1}(X, Y)_{\mathrm{top}} \to \left(C_n(X, Y; \Z)/B_n(X, Y; \Z) \xrightarrow{\del} Z_{n-1}(X, Y; \Z)\right), 
\end{align}
where $C_*$, $Z_*$ and $B_*$ denote the singular chains, cycles and boundaries, 
by choosing fundamental cycles on objects and morphisms of $h\mathrm{Bord}^\fr_{n-1}(X, Y)_{\mathrm{top}}$. 

Given a singular cohomology class $[c]  \in H^n(X, Y; \Z)$, 
take a representative by a singular $n$-cocycle $c \in Z^n(X, Y; \Z)$. 
Consider the functor of Picard groupoids, 
\begin{align}\label{eq_functor_ev_sing}
   \mathrm{ev}_c \colon  \left(C_n(X, Y; \Z)/B_n(X, Y; \Z) \xrightarrow{\del} Z_{n-1}(X, Y; \Z)\right) \to (\Z \to 0), 
\end{align}
defined by the evaluation of $c$ on morphisms. 
Then the natural isomorphism class of the composition of the functors $\mathrm{ev}_c \circ T_{X, Y}$ is independent of the choice of $T_{X, Y}$ and the cocycle $c$ representing $[c]$, and defines a homomorphism
\begin{align}\label{eq_lem_duality_HZ}
    H^n(X, Y; \Z) &\to \pi_0 \mathrm{Fun}_{\mathrm{Pic}} \left(h\mathrm{Bord}^\fr_{n-1}(X, Y)_{\mathrm{top}}, (\Z \to 0)\right) \simeq I\Z^n(X, Y)  \\
     [c] &\mapsto [\mathrm{ev}_c \circ T_{X, Y}], \notag
\end{align}
where the last isomorphism use the Fact \ref{fact_IZ_model} and the equivalences $(\Z \to 0) \simeq (\R \to \R/\Z)$ and \eqref{eq_equivalence_cat_top}. 
Moreover, the homomorphism \eqref{eq_lem_duality_HZ} coincides with the transformation given by the Anderson self-duality element,
\begin{align*}
    \gamma_H \colon H^n(X, Y; \Z) \to I\Z^n(X, Y) , 
\end{align*}

\end{lem}
\begin{proof}
The first claim is easy, because both the natural isomorphism classes of the functors $T_{X, Y}$ and $\mathrm{ev}_c$ are independent of the choices. 
We can easily check that the homomorphism \eqref{eq_lem_duality_HZ} is functorial and compatible with the relative coboundary maps,  thus defining a transformation $H\Z \to I\Z$ of cohomology theories on CW-pairs. 

Since we have $[H\Z, I\Z] = I\Z^0(H\Z) = \Hom(\pi_0(H\Z), \Z) =\Z$ and we have $[H\Z, I\Z] = \varprojlim_{n}[H\Z_n, I\Z_n]$ by the vanishing of phantoms $\varprojlim_{n}[H\Z_{n}, I\Z_{n-1}] = 0$ (\cite[Chapter III, Theorem 4.21]{rudyak1998}), 
the transformation of cohomology theories $H\Z \to I\Z$ on CW-pairs is classified by its value on $\pt$. 
We can easily check that the transformation given by \eqref{eq_lem_duality_HZ} coincides with $\gamma_H$ on $\pt$, thus we conclude that it coincides with $\gamma_H$ as a transformation of cohomology theories. 
This finishes the proof. 
\end{proof}

Now we prove Proposition \ref{prop_duality_HZ}. 
\begin{proof}[Proof of Proposition \ref{prop_duality_HZ}]

Here we use the model $\widehat{H\Z}_\HS$ of the ordinary differential cohomology theory in terms of {\it differential cocycles} \cite{HopkinsSinger2005}. 
An element in $\widehat{H}_\HS^n(X, Y; \Z)$ is represented by a triple $(c, h_\R, \omega) \in Z^n(X, Y; \Z) \times C^{n-1}(X, Y; \R) \times \Omega_{\mathrm{clo}}^n(X, Y)$ such that $\delta h_\R = c - \omega$ (as smooth singular $\R$-cochains). 
Such a triple is called a differential cocycle. 
The forgetful functor $I$ is given by
\begin{align*}
    I \colon \widehat{H}_\HS^n(X, Y; \Z) \to H^n(X, Y; \Z), &\ ([c, h_\R, \omega]) \mapsto [c]. 
\end{align*}
The models $\widehat{H\Z}_\HS$ and $\widehat{H\Z}_{\mathrm{CS}}$ are isomorphic with the isomorphism given by
\begin{align}\label{eq_HS_CS_isom}
    \widehat{H}_\HS^n(X, Y; \Z) \simeq \widehat{H}_{\mathrm{CS}}^n(X, Y; \Z),  \ [c, h_\R, \omega] \mapsto (\omega, h), 
\end{align}
where we set $h := h_\R \pmod \Z \colon Z_{n-1}(X, Y; \Z) \to \R/\Z$, and its restriction to $Z_{\infty, n-1}(X, Y; \Z)$ is denoted by the same symbol. 

Assume that we are given an element $(\omega, h) \in \widehat{H}_{\mathrm{CS}}^n(X, Y; \Z)$. 
Take a differential cocycle $(c, h_\R, \omega)$ which maps to $(\omega, h)$ under the map \eqref{eq_HS_CS_isom}. 
Then consider the diagram of functors, 
\begin{align}\label{diag_proof_duality_HZ}
\xymatrix{
    \left(C_{\infty, n}(X, Y; \Z)/B_{\infty, n}(X, Y; \Z) \xrightarrow{\del} Z_{\infty, n-1}(X, Y; \Z)\right)
    \ar[r]^-{\mathrm{ev}_c} \ar[rd]_-{\mathrm{ev}_{(\omega, h)}} & (\Z \to 0) \ar[d]^-{\simeq} \\
    & (\R \to \R/\Z)
    }
\end{align}
Here the top arrow is the restriction of \eqref{eq_functor_ev_sing} to smooth singular chains, and the rightdown arrow is given by the evaluation of $h$ on objects and $\omega$ on morphisms. 
We can easily check that the two compositions of functors in \eqref{diag_proof_duality_HZ} are naturally isomorphic, with natural transformation given by $h_\R$. 
Moreover, by definition of $\widehat{\gamma}_\dR$ we see that the functor \eqref{eq_associated_functor} $F_{\widehat{\gamma}_\dR(\omega, h)}$ associated to the element $\widehat{\gamma}_\dR(\omega, h) \in \widehat{I\Z}_\dR^n(X, Y)$ satisfies
\begin{align*}
    F_{\widehat{\gamma}_\dR(\omega, h)} = \mathrm{ev}_{(\omega, h)} \circ T_{\infty, X, Y}, 
\end{align*}
where $T_{\infty, X, Y}$ denotes an obvious smooth singular version of the functor \eqref{eq_functor_triangulation}. 
As explained in Subsection \ref{subsec_proof_isom}, the isomorphism $I\Z_{\dR} \simeq I\Z$ sends the class $\gamma_\dR([\omega, h])$ to the class of the associated functor $F_{\widehat{\gamma}_\dR(\omega, h)}$. 
By the natural isomorphism between two compositions in \eqref{diag_proof_duality_HZ} and Lemma \ref{lem_duality_HZ_singular}, together with the equivalence $(C_{\infty,n}(X, Y; \Z)/B_{\infty,n}(X, Y; \Z) \xrightarrow{\del} Z_{\infty,n-1}(X, Y; \Z) )  \simeq (C_n(X, Y; \Z)/B_n(X, Y; \Z) \xrightarrow{\del} Z_{n-1}(X, Y; \Z) ) $, we see that the class of the functor $F_{\widehat{\gamma}_\dR(\omega, h)}$ coincides with the class $\gamma_H([c])$. 
This completes the proof. 

\end{proof}

\subsection{The normal case}\label{subsec_normal}

So far we have focused on the {\it tangential} $G$-bordism theories and its Anderson duals. 
However, by a straightforward modification, we can construct the corresponding models for the Anderson duals $(I\Omega^{G^\perp})^*$ to the {\it normal} $G$-bordism theories $\Omega^{G^\perp}$ corresponding to the Thom spectrum $MG$. 
In this subsection we outline the construction. 

\begin{defn}[{Differential stable normal $G$-structures on vector bundles}]\label{def_diff_Gstr_vec_normal}
Let $V$ be a real vector bundle of rank $n$ over a manifold $M$. 
\begin{enumerate}
    \item[(1)]
A {\it representative of differential stable normal $G$-structure} on $V$ is a quadruple $\widetilde{g}^\perp = (d, P, \nabla, \psi)$, where $d \ge n$ is an integer, 
$(P, \nabla)$ is a principal $G_{d-n}$-bundle with connection over $M$ and
$\psi \colon (P \times_{\rho_{d-n}} \R^{d-n}) \oplus V \simeq  \underline{\R}^{d}  $ is an isomorphism of vector bundles over $M$. 
\item[(2), (3)] We define the {\it stabilization} of such $\widetilde{g}^\perp$ in the same way as Definition \ref{def_diff_Gstr_vec}, and a {\it differential stable normal $G$-structure} $g^\perp$ on $V$ is defined to be a class of representatives under the stabilization relation. 
\item[(4)] We define the {\it homotopy} relation between two such $g^\perp$'s also in the same way. 
\end{enumerate}
\end{defn}

\begin{defn}[{Differential stable normal $G$-structures}]\label{def_diff_Gstr_mfd_normal}
Let $M$ be a manifold. 
A {\it differential stable normal $G$-structure} is a differential stable normal $G$-structure on the tangent bundle $TM$. 
\end{defn}

Then, the various objects introduced in Section \ref{sec_Gstr} can be modified to the normal case easily. 
We get the notion of {\it differential stable normal $G$-cycles} $(M, g^\perp, f)$, the abelian groups $\mathcal{C}_n^{G_\nabla^\perp}(X, Y)$, the bordism relations and the Picard groupoids $h\mathrm{Bord}_n^{G_\nabla^\perp}(X, Y)$.

First note that we have
\begin{align*}
     N_{G^\perp}^\bullet := H^*(MG; \R) = \varprojlim_{d} H^*(BG_d; EG_d \times_{G_d}\R_{G_d})
     =\varprojlim_{d}(\mathrm{Sym}^{\bullet/2}\mathfrak{g}_d^* \otimes_\R \R_{G_d})^{G_d}. 
\end{align*}
The proof is the same as that of Lemma \ref{lem_IOmega_R}, where now the Madsen-Tillmann spectrum $MTG$ is replaced by the Thom spectrum $MG$. 
Note that we have $N_G^\bullet = N_{G^\perp}^\bullet$.
This is because the orientation bundles of a vector bundle and its normal bundle are canonically identified. 
We use the transformation analogous to \eqref{eq_ch_N_G}, 
\begin{align}\label{eq_ch_N_G_perp}
    \mathrm{ch}' \colon (I\Omega^{G^\perp})^* \to H^*(-; N_{G^\perp}^\bullet) \simeq \Hom(\Omega^{G^\perp}_*(-), \R). 
\end{align}

The variant of the Chern-Weil construction in Definition \ref{def_chern_weil} also applies to the normal settings. 
Given a differential stable normal $G$-structure $g^\perp$ on a vector bundle $V \to W$, by the same procedure to the tangential case in Definition \ref{def_chern_weil} we get a homomorphism
\begin{align}\label{eq_cw_perp}
    \mathrm{cw}_{g^\perp}  \colon \Omega^*\left(W; N_{G^\perp}^\bullet\right) \to \Omega^*(W; \mathrm{Ori}(V)). 
\end{align}
Applied to $V = TW$, \eqref{eq_cw_perp} induces the homomorphisms corresponding to \eqref{eq_def_cw_omega_mor} and \eqref{eq_def_cw_omega_closed}. 

\begin{defn}[{$(\widehat{I\Omega^{G^\perp}_{\mathrm{dR}}})^*$ and $(I\Omega^{G^\perp}_{\mathrm{dR}})^*$}]\label{def_hat_DOmegaG_normal}
Let $(X, Y)$ be a pair of manifolds and $n$ be a nonnegative integer. 
\begin{enumerate}
    \item 
Define $(\widehat{I\Omega^{G^\perp}_{\mathrm{dR}}})^n(X, Y)$ to be an abelian group consisting of pairs $(\omega, h)$, such that
\begin{enumerate}
    \item $\omega$ is a closed $n$-form $\omega\in \Omega_{\mathrm{clo}}^n\left(X, Y; N_{G^\perp}^\bullet\right)$. 
    \item $h$ is a group homomorphism
    $h \colon \mathcal{C}^{G^\perp_\nabla}_{n-1}(X, Y) \to \R/\Z$. 
\item $\omega$ and $h$ satisfy the compatibility condition analogous to Definition \ref{def_hat_DOmegaG} (1) (c) with respect to morphisms in $h\Bord^{G^\perp_\nabla}_{n-1}(X, Y)$. 
\end{enumerate}

\item
We define a homomorphsim of abelian groups, 
\begin{align*}
    a \colon \Omega^{n-1}\left(X, Y; N_{G^\perp}^\bullet\right)/\mathrm{Im}(d) &\to  (\widehat{I\Omega^{G^\perp}_{\mathrm{dR}}})^n(X, Y) \\
    \alpha &\mapsto (d\alpha, \mathrm{cw}(\alpha)).
\end{align*}
We set
\begin{align*}
    (I\Omega^{G^\perp}_{\mathrm{dR}})^n(X, Y) := (\widehat{I\Omega^{G^\perp}_{\mathrm{dR}}})^n(X, Y)/ \mathrm{Im}(a). 
\end{align*}

\end{enumerate}
For negative integers $n$ we set $(\widehat{I\Omega^{G^\perp}_{\mathrm{dR}}})^n(X, Y) := 0$ and $(I\Omega^{G^\perp}_{\mathrm{dR}})^n(X, Y) := 0$. 
\end{defn}

The structure homomorphisms $I$, $R$, $a$ and $p$, and the $S^1$-integration map $\int$, are also defined in the same way as Definitions \ref{def_str_map_IOmega_dR} and \ref{def_integration_dR}. 
We can check that the following sequence is exact,
 \begin{align*}
         \mathrm{Hom}(\Omega^{G^\perp}_{n-1}(X, Y), \R)\to \mathrm{Hom}(\Omega^{G^\perp}_{n-1}(X, Y), \R/\Z) \xrightarrow{p} (I\Omega^{G^\perp}_{\mathrm{dR}})^n(X, Y)\\ \xrightarrow{\mathrm{ch}'}
        \mathrm{Hom}(\Omega^{G^\perp}_n(X, Y), \R)
        \to \mathrm{Hom}(\Omega^{G^\perp}_{n}(X, Y), \R/\Z) \qquad \mbox{(exact)}. 
    \end{align*}
The normal version of Theorem \ref{thm_IOmega_dR=IOmega}, which says that the above $I\Omega^{G^\perp}_\dR$ is indeed a model for $I\Omega^{G^\perp}$, and that the quintuple $(\widehat{I\Omega^{G^\perp}_\dR}, R, I, a, \int)$ is its differential extension of $\left((I\Omega^{G^\perp})^*, \mathrm{ch}' \right)$ with $S^1$-integration, 
can be shown by the exactly same proof, replacing $MTG$ with $MG$.

\section{The multiplications by the bordism cohomology theories}\label{sec_module}
Assume we are given three tangential structure groups $G_i := \{(G_i)_d, (s_i)_d, (\rho_i)_d\}_{d \in \Z_{\ge 0}}$ for $i = 1, 2, 3$, and a {\it homomorphism} 
\begin{align}\label{eq_multi_G}
   \mu \colon G_1 \times G_2 \to  G_3
\end{align}
of tangential structure groups. Then we get a morphism between the Madsen-Tillmann spectra, 
\begin{align}\label{eq_multi_MTG}
    MTG_1 \wedge MTG_2 \to MTG_3. 
\end{align}
Here a homomorphism \eqref{eq_multi_G} is defined in a fairly obvious way, whose precise definition is given in Remark \ref{rem_def_multi_G} below. 
It consists of group homomorphisms
$(G_1)_{d} \times (G_2)_{d'} \to (G_3)_{d+d'}$, which are compatible with the structure homomorphisms $(s_i)_d, (\rho_i)_d$'s and the multiplicative structure on $\mathrm{O}$. 
There are many interesting examples as follows. 
\begin{ex}\label{ex_G123}
\begin{enumerate}
    \item An important class of examples arises from {\it multiplicative} $G$, where $MTG$ is a ring spectrum. 
    In this case we set $G = G_1 = G_2 = G_3$. 
For example $\mathrm{O}$, $\mathrm{SO}$, $\mathrm{Spin}$ and $\fr$ are equipped with the natural multiplicative structure. 
\item The case $G_1 = G_3 = \mathrm{Pin}^+$ and $G_2 = \mathrm{Spin}$. 
\item For any $G$, we have a homomorphism $G \times \fr \to G$. 
Here recall that $\fr = \{1\}_{d \in \Z}$. 
The group homomorphism $G_d \times \fr_{d'} = G_d \times 1 \to G_{d + d'}$ is the composition $s_{d+d'-1} \circ \cdots \circ s_d$ of the stabilization homomorphisms in $G$. 
Actually, as we will see in Remark \ref{rem_integration_push}, the differential pushforwards we introduce in Section \ref{sec_push} in this case recovers the $S^1$-integration map $\int$ of $\widehat{I\Omega_\dR^G}$ (Definition \ref{def_integration_dR}). 
\end{enumerate}
\end{ex}

In general, if we have a morphism of spectra
\begin{align*}
   t \colon  E_1 \wedge E_2 \to E_3, 
\end{align*}
we get the following morphism on the Anderson duals. 
\begin{align}\label{eq_module_IE_general}
    IE_3 \wedge E_2 \xrightarrow{It \wedge \id_{E_2}} I(E_1 \wedge E_2) \wedge E_2 \xrightarrow{\mathrm{ev}_{E_2}} IE_1. 
\end{align}
Here $It$ is the Anderson dual to $t$, and the second arrow is the evaluation on $E_2$ (recall $I(E_1 \wedge E_2) = [E_1 \wedge E_2, I\Z]$). 
Applying \eqref{eq_module_IE_general} to \eqref{eq_multi_MTG}, we get 
the morphism
\begin{align}\label{eq_module_IMTG}
    IMTG_3 \wedge MTG_2 \to IMTG_1, 
\end{align}
inducing the following homomorphism for each pair of integers $(n, r)$, 
\begin{align}\label{module_IOmega}
     (I\Omega^{G_3})^n(X, Y) \otimes (\Omega^{G_2})^{-r}(X)  \to (I\Omega^{G_1})^{n-r}(X, Y), 
\end{align}
which is natural in $(X, Y)$. 
The purpose of this section is to get its differential refinement, 
\begin{align}\label{eq_module_intro}
    (\widehat{I\Omega^{G_3}_{\mathrm{dR}}})^n(X, Y)\otimes (\widehat{\Omega^{G_2}})^{-r}(X) \xrightarrow{\cdot} (\widehat{I\Omega^{G_1}_{\mathrm{dR}}})^{n-r}(X, Y).
\end{align}
Here, for $(\widehat{\Omega^{G_2}})^{-r}(X)$ we use a cycle-model constructed by Bunke, Schick, Schr\"{o}der and Wiethaup \cite{BunkeSchickSchroderMU}, which we explain in Subsection \ref{subsec_hat_OmegaG}.

Before proceeding, we explain the rough idea of the construction. 
A particularly nice class of elements in $(\widehat{\Omega^{G_2}})^{-r}(X)$ are represented by cycles of the form $(p \colon N \to X, g_p)$, where $p$ is a fiber bundle whose fibers are closed $r$-dimensional manifold and $g_p$ is a differential stable $G_2$-structure on the relative tangent bundle. 
The multiplication \eqref{eq_module_intro} by such an element is given as follows. 
For an object $[M, g_M, f] \in \mathcal{C}^{(G_1)_\nabla}_{n-r-1}(X, Y)$, we can define an object $\mu\left([M, g_M, f] \times_X [p \colon N \to X, g_p] \right)\in \mathcal{C}^{(G_3)_\nabla}_{n-1}(X, Y)$ by the fiber product over $X$ (the pullback of the bundle) in a fairly obvious way. 
Then, for an element $(\omega, h) \in (\widehat{I\Omega^{G_3}_{\mathrm{dR}}})^n(X, Y)$, we assign an element in $(\widehat{I\Omega^{G_1}_{\mathrm{dR}}})^{n-r}(X, Y)$ whose evaluation on $[M, g_M, f] \in \mathcal{C}^{(G_1)_\nabla}_{n-r-1}(X, Y)$ is given by the evaluation of $h$ on this fiber product. 
As explained in Subsubsection \ref{subsubsec_compactification}, in a physical interpretation, this process corresponds to {\it compactification} of QFT's.

\begin{rem}\label{rem_def_multi_G}
Here we give the definition of {\it homomorphism} $\mu \colon G_1 \times G_2 \to G_3$ in \eqref{eq_multi_G}. 
$\mu$ consists of group homomorphisms
$\mu_{d, d'} \colon (G_1)_{d} \times (G_2)_{d'} \to (G_3)_{d+d'}$ for each $(d, d')$ with the following conditions. 
\begin{enumerate}
    \item For any $(d, d')$, the following diagram commutes. 
    \begin{align}\label{eq_cond1_multi_G}
        \xymatrix{
        (G_1)_d \times (G_2)_{d'} \ar[r]^-{\mu_{d, d'}} \ar[d]^{(\rho_1)_d \times (\rho_2)_{d'}} & (G_3)_{d+d'} \ar[d]^-{(\rho_3)_{d+d'}} \\
        \mathrm{O}(d, \R) \times \mathrm{O}(d', \R) \ar[r] & \mathrm{O}(d+d', \R)
        }, 
    \end{align}
    where the bottom arrow is the diagonal map in $\mathrm{O}$. 
    \item For any $(d, d')$, the left diagram below commutes, and the right diagram commute up to confugation by an element of $(G_3)_{d+d'+1}$ in the unit component. 
    \begin{align}\label{eq_cond2_multi_G}
        \xymatrix{
        (G_1)_d \times (G_2)_{d'} \ar[r]^-{\mu_{d, d'}} \ar[d]^{(s_1)_d \times \id} & (G_3)_{d+d'} \ar[d]^-{(s_3)_{d+d'}} \\
        (G_1)_{d+1} \times (G_2)_{d'} \ar[r]^-{\mu_{d+1, d'}} & (G_3)_{d+d'+1}
        }
        \xymatrix{
        (G_1)_d \times (G_2)_{d'} \ar[r]^-{\mu_{d, d'}} \ar[d]^{\id \times (s_2)_{d'}} & (G_3)_{d+d'} \ar[d]^-{(s_3)_{d+d'}} \\
        (G_1)_d \times (G_2)_{d'+1} \ar[r]^-{\mu_{d, d'+1}} & (G_3)_{d+d'+1}
        }
    \end{align}
\end{enumerate}

Here, for Condition (2), we may as well assume only the homotopy-commutativity also for the left diagram in order to produce the morphism of Madsen-Tillmann spectra in \eqref{eq_multi_MTG}. 
However, the strict-commutatibity is satisfied in most of the examples of interest. 
(In contrast to this, for the right diagram we can only expect the homotopy-commutatibity because of our convention for the stabilization \eqref{eq_stabilization}. )
We choose the above formulation to simplify the construction below. 
\end{rem}

\subsection{A preliminary--A cycle-model for $(\widehat{\Omega^G})^*$ following \cite{BunkeSchickSchroderMU}}\label{subsec_hat_OmegaG}
Bunke, Schick, Schr\"{o}der and Wiethaup \cite{BunkeSchickSchroderMU} gave a model for a differential extension of $(\widehat{\Omega^{G^\perp}})^*$ ({\it normal} $G$-bordism cohomology theory, represented by $MG$). 
They provide the detail for the case of complex bordisms, but as they note, their construction directly generalizes to any $G$. 
Moreover, their construction can be modified to give a model for $(\widehat{\Omega^G})^*$ ({\it tangential} $G$-bordisms). 
In this subsection we briefly explain it. 
For further details see \cite{BunkeSchickSchroderMU}. 
Sometimes we use different definitions from those for corresponding objects in \cite{BunkeSchickSchroderMU} for the compatibility with the conventions in the main body of this paper. 
The differences are not essential. 

\begin{defn}[Stable relative tangent bundles]\label{def_rel_tang_bundle}
Let $p \colon N \to X$ be a smooth map between manifolds with relative dimension $r := \dim N - \dim X$. 
Let us choose $(k, \phi)$, where $\phi \colon \underline{\R}^k \to p^*TX$ is a map of vector bundles over $N$ such that $\phi \oplus dp \colon \underline{\R}^k \oplus T_xN \to T_{p(x)}X$ is surjective for all $x \in N$. 
Given such $(k, \phi)$, we define the associated {\it stable relative tangent bundle for $p$ associated to $(k, \phi)$} to be the following real vector bundle of rank $(k + r)$ over $N$. 
\begin{align*}
    T(\phi, p) := \ker (\phi \oplus dp \colon \underline{\R}^k \oplus TN \to p^*TX). 
\end{align*}
\end{defn}

Using stable relative tangent bundles, we define the relative version of differential stable tangential $G$-structures as follows. 

\begin{defn}[{Differential stable relative tangential $G$-structures}]\label{def_diff_rel_Gstr}
Let $p \colon N \to X$ be a smooth map between manifolds. 
\begin{enumerate}
    \item A {\it representative of a differential stable relative tangential $G$-structure on $p$} consists of $\widetilde{g}_p =(k, \phi, P, \nabla, \psi)$, where $(k, \phi)$ is a choice as in Definition \ref{def_rel_tang_bundle}, 
    $P$ is a principal $G_{k+r}$-bundle over $N$ and $\psi \colon P \times_{G_{k+r}}\R^{k+r} \simeq T(\phi, p)$ is an isomorphism of vector bundles over $N$. 
    \item For such a $\widetilde{g}_p$, we can define its stabilization $\widetilde{g}_p(1)$ in the obvious way. 
    \item
A {\it differential stable relative tangential $G$-structure $g_p$ on $p$} is a class of such representatives under the relation $\widetilde{g}_p \sim \widetilde{g}_p(1)$. 
\end{enumerate}
\end{defn}

In particular, if $p \colon N \to X$ is a submersion, we can take $\phi = 0 \colon \underline{\R}^k \to p^*TX$ and we have $T(0, p) = \underline{\R}^k \oplus T(p)$, the stabilization of the relative tangent bundle $T(p) := \ker dp$. 
Thus a differential stable $G$-structure (Definition \ref{def_diff_Gstr_vec}) on $T(p)$ can be regarded as a special case of differential stable relative tangential $G$-structure on $p$. 
But note that the latter notion is more general. 

Recall that our manifolds are allowed to have corners. 
To define the differential stable relative tangential $G$-cycles, we need to use the following class of maps between manifolds with corners. 
\begin{defn}[{Neat maps, \cite[Appendix C]{HopkinsSinger2005}}]\label{def_neat_map}
A smooth map $p \colon N \to X$ between manifolds is called {\it neat} if it preserves the depth of points, and the map
\begin{align}
    dp \colon T_x N / T_x S^k (N) \to T_{p(x)}(X) / T_{p(x)} S^k(X)
\end{align}
is an isomorphism for all points $x \in N$, where $k := \mathrm{depth}(x) = \mathrm{depth}(p(x))$. 
\end{defn}

For example, the map $p_X \colon X \to \pt$ is neat only if $X$ has no boundary. 
The map $[0, \infty) \to [0, \infty)$, $x \mapsto x^2$, is not neat. 
We collect necessary facts on neat maps in Subsubsection \ref{subsubsec_transverse_category} below. 
As we explain there, neatness guarantees a nice theory on fiber products. 

\begin{defn}[Differential stable relative tangential $G$-cycles]\label{def_diff_rel_Gcyc}
Let $X$ be a manifold and $r$ be an integer. 
A {\it differential stable relative tangential $G$-cycle of dimension $r$} over $X$ is a pair $\widehat{c} = (p \colon N \to X, g_p)$, where $p$ is a proper neat map
with relative dimension $r$ and $g_p$ is a differential stable relative tangential $G$-structure on $p$. 
A {\it representative of a differential stable relative tangential $G$-cycle of dimension $r$} over $X$ is a pair $\widetilde{c} := (p \colon N \to X, \widetilde{g}_p)$, where $\widetilde{g}_p$ is now a representative. 
\end{defn}

\begin{defn}[{Differential stable relative tangential $G$-bordism data}]
Let $X$ be a manifold and $r$ be an integer. 
A {\it differential stable relative tangential $G$-bordism data of dimension $r$} is a differential relative stable tangential $G$-cycle $\widehat{b} = (q \colon W \to \R \times X, g_q)$ of dimension $r$ over $\R \times X$ such that $q$ is transverse (Definition \ref{def_transversality}) to $\{0\} \times X$ and $q^{-1}((-\infty, 0] \times X)$ is compact. 
It defines a differential relative stable tangential $G$-cycle $\del \widehat{b} := \left(q|_{q^{-1}(\{0\} \times X)} , g_q|_{q^{-1}(\{0\} \times X)}\right)$ over $X$ by Proposition \ref{prop_fiber_product}.   
\end{defn}

On the topological level, we have the Chern-Dold homomorphism
\begin{align}\label{eq_ch_MTG}
    \mathrm{ch} \colon (\Omega^G)^{-r}(X) \to H^{-r}(X; V_{\Omega^G}^\bullet). 
\end{align}
A differential relative stable tangential $G$-cycle $\widehat{c}$ over $X$ represents a class $[\widehat{c}] \in (\Omega^G)^{-r}(X)$, so we get a class $\mathrm{ch}([\widehat{c}]) \in H^{-r}(X; V_{\Omega^G}^\bullet) $. 

On the differential level, the homomorphism \eqref{eq_ch_MTG} is refined as follows. 
Applying \eqref{eq_ch_MTG} to the identity element $\mathrm{id}_{MTG} \in(\Omega^G)^{0}(MTG)$, we have an element $\mathrm{ch}(\mathrm{id}_{MTG}) \in H^0(MTG; V_{\Omega^G}^\bullet)$. 
Given a differential relative stable tangential $G$-cycle $\widehat{c} = (p \colon N \to X, g_p)$ over $X$,
by the Chern-Weil construction in Remark \ref{rem_cw_coefficient} with coefficient $\mathcal{V}^* = V_{\Omega^G}^*$, we get\footnote{
The element \eqref{eq_cw_gp_1} can be understood as follows. 
It induces a degree-preserving $\R$-linear homomorphism from $N_G^\bullet$ to $\Omega^\bullet(N;  \Ori(T(p)))$. 
This homomorphism coincides with $\cw_{g_p}$ in \eqref{eq_def_cw_hom}. 
}
\begin{align}\label{eq_cw_gp_1}
    \cw_{g_p}(\mathrm{ch}(\mathrm{id}_{MTG})) \in \Omega^0(N;  \Ori(T(p))\otimes_\R V_{\Omega^G}^\bullet) . 
\end{align}
Here we abuse the notation to write $\Ori(T(p)) := \Ori(T(\phi, p))$ for any choice of a representative of $g_p$, since the orientation bundles for stable relative tangent bundles do not depend on the choice of $\phi$.

We would like to integrate it along the fiber of $p \colon N \to X$. 
Note that unless $p$ is a submersion, the resulting form on $X$ is singular, so we need to deal with differentiable currents $\Omega_{-\infty}^*$. 
Recall that in general the fiber integration of differentiable currents along the map $p \colon N \to X$ is the following. 
\begin{align}\label{eq_current_integration}
    p_! \colon \Omega_{-\infty}^*(N; \Ori(T(p))) \to \Omega_{-\infty}^{*-r}(X). 
\end{align}
Here we choose the sign of the integration \eqref{eq_current_integration} so that it is compatible with the {\it left} $\Omega^\bullet(X)$-module structure, i.e., 
\begin{align}\label{eq_current_integration_sign}
   \eta \wedge p_!(\omega) = p_! (p^*\eta \wedge \omega)
\end{align}
for any $\eta \in \Omega^\bullet(X)$ and $\Omega_{-\infty}^*(N; \Ori(T(p)))$. 
Remark that it is different from the sign convention on the $S^1$-integration in \eqref{eq_int_form} and \eqref{eq_int_form_sign}. 
We set
\begin{align}\label{def_T(hatc)}
    T(\widehat{c}) := p_! (\cw_{g_p}(\mathrm{ch}(\mathrm{id}_{MTG}))) \in \Omega_{-\infty}^{-r}(X; V_{\Omega^G}^\bullet). 
\end{align}
This current represents the class $\mathrm{ch}([\widehat{c}])$ under the isomorphism between the de Rham and the currential cohomologies.

\begin{defn}[{$(\widehat{\Omega^G})^{*}$-cycles}]\label{def_hat_MTG_cycle}
Let $X$ be a manifold and $r$ be an integer. 
An {\it $(\widehat{\Omega^G})^{-r}$-cycle} over $X$ is a pair $(\widehat{c}, \alpha)$, where $\widehat{c}$ is an $r$-dimensional differential relative stable tangential $G$-cycle over $X$ and $\alpha \in \Omega_{-\infty}^{-r-1}(X; V_{\Omega^G}^\bullet)$ such that
\begin{align}\label{eq_R(calpha)}
    R(\widehat{c}, \alpha):=T(\widehat{c}) - d\alpha \in \Omega^{-r}(X; V_{\Omega^G}^\bullet). 
\end{align}
\end{defn}
The role of $\alpha$ in Definition \ref{def_hat_MTG_cycle} is to replace $T(\widehat{c})$ with a smooth differential form, without changing the cohomology class. 

The set of isomorphism classes of $(\widehat{\Omega^G})^{-r}$-cycles over $X$ is denoted by $(Z\widehat{\Omega^G})^{-r}(X)$. 
It has a structure of an abelian semigroup by the disjoint union on cycles and the addition on currents.  

For an $r$-dimensional bordism data $\widehat{b} = (q \colon W \to \R \times X, g_q)$, we define 
\begin{align}\label{eq_T(hatb)}
T(\widehat{b}) :=\int_{(-\infty, 0]} q_! \left(\cw_{g_q}(\mathrm{ch}(\mathrm{id}_{MTG}))|_{q^{-1}((-\infty, 0] \times X)}\right) \in \Omega_{-\infty}^{-r-1}(X; V_{\Omega^G}^\bullet). 
\end{align}
Then we can show that $(\del \widehat{b}, T(\widehat{b})) \in (Z\widehat{\Omega^G})^{-r}(X)$. 

\begin{defn}[{$(\widehat{\Omega^G})^*(X)$}]\label{def_hat_OmegaG}
Let $X$ be a manifold and $r$ be an integer. 
On $(Z\widehat{\Omega^G})^{-r}(X)$ we introduce the equivalence relation $\sim$ generated by $(\del \widehat{b}, T(\widehat{b})) \sim 0$ for a bordism data $\widehat{b}$. 
We define
\begin{align*}
    (\widehat{\Omega^G})^{-r}(X):=(Z\widehat{\Omega^G})^{-r}(X)/\sim. 
\end{align*}
We denote the class of $(\widehat{c}, \alpha)$ in $(\widehat{\Omega^G})^{-r}(X)$ by $[\widehat{c}, \alpha]$. 
\end{defn}

We can define the structure maps $R$, $a$ and $I$ for $(\widehat{\Omega^G})^*(X)$ in an analogous way to \cite{BunkeSchickSchroderMU}. 
The fact that the quadruple $(\widehat{\Omega^G}, R, a, I)$ is a differential extension of $\Omega^G$ can be easily checked as in the normal case. 

\subsection{The differential multiplication by $(\widehat{\Omega^{G_2}})^*$}
Now assume we are given a homomorphism $\mu \colon G_1 \times G_2 \to G_3$. 
The definition of the transformation \eqref{eq_module_intro} uses the fiber products between differential relative stable tantgential $G_2$-cycles and differential stable tangential $G_1$-cycles. 
To form the fiber products, we want to restrict our attention to differential stable tangential $G_1$-cycles $(M, g, f)$ with $f$ satisfying certain transversality conditions. 
For this, the result of Subsubsection \ref{sec_physical_app} is useful. 
In Subsubsection \ref{subsubsec_transverse_category} below, we construct a certain equivalent subcategory of $\hBordone_{n-r}(X, Y)$ consisting of objects with a suitable transversality. 
This point is technical, and the reader who is willing to admit the existence of a nice subcategory to form fiber products can go directly to Subsubsection \ref{subsubsec_module}. 

\subsubsection{A technical point : The construction of $\hBordone_{n-r}(X, Y)_{\pitchfork \widehat{c}}$}\label{subsubsec_transverse_category}
There are substantial technicalities concerning fiber products between manifolds with corners. 
For example see \cite{JoyceCorners}\footnote{
Be careful that ``smooth'' in this paper corresponds to ``weakly smooth'' in \cite{JoyceCorners}. 
The neatness in Definition \ref{def_neat_map} implies the ``smoothness'' in that paper. 
Using this, it is also possible to obtain the results in this subsubsection by applying the results in \cite[Section 6]{JoyceCorners}.
But since neat maps can be treated in an elementary way, we take a direct approach here. 
}. 
But recall that we required the neatness (Definition \ref{def_neat_map}) of the map in Definition \ref{def_diff_rel_Gcyc}. 
As we explain now, the theory on fiber products is very simple for neat maps. 

First we explain a useful local picture of neat maps. 
The following lemma directly follows by Definition \ref{def_neat_map}. 
\begin{lem}\label{lem_local_extension}
Let $p \colon N \to X$ be a neat map between manifolds. 
Let $x \in N$ be any point. 
Then there exist open neighborhoods $x \in V \subset N$ and $p(x) \in U \subset X$,
manifolds without boundaries $\widehat{V}$ and $\widehat{U}$ with embeddings $V \hookrightarrow \widehat{V}$ and $U \hookrightarrow \widehat{U}$, 
a smooth map $\widehat{p} \colon \widehat{V} \to \widehat{U}$ extending $p|_V$ such that
\begin{itemize}
    \item $\widehat{p}^{-1}(U) = V$, and
    \item $\widehat{p}$ is transverse to corners of $U$. 
\end{itemize}
Conversely, neatness is characterized by this local property. 
\end{lem}
In the following we call a set of data appearing in Lemma \ref{lem_local_extension} a {\it local extension} of $p$. 
If we have a smooth map $f \colon M \to U$ from another manifold, the property $\widehat{p}^{-1}(U) = V$ implies that
\begin{align}\label{eq_fiber_product_local}
    M \times_U V \simeq M \times_{\widehat{U}}\widehat{V}
\end{align}
as a topological space. 
Using \eqref{eq_fiber_product_local}, we can reduce the local theory on fiber products between a smooth map $f \colon M \to X$ and a neat map $p \colon N \to X$ to the case where $N$ and $X$ are boundaryless. 

Now we introduce the transversality condition. 
\begin{defn}[{Transversality between smooth maps}]\label{def_transversality}
Let $f \colon M \to X$ and $p \colon N \to X$ be smooth maps. 
We say that $f$ is {\it transverse} to $p$ or $f$ and $p$ are {\it transverse},  if for any points $x \in M$ and $y \in N$ with $f(x) = p(y) = z$, we have
\begin{enumerate}
       \item The images of $df \colon T_x M \to T_{z}X$ and $dp \colon T_yN \to T_{z} X$ span $T_zX$, and
   \item The images of $df\colon T_x S^k(M) \to T_{z}S^j(X)$ and $dp \colon T_yS^l(N) \to T_{z} S^j(X)$ span $T_z S^j(X)$. 
       Here $k$, $j$ and $l$ are the depths of the corresponding points. 
       \end{enumerate}
\end{defn}
If $p$ is neat, the condition (2) in Definition \ref{def_transversality} is equivalent to the condition that the images of $df\colon T_x S^k(M) \to T_{z}X$ and $dp \colon T_y N \to T_{z} X$ span $T_z X$. 
This implies that, in a local extension in Lemma \ref{lem_local_extension}, the transversality of $f$ and $p$ is equivalent to the transversality of $f$ and $\widehat{p}$. 

\begin{prop}\label{prop_fiber_product}
Let $f \colon M \to X$ be a smooth map and $p \colon N \to X$ be a neat map. 
Assume $f$ is transverse to $p$, and we assume one of the following condition for $f$. 
\begin{enumerate}
    \item $f$ is an embedding. 
    \item $M$ is equipped with a structure of $\langle k \rangle$-manifold with $(\del_0 M, \del_{1} M, \cdots, \del_{k-1} M)$, and there exists a collar structure near each $\del_j M$ on which $f$ is constant in the collar direction. 
\end{enumerate}
Then in the fiber product
\begin{align}\label{diag_fiber_product}
    \xymatrix{
     M \times_X N \ar[r]^-{\widetilde{f}} \ar[d]^-{\widetilde{p}}& N \ar[d]^-{p}\\
    M \ar[r]^-{f} &X,
    }
\end{align}
the space $M \times_X N$ is equipped with a canonical structure of manifold with corners so that
\begin{align*}
    S^k( M \times_X N) = S^k(M) \times_X N, 
\end{align*}
and the map $\widetilde{p}$ is neat. 
Moreover in the case (2) above, $M \times_X N$ is also equipped with a canonical structure of $\langle k \rangle$-manifold with $\del_j(M \times_X N) = \del_jM \times_X N$. 
\end{prop}
\begin{proof}
We can reduce to the case where $N$ and $X$ are boundaryless as explained above, and for such cases this result is well-known and easy to prove. 
We also remark that a part of this proposition follows from \cite[Appendix C.22]{HopkinsSinger2005}. 
\end{proof}

Now we define the subcategory $\hBordone_{m}(X, Y)_{\pitchfork \widehat{c}}$ of $\hBordone_{m}(X, Y)$. 

\begin{defn}[{$\hBordone_{m}(X, Y)_{\pitchfork \widehat{c}}$}]\label{def_transverse_subcategory}
Let $(X, Y)$ be a pair of manifolds and $\widehat{c} = (p \colon N \to X, g_p)$ be a differential stable relative tangential $G_2$-cycle over $X$. 
We define a Picard subcategory $\hBordone_{m}(X, Y)_{\pitchfork \widehat{c}}$ of $\hBordone_{m}(X, Y)$ spanned by objects $(M, g, f)$ with the following conditions. 
\begin{enumerate}
    \item There exists a collar structure near $\del M$ of $M$ on which the restriction of the map $f \colon (M, \del M) \to (X, Y)$ is constant in the collar direction.
    \item The map $f \colon M \to X$ is transverse to $p \colon N \to X$. 
\end{enumerate}
\end{defn}

In particular if $p$ is a submersion, we have $\hBordone_{m}(X, Y)_{\pitchfork \widehat{c}}=\hBordone_{m}(X, Y)$. 
Proposition \ref{prop_fiber_product} implies the following. 
\begin{cor}\label{cor_fiber_product_corner}
For any object $(M, g, f) \in \hBordone_{m}(X, Y)_{\pitchfork \widehat{c}}$ with $\widehat{c} = (p \colon N \to X, g_p)$, the fiber product $M \times_X N$ is equipped with a structure of a $\langle 1 \rangle$-manifold with $\del (M \times_X N) =\del M \times_X N$, with map $\tilde{f} \colon (M \times_X N, \del (M \times_X N)) \to (N, p^{-1}(Y))$. 
\end{cor}

Given a pair of manifolds $(X, Y)$, 
We say that a differential stable relative tangential $G_2$-cycle $\widehat{c} = (p \colon N \to X, g_p)$ is {\it transverse to $Y$} if the underlying map $p$ is transverse to $Y$.  
Remark that $(\widehat{\Omega^{G_2}})^{-r}(X)$ is generated by $(\widehat{c}, \alpha)$ with $\widehat{c}$ satisfying this transversality. 

\begin{lem}\label{lem_genericity_transversality}
If $\widehat{c}$ is transverse to $Y$, the inclusion $\hBordone_{m}(X, Y)_{\pitchfork \widehat{c}} \subset \hBordone_{m}(X, Y)$ is an equivalence. 
\end{lem}
\begin{proof}
It is enough to show that any elements in the relative bordism group $\Omega^{G_1}_m(X, Y)$ can be represented by an object in $\hBordone_{m}(X, Y)_{\pitchfork \widehat{c}}$. 
Take any element $e\in \Omega^{G_1}_m(X, Y)$. 
First we consider its image of the boundary map, $\del e \in \Omega^{G_1}_{m-1}(Y)$. 
Applying Proposition \ref{prop_fiber_product} case (1) to the embedding $\iota \colon Y \hookrightarrow X$ and $p \colon N \to X$, we get a canonical structure of a manifold with corners on $p^{-1}(Y)$ so that the map $p \colon p^{-1}(Y) \to Y$ is neat. 
We claim that we can represent $\del e$ by a smooth map $f_{\del} \colon M_\del \to Y$ transverse to $p\colon p^{-1}(Y) \to Y$. 
Indeed, using the homotopy equivalence $\mathring{Y} \sim Y$, it is easy to reduce to the genericity of transversality in the case without boundary. 

By the transversality of $p \colon N \to X$ with $\iota\colon Y \hookrightarrow X$, we see that the composition $\iota \circ f_{\del} \colon M_\del \to X$ is transverse to $p \colon N \to X$. 
This also implies that the map $f_\del \circ \mathrm{pr}_{M_\del} \colon (-1, 0] \times M_\del \to X$ is transverse to $p \colon N \to X$, and we use it as a collar of the desired object. 
Now that we have defined the maps on the collar, the desired object in $\hBordone_{m}(X, Y)_{\pitchfork \widehat{c}}$ which represents $f_\del$ can be obtained by reducing to the case where $N$ and $X$ are boundaryless as before, and using the usual genericity of transversality in the relative form. 
\end{proof}

By Lemma \ref{lem_genericity_transversality}, fixing $\widehat{c}$ which is transverse to $Y$, we can apply the machinery of Subsubsection \ref{sec_physical_app} to the equivalent Picard subcategory $\hBordone_{m}(X, Y)_{\pitchfork \widehat{c}}$. 
We use the notation $\mathcal{C}^{(G_1)_\nabla}_{m}(X, Y)_{\pitchfork \widehat{c}} $ for the corresponding subgroup of $\mathcal{C}^{(G_1)_\nabla}_{m}(X, Y)$, which is denoted by $\mathcal{C}_{\mathcal{D}}$ with $\mathcal{D} = \hBordone_{m}(X, Y)_{\pitchfork \widehat{c}}$ in Subsubsection \ref{sec_physical_app}.

\subsubsection{The differential multiplication by $(\widehat{\Omega^{G_2}})^*$}\label{subsubsec_module}

Now we proceed to construct the transformation \eqref{eq_module_intro}. 
Let $(X, Y)$ be a pair of manifolds and let $\widehat{c} = (p \colon N \to X, g_p)$ be a differential stable relative tangential $G_2$-cycle of dimension $r$ which is transverse to $Y$. 
First we define a functor between bordism Picard categories, essentially given by the fiber product over $X$ composed with the homomorphism $\mu$, but technically speaking we need to take a little care regarding stabilizations. 
In order to define a functor, we need the following additional choices (which eventually yields the same homomorphisms between $\mathcal{C}^{G_\nabla}$'s). 
\begin{itemize}
    \item A representative $\widetilde{c} = (p \colon N \to X, \widetilde{g}_p)$ of $c$ with $\widetilde{g}_p = (k, \phi, P_p, \nabla_p, \psi_p)$, and we require that $k$ is even. 
    \item A subbundle $H_p \subset \underline{\R}^k \oplus TN$ over $N$ which induces a splitting $\underline{\R}^k \oplus TN = H_p \oplus  T(\phi, p)$. 
    \item A Riemannian metric $g_X^{\mathrm{met}}$ on $X$. 
\end{itemize}
Then we define a functor
\begin{align}\label{eq_functor_fiber_product}
     \times_X (\widetilde{c}, H_p, g_X^{\mathrm{met}}) \colon \hBordone_{n-r-1}(X, Y)_{\pitchfork \widehat{c}} \to \hBordthree_{n-1}(N, p^{-1}(Y)),
\end{align}
as follows. 
For an object $(M, g_M, f)$ in $\hBordone_{n-r-1}(X, Y)_{\pitchfork \widehat{c}}$, we consider the fiber product $M \times_X N$ and use the notation of maps as in \eqref{diag_fiber_product}. 
Then by Corollary \ref{cor_fiber_product_corner}, $M \times_X N$ is a smooth $\langle 1 \rangle$-manifold with $\del (M \times_X N) = \del M \times_X N$ of dimension $(n-1)$. 
Moreover, the horizontal subbundle $H_p$ and the Riemannian metric $g_X^{\mathrm{met}}$ gives the identification\footnote{
The identification \eqref{eq_splitting} is explicitely given as follows. 
Notice that $TM$ and $TN$ are equipped with Riemannian metrics by the data $g_M$, $g_p$, $g_X^{\mathrm{met}}$ and $H_p$. 
This gives a splitting of vector bundles over $M \times_X N$ of the form
\begin{align}\label{eq_splitting_fiber_product}
    \widetilde{p}^*TM \oplus \widetilde{f}^*TN = T(M \times_X N) \oplus H'. 
\end{align}
The restriction of $df -dp \colon \widetilde{p}^*TM \oplus \widetilde{f}^*TN \to (p \circ \tilde{f})^*TX$ to $H'$ is an isomorphism, 
\begin{align*}
    (df - dp)|_{H'} \colon H' \simeq (p \circ \tilde{f})^*TX. 
\end{align*}
Denote by $h^\perp \colon \underline{\R}^k \oplus TN \to T(\phi, p)$ the projection induced by $H_p$. 
Then the map \eqref{eq_splitting} is given by mapping an element
\begin{align*}
    (v, m, n) \in \underline{\R}^k \oplus T(M \times_X N) \subset \underline{\R}^k \oplus TM \oplus TN
\end{align*}
to the element
\begin{align*}
    \left( m + \mathrm{pr}_{\widetilde{p}^*TM}((df - dp)|_{H'}^{-1}(\phi (v))), \ 
    h^{\perp}\left( v, n + \mathrm{pr}_{\widetilde{f}^*TN}((df - dp)|_{H'}^{-1}(\phi (v)))
    \right)
    \right). 
\end{align*}
We can check that this is indeed an isomorphism. 
Here we note that when $p$ is a submersion, the resulting identification \eqref{eq_splitting} does not depend on the choice of $g_X^{\mathrm{met}}$ because the splitting \eqref{eq_splitting_fiber_product} becomes the obvious one induced by $H_p$. 
\label{footnote_splitting}
}
\begin{align}\label{eq_splitting}
    \underline{\R}^k \oplus T(M \times_X N) \simeq \widetilde{p}^*TM \oplus \widetilde{f}^* T(\phi, p) . 
\end{align}

Take a representative $\widetilde{g}_M = (d, P_M, \nabla_M, \psi_M)$ for $g_M$ so that $d \ge n-r =  \dim M + 1$. 
We use the following isomorphism $\Psi$ defined by the composition, 
\begin{align}\label{eq_def_Psi}
\Psi \colon \underline{\R}^{d-(n-r-1)+k} \oplus T(M \times_X N)
&= \underline{\R}^{d-(n-r)} \oplus \underline{\R}^{k} \oplus \underline{\R} \oplus T(M \times_X N) \notag \\
&\xrightarrow{\mbox{flip}}  \underline{\R}^{d-(n-r)} \oplus \underline{\R} \oplus \underline{\R}^{k}\oplus T(M \times_X N) \\
&\xrightarrow{\eqref{eq_splitting}}  (\underline{\R}^{d-(n-r)+1} \oplus \widetilde{p}^*TM) \oplus \widetilde{f}^* \ker (\phi, f) . \notag
\end{align}
Here the second map flips $\underline{\R}^k$ and $\underline{\R}$
(the ``flip'' is necessary to make a functor). 
Then the representative $\widetilde{g}_M \times_{X} \widetilde{g}_p$ of differential stable tangential $G_3$-structure on $M \times_X N$ is defined to be
\begin{align}\label{eq_fiber_product_representative}
    \widetilde{g}_M \times_{X} \widetilde{g}_p := \left(d+k, \mu_{d, k}\left(\widetilde{p}^*(P_M, \nabla_M) \times \widetilde{f}^* (P_p, \nabla_p)\right), \Psi^{-1} \circ \left(\widetilde{p}^*(\psi_M) \times \widetilde{f}^* (\psi_p) \right)\right). 
\end{align}
Here for the last item we use the obvious isomorphism of the associated bundles given by the commutativity of \eqref{eq_cond1_multi_G}. 
Then by the commutativity of the left square of \eqref{eq_cond2_multi_G} we have $(\widetilde{g}_M \times_{X} \widetilde{g}_p) (1) =  (\widetilde{g}_M(1)) \times_{X} \widetilde{g}_p $. 
Thus, denoting the resulting differential stable tangential $G_3$-structure by $g_M \times_X \widetilde{g}_p$, we can define the following object of $\hBordthree_{n-1}(N, p^{-1}(Y))$. 
\begin{align*}
    (M, g_M, f) \times_X(\widetilde{c}, H_p, g_X^{\mathrm{met}}) := (M \times_X N, g_M \times_X \widetilde{g}_p,  \widetilde{f}). 
\end{align*}
For morphisms in $\hBordone_{n-r-1}(X, Y)_{\pitchfork \widehat{c}}$, note that we can always take representatives whose underlying maps to $X$ are transverse to $p$ and satisfy the condition (2) in Proposition \ref{prop_fiber_product}. 
Then the corresponding morphisms in $\hBordone_{n-1}(N, p^{-1}(Y))$ is defined in a similar way by the fiber product over $X$ using Proposition \ref{prop_fiber_product}, but in this case we do {\it not} insert the ``flip'' in \eqref{eq_def_Psi}. 
Then we can easily check that we get the functor \eqref{eq_functor_fiber_product} as desired.

As a result, we get a group homomorphism between $\mathcal{C}^{G_\nabla}$'s. 
Notice that the homotopy class (Definition \ref{def_diff_Gstr_vec} (4)) of the representative \eqref{eq_fiber_product_representative} does not depend on the choice of $H_p$ or $g_X^{\mathrm{met}}$. 
Moreover, recall that we have assumed that $k$ is even. 
If we use the two-fold stabilization $\widetilde{g}(2) :=(\widetilde{g}(1))(1)$, by the homotopy commutativity of \eqref{eq_cond2_multi_G} we see that $(\widetilde{g}_M \times_{X} \widetilde{g}_p)(2)$ and $\widetilde{g}_M \times_{X} (\widetilde{g}_p(2))$ are homotopic. 
Thus, for a differential stable relative tangential $G_2$-cycle $\widehat{c}$, we get a group homomorphism
\begin{align}\label{eq_hom_fiber_product}
    \times_X \widehat{c} \colon \mathcal{C}^{(G_1)_\nabla}_{n-r-1}(X, Y) _{\pitchfork \widehat{c}} \to \mathcal{C}^{(G_3)_\nabla}_{n-1}(N, p^{-1}(Y)). 
\end{align}
using any choice of $(\widetilde{c}, H_p, g_X^{\mathrm{met}})$ lifting $\widehat{c}$ as above. 
Recall that we have $N_G^\bullet = \Hom(\Omega^G_\bullet(\pt), \R)$ and $V_{\Omega^G}^\bullet = \Omega^G_{-\bullet} (\pt) \otimes \R$. 
Thus the homomorphism $\mu \colon G_1 \times G_2 \to G_3$ induces the homomorphism
\begin{align}\label{eq_module_coeff}
    \mu \colon N_{G_3}^q \otimes V_{\Omega^{G_2}}^{p} \to N_{G_1}^{q-p}
\end{align}
for each $p$ and $q$.

\begin{defn}\label{def_wedge}
Let $Z$ be a manifold. 
We define the linear maps
\begin{align}
    \wedge_\mu &\colon \Omega^{n}(Z; N_{G_3}^\bullet) \otimes   \Omega^{-r}(Z; V_{\Omega^{G_2}}^\bullet) \to \Omega^{n-r}(Z; N_{G_1}^\bullet), \\
   \wedge_\mu &\colon \Omega^{n}(Z; N_{G_3}^\bullet) \otimes   \Omega^{-r}_{-\infty}(Z; V_{\Omega^{G_2}}^\bullet) \to \Omega^{n-r}_{-\infty}(Z; N_{G_1}^\bullet), \notag
\end{align}
by the wedge product on forms and currents, and the homomorphism $\mu$ in \eqref{eq_module_coeff} on the coefficients. 
\end{defn}

By the definition of the functor \eqref{eq_functor_fiber_product}, for any morphism $[W, g_W, f_W]$ in $\hBordone_{n-r-1}(X, Y)_{\pitchfork \widehat{c}}$ and $\omega \in\Omega_{\mathrm{clo}}^{n}(N, p^{-1}(Y); N_{G_3}^\bullet) $, we have
\begin{align}\label{eq_cw_total}
    \cw(\omega) \left([W, g_W, f_W] \times_X (\widetilde{c}, H_p, g_X^{\mathrm{met}})
    \right)
    = \cw(p_! (\omega \wedge_\mu \cw_{g_p}(\mathrm{ch}(\mathrm{id}_{MTG_2}))))([W, g_W, f_W]), 
\end{align}
where $(\widetilde{c}, H_p, g_X^{\mathrm{met}})$ is any choice lifting $\widehat{c}$,  $\cw_{g_p}(\mathrm{ch}(\id_{MTG_2})) \in \Omega^0(N;  \Ori(T(p))\otimes_\R V_{\Omega^{G_2}}^\bullet) $ is defined in \eqref{eq_cw_gp_1} and $p_!$ is the fiber integration of currents \eqref{eq_current_integration}. 
In the right hand side of \eqref{eq_cw_total}, we use the obvious currential version of \eqref{eq_def_cw_omega_mor}. 
We use this generalization throughout the rest of the paper. 

In particular, for $\omega \in\Omega_{\mathrm{clo}}^{n}(X, Y; N_{G_3}^\bullet) $, we have
\begin{align}\label{eq_cw_T(hatc)}
    \cw(\omega) \circ p_* \left([W, g_W, f_W] \times_X (\widetilde{c}, H_p, g_X^{\mathrm{met}})
    \right)
    = \cw(\omega \wedge_\mu  T(\widehat{c}))([W, g_W, f_W]). 
\end{align}
Here we denoted the functor $p_*$ on the bordism Picard categories over $(N, p^{-1}(Y))$ to $(X, Y)$ given by the composition of $p$. 
We will use the same notation for the corresponding homomorphism between $\mathcal{C}^{(G_1)_\nabla}$'s. 

Now we proceed to define the multiplication \eqref{eq_module_intro}. 
First, we define the multiplication of each element $(\widehat{c}, \alpha) \in (Z\widehat{\Omega^{G_2}})^{-r}(X)$ such that the underlying map $p$ is transverse to $Y \subset X$. 
Given such an element, we set $(\widehat{I\Omega^{G_1}_{\pitchfork \widehat{c}}})^{n-r}(X, Y)$ to be the group in Definition \ref{def_sub_model} for $\mathcal{D} = \hBordone_{n-r-1}(X, Y)_{\pitchfork \widehat{c}}$. 
It is isomorphic to $(\widehat{I\Omega^{G_1}_{\mathrm{dR}}})^{n-r}(X, Y)$ by Proposition \ref{prop_sub_model_isom} and Lemma \ref{lem_genericity_transversality}.

\begin{defn}\label{def_module_str_cyclewise}
Let $(X, Y)$ be a pair of manifolds and $r$ be an integer. 
Given $(\widehat{c}, \alpha) \in (Z\widehat{\Omega^{G_2}})^{-r}(X)$ such that $\widehat{c}$ is transverse to $Y$, we define a linear map
\begin{align}\label{eq_module_str_cyclewise}
   \times_X (\widehat{c}, \alpha)  \colon  (\widehat{I\Omega^{G_3}_{\mathrm{dR}}})^n(X, Y) \to (\widehat{I\Omega^{G_1}_{\pitchfork \widehat{c}}})^{n-r}(X, Y), 
\end{align}
by sending $(\omega, h)$ to $\left(\omega \wedge_\mu R([\widehat{c}, \alpha]) , (\widehat{c}, \alpha)_*h\right)$, 
where 
\begin{align}
    (\widehat{c}, \alpha)_*h := h \circ p_* \circ (\times_X \widehat{c})  - \mathrm{cw}(\omega \wedge_\mu \alpha) \colon \mathcal{C}^{(G_1)_\nabla}_{n-r-1}(X, Y)_{\pitchfork \widehat{c}} \to \R/\Z. 
\end{align}
\end{defn}
The compatibility condition for $\left(\omega \wedge_\mu R([\widehat{c}, \alpha]) , (\widehat{c}, \alpha)_*h\right)$ is checked as follows. 
Take a morphism $[W, g_W, f_W] \colon (M_-, g_-, f_-) \to (M_+, g_+, f_+)$ in $\hBordone_{n-r-1}(X, Y)_{\pitchfork \widehat{c}}$. 
We have, choosing any $(\widetilde{c}, H_p, g_X^{\mathrm{met}})$, 
\begin{align*}
    h\circ p_* \left(([M_+, g_+, h_+] - [M_-, g_-. f_-])\times_X \widehat{c}\right)  &= 
    \cw(\omega)\circ p_*\left( [W, g_W, f_W] \times_X (\widetilde{c}, H_p, g_X^{\mathrm{met}})
    \right) \\
    &= \cw(\omega \wedge_\mu  T(\widehat{c}))([W, g_W, f_W])
\end{align*} 
by the compatibility of $(\omega, h)$ and \eqref{eq_cw_T(hatc)}. 
Moreover we have
\begin{align*}
    \cw(\omega \wedge_\mu \alpha)\left( [M_+, g_+, h_+] - [M_-, g_-. f_-]
    \right)
    = \cw (\omega \wedge_\mu (T(\widehat{c}) - R([\widehat{c}, \alpha])))([W, g_W, f_W])
\end{align*}
by \eqref{eq_R(calpha)}. 
Thus we get the desired compatibility. 

\begin{lem}\label{lem_module_str_welldef}
The composition of the map \eqref{eq_module_str_cyclewise} with the isomorphism $(\widehat{I\Omega^{G_1}_{\pitchfork \widehat{c}}})^{n-r}(X, Y) \simeq (\widehat{I\Omega^{G_1}_{\mathrm{dR}}})^{n-r}(X, Y)$ in Proposition \ref{prop_sub_model_isom} only depends on the class $[\widehat{c}, \alpha] \in (\widehat{\Omega^{G_2}})^{-r}(X)$ of $(\widehat{c}, \alpha)$. 
\end{lem}
\begin{proof}
By Definition \ref{def_hat_OmegaG}, it is enough to check that for any $r$-dimensional bordism data $\widehat{b} = (q \colon W \to \R \times X, g_q)$ and any element $[M, g, f] \in \mathcal{C}^{(G_1)_\nabla}_{n-r-1}(X, Y)_{\pitchfork \del \widehat{b}}$, we have
\begin{align}\label{eq_welldef_module}
    h \circ (q|_\del)_*\left([M, g, f] \times_X \del \widehat{b}\right) = \cw \left(\omega \wedge_\mu T( \widehat{b})\right)([M, g, f]) \pmod \Z. 
\end{align}
To check it, take an object $(M, g, f) \in \hBordone_{n-r-1}(X, Y)_{\pitchfork \del \widehat{b}}$ representing $[M, g, f]$ and a data $(\widetilde{b}, H_q, g_{\R \times X}^{\mathrm{met}})$ lifting $\widehat{b}$, where we require $g_{\R \times X}^{\mathrm{met}}$ to be a cylindrical metric induced by some $g_X^{\mathrm{met}}$ on $X$. 
Then we can construct a bordism over $(X, Y)$ (in the sense of Definition \ref{def_bordism_Gcyc}) from $\varnothing$ to $(M, g, f) \times_{X} (\del{\widetilde{b}}, H_q|_{\del}, g_{X}^{\mathrm{met}})$, essentially by ``$(M, g, f) \times_{ X}\left(\mathrm{pr}_X^*((\widetilde{b}, H_q, g_{\R \times X}^{\mathrm{met}})|_{(-\infty, 0] \times X})\right)$''. 
Here the fiber product over $X$ is extended to this case in the obvious way, and we need a suitable deformation to have a collar structure.  
Since any choice are bordant, we abuse the notation and denote the resulting morphism in $\hBordthree_{n-1}(X, Y)$ by $q_* \left[(M, g, f) \times_{ X}\left(\mathrm{pr}_X^*((\widetilde{b}, H_q, g_{\R \times X}^{\mathrm{met}})|_{(-\infty, 0] \times X}) \right)\right]$. 
Now the compatibility condition of $(\omega, h)$ implies that
\begin{align*}
    h \circ (q|_\del)_* \left([M, g, f] \times_X \del \widehat{b}\right) = \cw(\omega)\left(q_*\left[ (M, g, f) \times_{ X}\left(\mathrm{pr}_X^*((\widetilde{b}, H_q, g_{\R \times X}^{\mathrm{met}})|_{(-\infty, 0] \times X}) \right)\right]\right) \\ \pmod \Z. 
\end{align*}
On the other hand, \eqref{def_T(hatc)} and \eqref{eq_T(hatb)} implies that the right hand side is equal to the right hand side of \eqref{eq_welldef_module}. 
This completes the proof. 
\end{proof}

Thus, we can define the following. 

\begin{defn}\label{def_module_str}
Let $(X, Y)$ be a pair of manifolds. 
We define a linear map
\begin{align}\label{eq_def_module_str}
      (\widehat{I\Omega^{G_3}_{\mathrm{dR}}})^n(X, Y) \otimes (\widehat{\Omega^{G_2}})^{-r}(X) \xrightarrow{\cdot} (\widehat{I\Omega^{G_1}_{\mathrm{dR}}})^{n-r}(X, Y). 
\end{align}
by sending $ (\omega, h) \otimes [\widehat{c}, \alpha]  $ to
$(\omega, h) \times_X (\widehat{c}, \alpha) \in (\widehat{I\Omega^{G_1}_{\pitchfork \widehat{c}}})^{n-r}(X, Y) \simeq (\widehat{I\Omega^{G_1}_{\mathrm{dR}}})^{n-r}(X, Y)$. 
This does not depend on the representative $(\widehat{c}, \alpha)$ by Lemma \ref{lem_module_str_welldef}. 
\end{defn}

Finally we show the following. 
\begin{thm}\label{thm_module_str}
The map \eqref{eq_def_module_str} refines the transformation \eqref{module_IOmega} defined by \eqref{eq_module_IE_general}. 
\end{thm}

\begin{proof}
We use the arguments in Subsection \ref{subsec_proof_isom}. 
Recall that for an element $(\omega, h) \in (\widehat{I\Omega^G_\dR})^{N}(X, Y)$ we associated a functor $F_{(\omega, h)}$ in \eqref{eq_associated_functor}. 
By the proof of Theorem \ref{thm_IOmega_dR=IOmega}, the isomorphism $I\Omega^G_\dR \simeq I\Omega^G$ is given by mapping $I((\omega, h))$ to the natural isomorphism class of $F_{(\omega, h)}$.

Take any element $(\widehat{c}, \alpha) \in (Z\widehat{\Omega^{G_2}})^{-r}(X)$ such that $\widehat{c}$ is transverse to $Y$ and $(\omega, h) \in (\widehat{I\Omega^{G_3}_{\mathrm{dR}}})^n(X, Y)$, so that we get $(\omega, h) \cdot [\widehat{c}, \alpha]   \in (\widehat{I\Omega^{G_1}_{\mathrm{dR}}})^{n-r}(X, Y)$. 
Then we claim that the restriction to $\hBordone_{n-r-1}(X, Y)_{\pitchfork \widehat{c}}$ of the associated functor \eqref{eq_associated_functor}, 
\begin{align}\label{eq_proof_module_1}
    F_{ (\omega, h) \cdot [\widehat{c}, \alpha]  } \colon \hBordone_{n-r-1}(X, Y)_{\pitchfork \widehat{c}} 
    \to (\R \to \R/\Z), 
\end{align}
is naturally isormorphic to the composition, 
\begin{align}\label{eq_proof_module_2}
    \hBordone_{n-r-1}(X, Y)_{\pitchfork \widehat{c}} \xrightarrow{p_* \circ (\times_X (\widetilde{c}, H_p, g_X^{\mathrm{met}}))} \hBordthree_{n-1}(X, Y) \xrightarrow{F_{(\omega, h)}}
    (\R \to \R/\Z), 
\end{align}
for any choice of $(\widetilde{c}, H_p, g_X^{\mathrm{met}})$ lifting $\widehat{c}$. 
Indeed, we have a natural isomorphism given by the transformation $F_{\omega \wedge_\mu \alpha}$. 
Here $\omega \wedge_\mu \alpha \in \Omega_{-\infty}^{n-r-1}(X, Y; N_{G_1}^\bullet)$ is now a differential current, and the definition of the natural transformation \eqref{eq_associated_transformation} is extended to currents in the obvious way. 

We use the equivalence of Picard groupoids in Lemma \ref{lem_cat_equivalence}. 
By recalling the Pontryagin-Thom construction, we see that the first arrow in \eqref{eq_proof_module_2} is naturally isomorphic to the induced functor on $\pi_{\le 1}((-)_{-n+r+1})$ to the following compostion of maps of spectra. 
\begin{align}
    (X/Y) \wedge MTG_1 &\xrightarrow{\mathrm{diag} \wedge \id} (X/Y) \wedge X^+ \wedge MTG_1 \\
    &\xrightarrow{\id \wedge [c] \wedge \id} (X/Y) \wedge \Sigma^{-r}MTG_2 \wedge MTG_1 \notag \\
    &\xrightarrow{\id \wedge \mu} (X/Y) \wedge \Sigma^{-r}MTG_3. \notag 
\end{align}
The Anderson dual to this composition is the definition of the multiplication by $[c] \in (\Omega^{G_2})^{-r}(X)$ on the topological level \eqref{module_IOmega} defined by \eqref{eq_module_IE_general}. 
This completes the proof. 
\end{proof}

\section{Differential pushforwards}\label{sec_push}
Let $\mu \colon G_1 \times G_2 \to G_3$ be a homomorphism of tangential structure groups as in Section \ref{sec_module} (Remark \ref{rem_def_multi_G}). 
In this section, we introduce the differential refinement of {\it pushforward maps} associated to $\mu$. 
Here, what we call the pushforwards here is a generalization of pushforward maps (also called {\it integrations}) associated to multiplicative genera, which we explain in Subsection \ref{subsec_push_general}. 
For each pair of manifolds $(X, Y)$ and differential relative stable tangential $G$-cycle (Definition \ref{def_diff_rel_Gcyc}) $\widehat{c}= (p \colon N \to X, g_p)$ of dimension $r$ over $X$ whose underlying map $p$ is a submersion, 
we define a homomorphism which we call the {\it differential pushforward map} along $\widehat{c}$, 
\begin{align}\label{eq_diff_push_IOmega}
    \widehat{c}_* \colon (\widehat{I\Omega^{G_3}_\dR})^n(N, p^{-1}(Y)) \to (\widehat{I\Omega^{G_1}_\dR})^{n-r}(X, Y), 
\end{align}
which refines the {\it topological pushforward map} defined for a topological relative stable tangential $G_2$-cycle $c = (p \colon N \to X, g_p^{\mathrm{top}})$, 
\begin{align}\label{eq_top_push_IOmega}
    c_* \colon (I\Omega^{G_3})^{n}(N, p^{-1}(Y)) \to (I\Omega^{G_1})^{n-r}(X, Y), 
\end{align}
by the general procedure in Subsection \ref{subsec_push_general} applied to the morphism $IMTG_3 \wedge MTG_2 \to IMTG_1$ in \eqref{eq_module_IMTG} associated to $\mu$. 
Actually, the homomorphism \eqref{eq_diff_push_IOmega} is given by a straightforward modification of the multiplication by $\widehat{\Omega^{G_2}}$ in Section \ref{sec_module}. 
As we explain in Remark \ref{rem_integration_push}, we can recover the $S^1$-integration map $\int$ of $\widehat{I\Omega^G_\dR}$ (Definition \ref{def_integration_dR}) as a special case of the construction in this section. 
Also we explain in Remark \ref{rem_push_not_submersion} that if we use the obvious {\it currential} version of $\widehat{I\Omega^G_\dR}$, we can generalize the results in this section to the case where $p$ is not necessarily a submersion. 

\subsection{Generalities on topological pushforwards}\label{subsec_push_general}
Here we introduce the definition of {\it (topological) pushforward maps} which we use in this paper. 
This is a generalization of the most common notion of pushforward maps associated to multiplicative genera. 
We only explain the tangential version, but the normal version also works by just replacing $MTG$ to $MG$. 

The setting is the following. 
Assume we have tangential structure groups $G= \{G_d, s_d, \rho_d\}_{d \in \Z_{\ge 0}}$ and two spectra $E$ and $F$, together with a homomorphism of spectra, 
\begin{align}\label{eq_nu_push}
    \mu \colon E \wedge MTG \to F. 
\end{align}
We claim that \eqref{eq_nu_push} defines, for a topological relative stable $G$-cycle $c = (p \colon N \to X, g_p^{\mathrm{top}})$ of dimension $r$, a homomorphism
\begin{align}\label{eq_def_push_top}
    c_* \colon E^{n}(N, p^{-1}(Y)) \to F^{n-r}(X, Y), 
\end{align}
which we call the {\it (topological) pushforward map} along $c$. 

For simplicity we first explain the absolute case, $Y = \varnothing$. 
Given $c = (p \colon N \to X, g_p^{\mathrm{top}})$ as above, 
choose an embedding $\iota \colon N \hookrightarrow \R^k \times X$ over $X$ (i.e., $\mathrm{pr}_X \circ \iota= p$) for $k$ large enough. 
Choose a tubular neighborhood $U$ of $N$ in $\R^k \times X$ with a vector bundle structure $\pi \colon U \to N$. 
Then $g_p^{\mathrm{top}}$ induces a homotopy class of stable normal $G$-structures $g_{\pi}^{\perp, \mathrm{top}}$ on the vector bundle $\pi \colon U \to N$.   
Thus we get the (universal) Thom element for this vector bundle in $\Omega^{G}$, 
\begin{align*}
    \nu(g_\pi^{\perp, \mathrm{top}}) \in (\Omega^G)^{k-r}(\Thom(\pi \colon U \to N)). 
\end{align*}
The pushforward map \eqref{eq_def_push_top} is defined as the following composition. 
\begin{align}\label{eq_def_diff_push_tangential}
    {E}^n(N) \xrightarrow{\cdot {\nu}(g_\pi^{\perp, \mathrm{top}})} {F}^{n+k-r}(\Thom(\pi \colon U \to N)) \xrightarrow{\iota_*} {F}^{n+k-r}(\Thom(\R^k \times X \to X)) \xrightarrow{\mathrm{desusp}}{F}^{n-r}(X), 
\end{align}
where the first multiplication uses $\mu$, and the middle arrow is associated to the open embedding $\iota \colon U \hookrightarrow \R^k \times X$. 

In the general case $Y \neq \varnothing$, the definition is basically the same, but we need to be careful because $(\R^k \times Y) \cap \pi^{-1}(N \setminus p^{-1}(Y)) \neq \varnothing$ in general.
This subtlety does not arise when $p$ is a submersion, because we can take $\pi$ to be a map over $X$. 
In general, we need to perturb the map $p$ by a homotopy so that $p$ is transverse to the inclusion $Y \hookrightarrow X$. 
Then we can take the choices above so that $(\R^k \times Y) \cap \pi^{-1}(N \setminus p^{-1}(Y)) = \varnothing$. 
Applying this procedure to the morphism $IMTG_3 \wedge MTG_2 \to IMTG_1$ in \eqref{eq_module_IMTG}, we get the topological pushforward maps \eqref{eq_top_push_IOmega}.

An important class of the examples of homomorphisms \eqref{eq_nu_push} comes from {\it multiplicative genera}, i.e., homomorphisms of ring spectra
\begin{align}
    \mathcal{G} \colon MTG \to E, 
\end{align}
for multiplicative $G$ and $E$. 
In this case, $\mathcal{G}$ induces a $MTG$-module structure on $E$, 
\begin{align}
    \mu_{\mathcal{G}} \colon E \wedge MTG \to E. 
\end{align}
Applying the construction in this subsection to $\mu_{\mathcal{G}}$, we recover the usual pushforward in $E$ for tangentially $G$-oriented proper maps.

\subsection{The differential pushforwards}\label{subsec_diff_push_IOmega}
Now let us fix a homomorphism $\mu \colon G_1 \times G_2\to G_3$. 
We construct the differential pushforward maps \eqref{eq_diff_push_IOmega}. 
Actually, we have already prepared the necessary ingredients in Section \ref{sec_module}. 
Recall that, given a differential relative stable tangential $G_2$-cycle $\widehat{c}= (p \colon N \to X, g_p)$ of dimension $r$ such that $p$ is transverse to $Y \subset X$, we defined a group homomorphism \eqref{eq_hom_fiber_product}
\begin{align*}
    \times_X \widehat{c} \colon \mathcal{C}^{(G_1)_\nabla}_{n-r-1}(X, Y)_{\pitchfork \widehat{c}} \to \mathcal{C}^{(G_3)_\nabla}_{n-1}(N, p^{-1}(Y)). 
\end{align*}
If $p$ is a submersion, we have $\mathcal{C}^{(G_1)_\nabla}_{n-r-1}(X, Y)_{\pitchfork \widehat{c}} = \mathcal{C}^{(G_1)_\nabla}_{n-r-1}(X, Y)$. 
Also recall that we have \eqref{eq_cw_gp_1}
\begin{align*}
    \mathrm{cw}_{g_p}(\mathrm{ch}(\mathrm{id}_{MTG_2})) \in \Omega^0(N;  \Ori(T(p))\otimes_\R V_{\Omega^{G_2}}^\bullet) , 
\end{align*}
Given $\omega \in \Omega^{n}(N, p^{-1}(N); N_{G_3}^\bullet)$, using the wedge product in Definition \ref{def_wedge} and the fiber integration \eqref{eq_current_integration}, we have
\begin{align}\label{eq_IOmega_push_form}
    p_! (\omega\wedge_\mu \cw_{g_p}(\mathrm{ch}(\mathrm{id}_{MTG_2}))) \in \Omega_{-\infty}^{n-r}(X, Y; N_{G_1}^\bullet). 
\end{align}
If $p$ is a submersion, the element \eqref{eq_IOmega_push_form} is in $\Omega^{n-r}(X, Y; N_{G_1}^\bullet)$.

\begin{defn}[{The differential pushforward maps associated to $\mu \colon G_1 \times G_2 \to G_3$}]\label{def_diff_push_IOmega}
Let $(X, Y)$ be a pair of manifolds and $n$ and $r$ be nonnegative integers such that $n\ge r$. 
For each differential relative stable tangential $G_2$-cycle $\widehat{c}= (p \colon N \to X, g_p)$ of dimension $r$ over $X$ whose underlying map $p$ is a submersion, 
we define a homomorphism which we call the {\it differential pushforward map} along $\widehat{c}$, 
\begin{align}\label{eq_def_diff_push_IOmega}
    \widehat{c}_* \colon (\widehat{I\Omega^{G_3}_\dR})^n(N, p^{-1}(Y)) \to (\widehat{I\Omega^{G_1}_\dR})^{n-r}(X, Y), 
\end{align}
by mapping $(\omega, h)$ to $\left(p_! (\omega \wedge_{\mu} \cw_{g_p}(\mathrm{ch}(\mathrm{id}_{MTG_2}))), \widehat{c}_* h\right)$, where
\begin{align}
    \widehat{c}_* h := h \circ (\times_X \widehat{c}) \colon \mathcal{C}_{n-r-1}^{(G_1)_\nabla}(X, Y) \to \R/\Z. 
\end{align}
\end{defn}
The compatibility condition for the pair $\left(p_! (\omega \wedge_{\mu} \cw_{g_p}(\mathrm{ch}(\mathrm{id}_{MTG_2}))), \widehat{c}_* h\right)$ can be checked by the same way as the corresponding compatibility checked for Definition \ref{def_module_str_cyclewise}, 
by using \eqref{eq_cw_total} instead of \eqref{eq_cw_T(hatc)}.

\begin{thm}\label{thm_push}
The differential pushforward map \eqref{eq_def_diff_push_IOmega} refines the topological pushforward map \eqref{eq_top_push_IOmega} defined by \eqref{eq_def_diff_push_tangential}. 
\end{thm}
\begin{proof}
Let us fix an element $(\omega, h) \in (\widehat{I\Omega^{G_3}_\dR})^n(N, p^{-1}(Y))$. 
As we did in the proof of Theorem \ref{thm_module_str}, the proof is given by checking the natural isomorphism class of the functor \eqref{eq_associated_functor} associated to the element $\widehat{c}_*(\omega, h) \in (\widehat{I\Omega^{G_1}_\dR})^{n-r}(X, Y)$, 
\begin{align}\label{eq_proof_push_1}
    F_{\widehat{c}_*(\omega, h)} \colon \hBordone_{n-r-1}(X, Y)
    \to (\R \to \R/\Z), 
\end{align}
coincides,  
under the equivalence in Lemma \ref{lem_cat_equivalence},  with the element in 
\begin{align*}
   \pi_0 \mathrm{Fun}_{\mathrm{Pic}}\left( \pi_{\le 1}(L((X/Y) \wedge MTG_1)_{1-n+r}) ,
     (\R \to \R/\Z)\right),
\end{align*}
specified by the topological pushforward \eqref{eq_top_push_IOmega}. 

Recall that the group homomorphism \eqref{eq_hom_fiber_product} between $\mathcal{C}^{G_\nabla}$'s comes from the functor between the bordism Picard groupoids \eqref{eq_functor_fiber_product},
\begin{align}\label{eq_proof_push_4}
     \times_X (\widetilde{c}, H_p) \colon \hBordone_{n-r-1}(X, Y) \to \hBordthree_{n-1}(N, p^{-1}(Y)),
\end{align}
by choosing additional data of $(\widetilde{c}, H_p)$ lifting $\widehat{c}$ (In the case here $p$ is submersion, so we do not need the choice of a Riemannian metric $g_X^{\mathrm{met}}$ as remarked at the end of Footnote \ref{footnote_splitting}). 
We easily see that the functor \eqref{eq_proof_push_1} coincides with the following composition. 
\begin{align}\label{eq_proof_push_2}
    \hBordone_{n-r-1}(X, Y) \xrightarrow{\times_X (\widetilde{c}, H_p)} \hBordthree_{n-1}(N, p^{-1}(Y)) \xrightarrow{F_{(\omega, h)}}
    (\R \to \R/\Z). 
\end{align}

On the other hand, the topological pushforward map \eqref{eq_top_push_IOmega} is defined as the composition \eqref{eq_def_diff_push_tangential}. 
On the level of spectra, this homomorphism is given by the Anderson dual of the composition
\begin{align}\label{eq_proof_push_3}
   MTG_1 \wedge \Sigma^k (X/Y) &\xrightarrow{\mathrm{id} \wedge \iota} MTG_1 \wedge \left(\mathrm{Thom}(\pi \colon U \to N)/\pi^{-1}(p^{-1}(Y)) \right)\\
   &\to MTG_1 \wedge \Sigma^{k-r} MTG_2 \wedge (N/p^{-1}(Y))  \notag \\
   &\xrightarrow{\mu \wedge \id} \Sigma^{k-r}MTG_3 \wedge (N/p^{-1}(Y)) , \notag
\end{align}
where the second morphism is the classifying map of the stable normal $G$-structure $g_\pi^{\perp, \mathrm{top}}$ on $\pi \colon U \to N$ and the identity on $MTG_1$. 
Recalling the Pontryagin-Thom construction, we see that the functor between the fundamental Picard groupoids
\begin{align*}
    \pi_{\le 1}(L((X/Y) \wedge MTG_1)_{1-n+r}) \to \pi_{\le 1}(L((N/p^{-1}(N)) \wedge MTG_3)_{1-n})
\end{align*}
induced by \eqref{eq_proof_push_3} is naturally isomorphic to the fiber product functor \eqref{eq_proof_push_4} under the equivalences in Lemma \ref{lem_cat_equivalence}. 
This completes the proof. 
\end{proof}

\begin{rem}\label{rem_integration_push}
We can recover the $S^1$-integration map $\int$ of $\widehat{I\Omega^G_\dR}$ (Definition \ref{def_integration_dR}) as a special case of the construction in this section. 
As explained in Example \ref{ex_G123} (3), for any $G$ we have a canonical homomorphism $G \times \fr \to G$. 
Let $X$ be a manifold. 
The bounding (differential) stable tangential fr-structure $g_{S^1}$ on $S^1 = S^1 \times \pt$ in Definition \ref{def_Gstr_S1} induces the differential stable relative tangential fr-structure on $\mathrm{pr}_X \colon X \times S^1 \to X$. 
We easily see that the differential pushforward along this differential stable relative fr-cycle coincides with the $S^1$ integration map, but up to sign. 
The difference of the sign is due to the fact that the suspension multiplies $S^1$ from the left. 
This difference also appears in the difference between the signs for the $S^1$-integration on forms in \eqref{eq_int_form}, \eqref{eq_int_form_sign} and the fiber integration in \eqref{eq_current_integration}, \eqref{eq_current_integration_sign}. 
\end{rem}

\begin{rem}\label{rem_push_not_submersion}
We have only defined differential pushforwards for proper {\it submersions}. 
This requirement is necessary for the element \eqref{eq_IOmega_push_form} to be a differential form. 
Actually, we can get refinements of pushforwards along general differential stable relative cycles if we introduce {\it currential} refinement of $I\Omega^G$'s. 
In general, currential refinements of cohomology theories are axiomatized by simply replacing forms to currents in Definition \ref{def_diffcoh}. 
Such refinements are used for example in \cite{FL2010} in the case of $K$-theory. 
In our case, it is obvious that we obtain the currential refinement of $I\Omega^G$, which we denote by $\widehat{I\Omega^G_{-\infty}}$, by just allowing $\omega$ in $(\omega, h)$ to be closed currents. 

Then, given any differential stable relative $G_2$-cycle $\widehat{c}= (p \colon N \to X, g_p)$ of dimension $r$ over $X$ such that $p$ is transverse to $Y \subset X$, by the same construction we get the refinement of the pushforward map, 
\begin{align}
    \widehat{c}_* \colon (\widehat{I\Omega^{G_3}_\dR})^n(N, p^{-1}(Y)) \to (\widehat{I\Omega^{G_1}_{-\infty}})^{n-r}(X, Y). 
\end{align}
We remark that in the construction we use the obvious currential version of Subsection \ref{sec_physical_app}. 
\end{rem}

\section{Relation to physics}\label{sec:physics}

In this section, we explain the basic backgrounds and motivations of our results in the context of physics.
This section is mainly aimed at physicists, and the statements will not be mathematically precise.

In the following discussion, whenever we say ``manifolds'',
they are always supposed to be equipped with some differential structure
such as Riemannian metric, bundles and their connections, and so on.
What differential structure we consider should be specified in advance. 
Another remark is that whenever we say ``invertible field theory'' in this subsection,
we only consider non-extended versions of QFT's unless otherwise stated, as opposed to fully extended versions of QFT's as in \cite{Freed:2016rqq}.
In other words, we only consider Hilbert spaces, amplitudes, and partition functions as explained below. 

\subsection{Some backgrounds on QFT and TQFT}
Very roughly speaking, a $D$-dimensional QFT (which is not extended) is a
symmetric monoidal functor from some geometric bordism category to the (super)vector space category as follows.
A QFT assigns a Hilbert space of physical states ${\mathcal H}(N)$ to each $(D-1)$-dimensional closed manifold $N$. 
In particular, we assume that for the empty manifold $N=\varnothing$, we have a canonical isomorphism ${\mathcal H}(\varnothing) \simeq\C$,
and for disjoint unions $N_1 \sqcup N_2$, we have ${\mathcal H}(N_1 \sqcup N_2) \simeq {\mathcal H}(N_1) \otimes {\mathcal H}(N_2)$.
It assigns a linear map $Z(M) : {\mathcal H}(N_1) \to {\mathcal H}(N_2) $
to each $D$-dimensional compact manifold $M$ with boundaries $\partial M =  \overline N_1 \sqcup N_2$ where
$\overline N_1$ is a manifold which has the opposite structure to that of $N_1$ (such as orientation reversal),
and we have assumed that $M$ has appropriate collar structure near the boundaries,
$[0,\epsilon) \times N_1$ and $(-\epsilon,0] \times N_2$ for some $\epsilon>0$.
We do not try to make these axioms precise,
but we remark that they are motivated by (Euclidean) path integrals in physics.

An invertible field theory is a QFT in which the Hilbert space of states ${\mathcal H}(N)$ on any closed manifold $N$
is one-dimensional, $\dim {\mathcal H}(N) =1$.
Invertible field theories play crucial roles in the study of anomalies. (See e.g. \cite{Freed:2014iua,Monnier:2019ytc} for overviews.)
In fact, the classification of deformation classes of invertible QFT's in $D$-dimensions is believed to be the same as the classification of 
anomalies in $(D-1)$-dimensions.\footnote{
We neglect anomalies which do not fit into the general framework, such as Weyl anomalies. 
Also, there may be subtleties in reflection non-positive theories~\cite{Chang:2020aww}. }
(We will explain what we mean by ``deformation classes'' in a little more detail later.)
In the context of condensed matter physics, deformation classes of invertible field theories are also called
invertible phases of matter, and they are the low energy limit of symmetry protected topological (SPT) phases.\footnote{
We warn the reader that the meaning of the terminology ``SPT phases'' depends on contexts and authors.
} 
Anomalous $(D-1)$-dimensional theories appear on the boundaries of these invertible phases
and have various applications in physics. For instance, see \cite{Witten:2015aba}, references therein, and papers citing it, 
for details about fermions which are important both for condensed matter physics and superstring theories.
Therefore, it is an important problem to classify invertible phases.

In the case of topological QFT (TQFT), the classification of invertible phases
has been conjectured to be given by certain cobordism groups~\cite{Kapustin:2014tfa,Kapustin:2014dxa},
and later proved at least for some physically motivated classes of structure types.
It is proved under the axioms of fully extended TQFT with a choice of the target category in~\cite{Freed:2016rqq}, 
and the 1-categorical version of Atiyah-Segal axioms in \cite{Yonekura:2018ufj}. 
Let ${\mathcal S}$ be the structure type under consideration.
For instance, we can consider manifolds equipped with Spin structures, and in that case we denote ${\mathcal S} = {\rm Spin}$. 
The results of this paper are valid for the case that ${\mathcal S}$ is given by a tangential structure
stated in Definition~\ref{def_phys_Gstr}. However, there are structure types not covered in Definition~\ref{def_phys_Gstr},
such as string structure and its generalizations which appear in heterotic string theories. Thus we use ${\mathcal S} $
in this section because we expect that similar results are also true for those other structure types.

Then we may define a bordism group $\Omega^{\mathcal S}_D(\pt)$ of $D$-dimensional manifolds equipped with 
structure of the type ${\mathcal S}$ roughly as follows.
We introduce a monoid structure on the set of (isomorphism classes of) manifolds by disjoint union,
$M_1 \sqcup M_2$. The empty manifold $\varnothing$ is the unit of this monoid since $M \sqcup \varnothing \simeq M$.
Then we divide this monoid by an equivalence relation.
If a closed $D$-manifold $M$ is a boundary of some compact $(D+1)$-manifold $W$, $M = \partial W$,
then it is defined to be equivalent to the empty set, $M \sim \varnothing$. 
Let $\overline{M}$ denote the orientation reversal 
(or its generalization explained in \cite{Freed:2016rqq,Yonekura:2018ufj}) of $M$.
By using the fact that $W = [0,1] \times M$ has the boundary $\partial W = M \sqcup \overline{M}$,
one can see that we get a group whose elements are represented in terms of manifolds $M$ as $[M]$.
In particular, the inverse of $[M]$ is $[\overline M]$. This group is denoted as $\Omega^{\mathcal S}_D(\pt)$.

According to \cite{Kapustin:2014tfa,Kapustin:2014dxa,Freed:2016rqq,Yonekura:2018ufj},
deformation classes of invertible TQFT's are classified by the group $\Hom(( \Omega^{\mathcal S}_D(\pt))_{\rm tor}, \R/\Z)$,
where the subscript ${\rm tor}$ means to take the torsion part of the group.
The reason that we take the torsion part is that we are considering deformation classes.
To explain this point, let us first consider the group $\Hom(  \Omega^{\mathcal S}_D(\pt), \R/\Z)$.
Then the relation between this group and the above axioms of QFT is the following.
If we are given an element $h \in \Hom( \Omega^{\mathcal S}_D(\pt) , \R/\Z)$,
it means that we can assign to each closed $D$-dimensional manifold $M$ a number
\beq\label{eq:PF}
Z(M) = \exp\left( 2\pi \sqrt{-1} h([M]) \right),
\eeq
where $[M] \in  \Omega^{\mathcal S}_D(\pt)$ is the bordism class represented by $M$.
Notice that for a closed manifold $M$, 
a QFT should assign a linear map $Z(M) : {\mathcal H}(\varnothing) \to {\mathcal H}(\varnothing) $.
Since  $ {\mathcal H}(\varnothing) \simeq\C$, the quantity $Z(M)$ can be regarded just as a number $Z(M) \in \C$.
The function which assigns a number $Z(M) \in \C$ to each closed manifold $M$ is called a partition function in physics. 
From an element $h \in \Hom( \Omega^{\mathcal S}_D(\pt) , \R/\Z)$, we can construct a partition function by \eqref{eq:PF}. 

A partition function itself does not give full data for the axioms of QFT.
However, the theorems proved in \cite{Freed:2016rqq,Yonekura:2018ufj}
imply that we can construct a TQFT from a given $h \in \Hom( \Omega^{\mathcal S}_D(\pt) , \R/\Z)$.
(See Theorem~4.3 of \cite{Yonekura:2018ufj} for explicit construction in the case of Atiyah-Segal axioms.)
Conversely, the partition function of any invertible TQFT can be deformed continuously to a partition function
given by some $h \in \Hom( \Omega^{\mathcal S}_D(\pt) , \R/\Z)$. 
Among the elements of $\Hom( \Omega^{\mathcal S}_D(\pt) , \R/\Z)$, the ones which come from 
$\Hom( \Omega^{\mathcal S}_D(\pt) , \R)$ can be deformed continuously.
In this way, we arrive at the classification of
deformation classes of invertible TQFT's by $\Hom(( \Omega^{\mathcal S}_D(\pt))_{\rm tor}, \R/\Z)$.

\subsection{The exact sequence and classification of invertible phases}
How about the cases which are not necessarily topological?
Freed and Hopkins have conjectured a classification of fully extended invertible QFT's~\cite{Freed:2016rqq}
in terms of the Anderson dual of bordism groups as stated in Conjecture~\ref{conj_intro}. 
The authors of the present paper expect that the classification of invertible QFT's 
which are not extended is also given by the same group.
Before going to discuss it, let us first explain some additional background.

Suppose we are given a manifold $X$. 
Then we can consider a new structure given as follows.
In addition to the differential structure already present in a manifold $M$,
we consider an additional datum $f : M \to X$ which is a map from $M$ to $X$.
In the context of invertible phases and anomalies in physics, we can consider various types of $X$.
If the manifold $X$ is taken to be the target space of a sigma model,
it is relevant to sigma model anomalies~\cite{Moore:1984dc,Moore:1984ws,Thorngren:2017vzn}.
On the other hand, if $X$ is taken to be the space of coupling constants,
it is relevant to more subtle anomalies discussed in 
e.g. \cite{Tachikawa:2017aux,Seiberg:2018ntt,Schwimmer:2019efk,Lawrie:2018jut,Cordova:2019jnf,Cordova:2019uob,Hsin:2020cgg}.
Let us denote the new structure type as $({\mathcal S},X)$, where ${\mathcal S}$ is the original one already considered on $M$. 
Then we denote $\Omega_D^{{\mathcal S}}(X):= \Omega_D^{({\mathcal S},X)}(\pt) $.
For appropriate structure types ${\mathcal S}$, it is known that $\Omega_*^{{\mathcal S}}$
gives a generalized homology theory, and $\Omega_*^{{\mathcal S}}(X)$ are the generalized homology groups of $X$.

Given a generalized homology theory $E_*$, we have the Anderson dual cohomology theory $IE^*$ satisfying the exact sequence \eqref{eq_exact_DE_abst}.
This exact sequence is analogous to the one in the ordinary cohomology theory
associated to the short exact sequence of coefficient groups $0 \to \Z \to \R \to \R/\Z \to 0$. 

The conjecture mentioned above, with a generalization including sigma models and more general types of ${\mathcal S}$, is as follows;
deformation classes of invertible field theories (extended or not) with the structure type $({\mathcal S},X)$ is given by 
$(I\Omega^{\mathcal S})^{D+1}(X)$,
where $(I\Omega^{\mathcal S})^*$ is the Anderson dual of $\Omega_*^{{\mathcal S}}$.
It fits into the exact sequence 
\begin{multline}
    \cdots \to \mathrm{Hom}(\Omega^{\mathcal S}_{D}(X), \R) \to \Hom(\Omega^{\mathcal S}_{D}(X), \R/\Z) \to(I\Omega^{\mathcal S})^{D+1}(X) \\  
    \to \mathrm{Hom}(\Omega^{\mathcal S}_{D+1}(X), \R) \to \Hom(\Omega^{\mathcal S}_{D+1}(X), \R/\Z) \to \cdots .  \nonumber 
\end{multline}
Let us explain physical reasons to believe that this conjecture is reasonable, following \cite{Lee:2020ojw}.
(See also \cite{Davighi:2020uab} for some applications.)

First, let us consider elements of $(I\Omega^{\mathcal S})^{D+1}(X)$ which are in the kernel of 
the map $(I\Omega^{\mathcal S})^{D+1}(X) \to \Hom( \Omega^{\mathcal S}_{D+1}(X), \R) $.
We denote the homomorphism $ \Hom( \Omega_{D}^{\mathcal S}(X), \R) \to \Hom( \Omega_{D}^{\mathcal S}(X), \R/\Z)$ as $p$.
By the exact sequence, the kernel is isomorphic to
\beq
\Hom(  \Omega_{D}^{\mathcal S}(X), \R/\Z)/  \mathrm{Im}(p) \simeq\Hom((\Omega^{\mathcal S}_D(\pt))_{\rm tor}, \R/\Z).
\eeq
This is what we have discussed before in the case of TQFT.
Any element of $\Hom(  \Omega_{D}^{\mathcal S}(X), \R/\Z)$ gives a TQFT.
The division by the image of the map 
$p :  \Hom( \Omega_{D}^{\mathcal S}(X), \R) \to \Hom( \Omega_{D}^{\mathcal S}(X), \R/\Z)$
is due to the fact that we are considering deformation classes. The group $ \Hom( \Omega_{D}^{\mathcal S}(X), \R)$
is a vector space over $\R$, and any two elements of this group can be continuously deformed into one another.
Therefore, we should divide $\Hom(  \Omega_{D}^{\mathcal S}(X), \R/\Z)$ by $\mathrm{Im}(p)$ when we consider
deformation classes of TQFT's. 

Next, let us consider the physical meaning of the map 
$(I\Omega^{\mathcal S})^{D+1}(X)  \to \Hom( \Omega^{\mathcal S}_{D+1}(X), \R) $.
In physics, we may expect the following property of invertible QFT.
Suppose that a closed $D$-manifold $M$ is the boundary of a $(D+1)$-manifold $W$
with a collar structure $(-\epsilon, 0] \times M \subset W$ near the boundary, including geometric data
such that $(-\epsilon, 0]$ has the trivial differential structure. 
We expect to have a closed differential $(D+1)$-form $I_{D+1}$ on $W$ (which is sometimes called an anomaly polynomial in the context of anomalies).
It is constructed from geometric data on $W$. For example, $W$ may have connections of some bundles from which
we can construct characteristic forms. Also, $W$ is equipped with a map $f_W : W \to X$
and hence we can pullback differential forms from $X$ to $W$ by using $f_W$.
The closed form $I_{D+1}$ is constructed by using such differential forms, and it is given by $\mathrm{cw}_{g_W}(f_W^*\omega)$ 
which appeared in \eqref{eq_chernweil_intro}
in the case of ${\mathcal S}=G$.
Then, physicists may expect that the
partition function of an invertible QFT evaluated on $M=\partial W$ is given (after some continuous deformation of the theory\footnote{
In generic theories, there can be nonuniversal terms such as the cosmological constant of the background Riemannian metric
and the Euler number term.
We need to eliminate them by continuous deformation of the theory for the following claim to be valid.
This deformation can be done by a procedure similar to \eqref{eq:oneparameter} below.}) 
by
\beq\label{eq:PT2}
Z(\partial W) = \exp\left( 2\pi \sqrt{-1} \int_W I_{D+1} \right).
\eeq
When $I_{D+1}=0$, this equation is precisely the bordism invariance of the partition function as implied by \eqref{eq:PF},
since $[\partial W] = [\varnothing]$ by the definition of bordism groups.

Now, by using $I_{D+1}$, we can define an element of $\Hom( \Omega^{\mathcal S}_{D+1}(X), \R) $
by 
\beq \label{eq:(D+1)map}
\Omega^{\mathcal S}_{D+1}(X) \ni [W] \mapsto \int_W I_{D+1} \in \R
\eeq
for any closed $(D+1)$-manifold $W$.
Its well-definedness (i.e. it only depends on the equivalence class $[W]$ rather than a representative $W$)
is immediate from the Stokes theorem and $dI_{D+1}=0$. 
Moreover, for the partition function \eqref{eq:PT2} to be well-defined,
the integral $\int_W I_{D+1}$ on any closed $W$ must be an integer
since $\partial W = \varnothing$ implies $Z(\partial W) =1$.
Therefore, \eqref{eq:(D+1)map} is actually an element of $\Hom( \Omega^{\mathcal S}_{D+1}(X), \Z) $,
and, equivalently, it is in the kernel of the map $\Hom( \Omega^{\mathcal S}_{D+1}(X), \R)  \to \Hom( \Omega^{\mathcal S}_{D+1}(X), \R/\Z) $.
For appropriate (but not all)\footnote{
The following statement fails when the Chern-Weil construction does not give an isomorphism. It happens in some noncompact groups,
such as ${\rm SL}(2,\R)$. Thus we have assumed that the groups $G_d$ in this paper are compact.
However, there are also groups which are noncompact but the Chern-Weil isomorphism holds.
An example is ${\rm SL}(2,\Z)$ which has a trivial real cohomology $H^*(B{\rm SL}(2,\Z), \R)$.
This group can also have anomalies which are physically relevant~\cite{Seiberg:2018ntt,Hsieh:2019iba,Hsieh:2020jpj}. }
${\mathcal S}$, the fact that any element of $\Hom( \Omega^{\mathcal S}_{D+1}(X), \Z) $ can be realized
in this way by some $I_{D+1}$ will follow from the Chern-Weil theory and the Hurewicz theorem.
This gives the exactness at $\Hom( \Omega^{\mathcal S}_{D+1}(X), \R) $.
Also, if $I_{D+1} = d J_{D}$ for some $J_{D}$ which is constructed by geometric data,
we get $Z(\partial W) = \exp\left( 2\pi \sqrt{-1} \int_{\partial W} J_{D} \right)$.
This kind of contribution can be continuously deformed to zero by considering
a one parameter family of QFT's parametrized by $t \in [0,1]$ as 
\beq
Z_t(M)=\exp\left( 2\pi \sqrt{-1} \int_{M} tJ_{D} \right), \label{eq:oneparameter}
\eeq
so 
it does not contribute to the deformation classes of QFT's.
This corresponds to taking the equivalence classes as in \eqref{eq_equivalence_intro}.
The $\R/\Z$-valued functions $h$ which appeared in the definition of 
elements of $(\widehat{I\Omega_{\mathrm{dR}}^G})^n(X)$ correspond to partition functions
as $Z = \exp(2\pi \sqrt{-1}h)$, and $\mathrm{cw}_{g}(f^*\alpha)$ in \eqref{eq_er_intro} corresponds to $J_D$.

The authors are not aware of completely general proof of the expectation that the partition function can be expressed as \eqref{eq:PT2}. 
However, there is various evidence supporting this claim.
First, \eqref{eq:PT2} is exactly what is used in the construction of Wess-Zumino-Witten terms~\cite{Wess:1971yu}
with the target space $X$ by extending a manifold $M$ to $W$~\cite{Witten:1983tw}. 
In physics literature, Chern-Simons invariants are also described by \eqref{eq:PT2}. (See Example~\ref{ex_CCS} for more precise discussions.) 
Second, invertible field theories constructed from massive fermions in the large mass limit
satisfy \eqref{eq:PT2}. (See \cite{Witten:2019bou} for a systematic discussion).
Third, other nontrivial examples of invertible QFT's 
also satisfy \eqref{eq:PT2}, such as the one relevant for the anomalies of chiral $p$-form fields~\cite{Hsieh:2020jpj}. 
Finally, it is possible to give a physically reasonable derivation of a weaker version of the claim as follows.
The functional derivative of the log of the partition function $\log Z(M)$ in terms of a background field $\phi$ 
(i.e. geometric data such as Riemann metric, connections, etc.)
is given by a one-point function of some local operator $O$, 
\beq
\frac{\partial \log Z(M)}{\partial \phi(x)} = \langle O(x) \rangle. \qquad (x \in M)
\eeq
In theories whose low energy limits are invertible QFT's, there are no light degrees of freedom and all
Feynman propagators are short range. Thus we expect that 
the one-point function $\langle O(x) \rangle$ is given by local geometric data at the point $x \in M$.
Therefore, if two manifolds equipped with differential structure, $M$ and $M'$, are homotopy equivalent,
the ratio of their partition functions is given by an integral of some local quantity.~\footnote{
This argument itself also applies to the case in which the theory is not invertible but is topologically ordered.
Thus this argument does not give a complete proof of the claim.
}

So far, we have argued that the exact sequence satisfied by $(I\Omega^{\mathcal S})^{D+1}(X)$
is physically reasonable. However, the above physical arguments do not tell us anything 
about the group $(I\Omega^{\mathcal S})^{D+1}(X)$ itself beyond the exact sequence. The Anderson dual 
is defined in a very abstract way, and it is hard to find a direct physical interpretation of the Anderson dual.
One of our main results stated in Theorem~\ref{thm_intro_main} is a natural isomorphism of the cohomology theory $I\Omega^{G}$
to the theory $I\Omega_{\mathrm{dR}}^{G}$. 
Here the structure type $\mathcal S$ is taken to be a specific kind specified by $G$.
The cohomology theory $I\Omega_{\mathrm{dR}}^{G}$ is constructed in a way which closely follows the above physical discussions.
Therefore, our results give a very strong support of the conjecture that the deformation classes of invertible QFT's are given by $(I\Omega^{\mathcal S})^{D+1}(X)$.

Although $(I\Omega_{\mathrm{dR}}^{G})^n(X)$ is the group which is believed to classify the deformation classes of invertible field theories,
the group $(\widehat{I\Omega_{\mathrm{dR}}^G})^n(X)$ before taking the deformation classes is also physically very relevant. 
When we make background fields to dynamical fields, invertible field theories give topologically interesting terms in the action of the dynamical fields.
Examples of this kind include topological $\theta$-terms in gauge theories and Wess-Zumino-Witten terms in sigma models.
The elements of $(\widehat{I\Omega_{\mathrm{dR}}^G})^{D+1}(X)$ which are realized by $\alpha \in \Omega^{n-1}(X; \varprojlim_d(\mathrm{Sym}^{\bullet/2}\mathfrak{g}_d^*)^{G_d})$
in Theorem~\ref{thm_intro_main} are, in physics language, the terms in the Lagrangian which are manifestly gauge invariant and local in $D$-dimensions.
These terms are gauge invariant even if we evaluate them on a $D$-manifold $M$ with boundaries, so they do not contribute
to anomalies via anomaly inflow. But they are interesting terms in the path integral.

Let us comment on reflection positivity which is an important physical condition.
In the context of partition functions, reflection positivity is the following requirement. 
Consider a $D$-manifold $M$ with a boundary. We glue $M$ and its opposite $\overline M$ along their boundaries
to get a closed manifold $M \cup \overline M$ which is called a double. Reflection positivity is a requirement that
$Z(M \cup \overline M)$ is a nonnegative real number.
In the case of TQFT, reflection positivity is an important ingredient in the classification of \cite{Freed:2016rqq,Yonekura:2018ufj}.
Indeed, there are counterexamples to the classification
if we do not impose reflection positivity. See \cite{Freed:2016rqq,Hsieh:2020jpj} for these examples.
The partition function on a sphere $S^D$ becomes negative, $Z(S^D)=-1$, although $S^D$ can be constructed as the double of a hemisphere.
These examples have the property that $I_{D+1}$ discussed above contains $\frac12 E$,
where $E$ is the Euler density which gives the Euler number of the manifold when it is integrated.
The Euler density is excluded in this paper by imposing a stability condition which we will define later in this paper.
But the stabilization is not important for the types of $G$ considered in \cite{Freed:2016rqq}. We will make more comments on this point in Subsection \ref{sec_physical_app}.

\subsection{Compactification}\label{subsubsec_compactification}
We can get a QFT in $(D-d)$-dimensions by compactification of another QFT in $D$-dimensions on a $d$-dimensional manifold.
Let us consider a $d$-manifold $L$ on which we will compactify the theory. 
We put a geometric structure on $L$. Actually, it is better to consider a family of geometric structures on $L$.
For example, we can consider various Riemannian metrics on $L$ and connections of a $G$-bundle on $L$,
and the parameter space may be called the moduli space of geometric structures. 
Usually, it is enough for many purposes to consider a finite dimensional approximation
to the full moduli space of geometric structures. For instance, if $L$ is a two dimensional torus $T^2$,
we may only consider its complex structure modulus and the total area as a finite dimensional approximation to
the full moduli space of Riemannian metrics on $T^2$. If there is an internal symmetry group, we may also consider holonomies of connections around cycles on $T^2$.
From the point of view of the lower dimensional theory after compactification,
the space of these parameters are the target space of a new sigma model in lower dimensions. 
If the target space of the $D$-dimensional theory is $X$, the new target space
of the $(D-d)$-dimensional theory may be $X \times \mathcal{Y}$ where $\mathcal{Y}$ is such a moduli space.

Motivated by the above considerations, let us consider a fiber bundle $N \to Y$
in which the fiber is a $d$-manifold $L$ and $Y$ is a smooth base manifold.
Here, $Y$ is what we intuitively want to consider as a moduli space of geometric structures $\mathcal{Y}$, but we remark that $Y$ is 
just an arbitrary smooth manifold. 
We assume that each fiber is equipped with a geometric structure,
and $Y$ may be intuitively regarded as parametrizing a family of geometric structures. 
(We will discuss more general situations in Sec.~\ref{sec_module} than fiber bundles,
where precise definitions are also given.) 
If we are given an element of $(\widehat{I\Omega_{\mathrm{dR}}^G})^{D+1}(X)$,
then the compactification on $L$ with a family of geometric structures $Y$
may be expected to give an element of $(\widehat{I\Omega_{\mathrm{dR}}^G})^{(D-d)+1}(X \times Y)$.
This is an interpretation of \eqref{eq_module_first_intro}. By making the target spaces more explicit,
it is the transformation we construct in Section \ref{sec_module}, \footnote{
This map is in the form of the external product, which is recovered by the internal product we use in Section \ref{sec_module} by pulling back to $X \times Y$ and multiply there. 
Conversely, we can recover the internal product from the external one by pulling back by the diagonal map. 
} 
\begin{align}\label{intro_phys_comp1}
(\widehat{I\Omega^{G_3}_{\mathrm{dR}}})^{D+1}(X)\otimes (\widehat{\Omega^{G_2}})^{-d}(Y) 
\xrightarrow{\cdot} (\widehat{I\Omega^{G_1}_{\mathrm{dR}}})^{(D-d)+1}(X \times Y), 
\end{align} 
Here, the group $G_2$ is the structure type of fiber manifolds $L$, and $(\widehat{\Omega^{G_2}})^{-d}(Y) $ is roughly the abelian group of
fiber bundles $N \to Y$ equipped with geometric structures of fiber manifolds. 
The abelian group structure is given by fiber-wise disjoint union with the fixed base $Y$.
The group $G_1$ is the structure type of $(D-d)$-manifolds $M$,
and we assume that products of manifolds (or more generally fiber bundles $L \to M' \to M$ with base $M$ and fiber $L$) 
with the $G_1$ and $G_2$ structures can be equipped with the $G_3$ structure.
For example, we can obtain a $\mathrm{Pin}^+$ $D$-manifold from a product of a $\mathrm{Pin}^+$ $d$-manifold and a 
$\mathrm{Spin}$ $(D-d)$-manifold. This means that if a theory defined on $\mathrm{Pin}^+$-manifolds
is compactified on a $\mathrm{Pin}^+$-manifold, we get a theory defined on $\mathrm{Spin}$-manifolds. 
See \cite[Section 2.2.7]{tachikawa2021topological} for another example. 

We can also consider a slightly different type of compactification, which corresponds to the construction in Section \ref{sec_push}. 
Let $M'$ be the $D$-manifold which is the fiber bundle with the base $(D-d)$-manifold $M$ and the fiber $d$-manifold $L$.
The compactification obtained by the above procedure is restricted
to the case that the map $M' \to X$ factors as $M' \to M \to X$, where $M' \to M$ is the bundle projection. 
For some applications, it is also useful to have a variant of the above construction.
Consider a fiber bundle $p : X' \to X$ where the fiber is a $d$-manifold $L$ with $G_2$-structure. 
By a similar construction as above, we get a map 
\begin{align}\label{intro_phys_comp2}
 (\widehat{I\Omega^{G_3}_{\mathrm{dR}}})^{D+1}(X') \to  (\widehat{I\Omega^{G_1}_{\mathrm{dR}}})^{(D-d)+1}(X).
 \end{align}
Here, the map $M \to X$ is uplifted to $M' \to X'$. 
A simple example is given by compactification of brane actions in string theory.
We interpret $X'$ as a target spacetime manifold of string theory, $M'$ as the worldvolume of a brane,
and we have a map $f' : M' \to X'$ which describes how the brane is placed in $X'$. 
Now, we compactify the target spacetime from $X'$ to $X$, and ``wrap'' the brane to the fiber $L$ of $X' \to X$.
This situation describes the above map \eqref{intro_phys_comp2}. The left and right hand sides describe the (imaginary part of) the brane actions before and after the compactification, respectively.

Both of \eqref{intro_phys_comp1} and \eqref{intro_phys_comp2} are useful for describing compactification
of different types. There can be more general situations, but we do not discuss them in this paper.

\section*{Acknowledgment}
The authors are grateful to Yosuke Morita for collaboration on many parts of the paper, 
and Kaoru Ono for helpful discussion and comments. 
The work of MY is supported by Grant-in-Aid for JSPS KAKENHI Grant Number 20K14307 and JST CREST program JPMJCR18T6.
The work of KY is supported by JST FOREST Program (Grant Number JPMJFR2030, Japan), and JSPS KAKENHI Grant Number 17K14265.

\bibliographystyle{ytamsalpha}
\bibliography{QFT}

\end{document}